\documentclass[11pt]{article}

\usepackage{amsfonts,amsthm,amssymb,amsmath,mathrsfs}
\usepackage{amscd}
\usepackage{cancel}
\usepackage{color}
\usepackage{graphicx}
\usepackage{enumitem}
\usepackage{fancyhdr}
\usepackage{stmaryrd}
\usepackage{ytableau}
\usepackage[all]{xy}

\theoremstyle{plain}
\newtheorem{thm}{Theorem}[section]

\newtheorem{prop}[thm]{Proposition}
\newtheorem{cor}[thm]{Corollary}
\theoremstyle{definition}
	
\newtheorem{remark}[thm]{Remark}
\newtheorem*{rk}{Remark}
\theoremstyle{example}
\newtheorem{example}{Example}[section]

\theoremstyle{remark}

\numberwithin{equation}{section}

   \newtheoremstyle{thmn}
        {\topsep}{\topsep}              
        {\itshape}                      
        {}                              
        {\bfseries}                     
        {.}                             
        { }                             
        {\thmname{#1}\thmnote{ \bfseries #3}}
    \theoremstyle{thmn}
    \newtheorem{dupThm}{Theorem}
    \newtheorem{dupProp}{Proposition}
    \newtheorem{dupCor}{Corollary}

\providecommand{\keywords}[1]{\textbf{\textit{Key words---}} #1}

\def\CC{\mathbb{C}}

\def\QQ{\mathbb{Q}}
\def\RR{\mathbb{R}}

\def\ZZ{\mathbb{Z}}

\def\Card{\mathrm{Card}}
\def\diag{\mathrm{diag}}
\def\dim{\mathrm{dim}}

\def\ev{\mathrm{ev}}

\def\Ind{\mathrm{Ind}}

\def\Tr{\mathrm{Tr}}

\usepackage{etex} 
\usepackage{pictexwd}
\usepackage{tikz}
	\usepgflibrary[patterns] 
	\usetikzlibrary{patterns} 
	\usepgflibrary{shapes.geometric}

\newcommand{\TikZ}[1]{
\begin{matrix}\begin{tikzpicture}#1\end{tikzpicture}\end{matrix}
}

\newcounter{r}
\newcounter{s}

\newcommand\Part[1]{
        \setcounter{r}{1}
	 \foreach \x in {#1}{
 	{\ifnum\value{r}=1
		\draw (0,\value{r}-1)--(\x,\value{r}-1); 
		\fi}
	\draw (0,\value{r}) to (\x,\value{r});
   	\foreach \y in {0, ..., \x} {\draw (\y,\value{r})--(\y,\value{r}-1);}
	\addtocounter{r}{1}
 }}
 
\newcommand\Tableau[1]{
        \foreach \x [count = \c from 1] in {#1} {
		\foreach \y [count = \d from 1] in \x{
			\node at (\d-.5,\c-.5) {\scriptsize$\y$}; 
			\draw (\d,\c) to (\d,\c-1);
			{\ifnum\d=1
				\draw (0,\c) to (0,\c-1);
				\fi}
			\setcounter{r}{\d}
		}
		{\ifnum\c=1
			\draw (0,0)--(\value{r},0);
			\fi}
		\draw(0,\c) to (\value{r},\c);
		\setcounter{s}{\c}}}
		
\newcommand\sTableau[1]{
        \foreach \x [count = \c from 1] in {#1} {
		\foreach \y [count = \d from 1] in \x{
			\node at (\d-.5,\c-.5) {\tiny$\y$}; 
			\draw (\d,\c) to (\d,\c-1);
			{\ifnum\d=1
				\draw (0,\c) to (0,\c-1);
				\fi}
			\setcounter{r}{\d}
		}
		{\ifnum\c=1
			\draw (0,0)--(\value{r},0);
			\fi}
		\draw(0,\c) to (\value{r},\c);
		\setcounter{s}{\c}}}

\tikzstyle{V}=[draw, fill =black, circle, inner sep=0pt, minimum size=1.5pt]
\tikzstyle{wV}=[draw, fill =white, circle, inner sep=0pt, minimum size=4.5pt]
\tikzstyle{bV}=[draw, fill =black, circle, inner sep=0pt, minimum size=4.5pt]
\tikzstyle{over}=[draw=white,double=black,line width=2pt, double distance=.5pt]

\def\Over[#1,#2][#3,#4]{ 
	\draw[style=over]   (#2,#1) .. controls ++(#4*.5-#2*.5,0) and ++(-#4*.5+#2*.5,0) .. (#4,#3);}
\def\Under[#1,#2][#3,#4]{ 
	\draw  (#2,#1) .. controls ++(#4*.5-#2*.5,0) and ++(-#4*.5+#2*.5,0) .. (#4,#3);}
\def\Cross[#1,#2][#3,#4]{
	\Under[#3,#2][#1,#4]\Over[#1,#2][#3,#4]}

\def\Tops[#1][#2][#3]{
	\foreach\x in {#1}{
		\draw (#2,\x+.15) -- (#2+.1, \x+.15) (#2, \x-.15) -- (#2+.1, \x-.15) ;
		\draw (#2+.1,\x) arc (0:360:.75mm and 1.5mm);}
	\foreach \x in {1,...,#3} {\draw (#2,\x)  to (#2+.05,\x); \node[V] at (#2+.05,\x){};}
	}
\def\Bottoms[#1][#2][#3]{
	\foreach\x in {#1}{
		\draw (#2, \x+.15) -- (#2-.1, \x+.15) (#2, \x-.15) -- (#2-.1, \x-.15) ;
		\draw (#2-.1, \x+.15) arc (90:270:.75mm and 1.5mm);}
	\foreach \x in {1,...,#3} {\draw (#2, \x)  to (#2-.05, \x); \node[V] at (#2-.05, \x){};}
	}
\def\Caps[#1][#2,#3][#4]{
	\Tops[#1][#3][#4]
	\Bottoms[#1][#2][#4]
	}
\def\Pole[#1][#2,#3]{
	\shade[left color=white,right color=white] (#2,#1+.15) rectangle (#3,#1-.15);
	\draw[over] (#2,#1+.15) to (#3,#1+.15) (#2,#1-.15) to (#3,#1-.15) ;}
\def\Label[#1,#2][#3][#4]{
	\node[right] at (#2+.1,#3) {#4};
	\node[left] at (#1-.1,#3) {#4};		}
\def\Nodes[#1][#2]{
	 \foreach \x in {1,...,#2} {\node[V] at (#1,\x){};	}
	}
\def\PoleCaps[#1][#2,#3]{
	\foreach\x in {#1}{
		\draw (#2,\x+.15) -- (#2-.1,\x+.15) (#2,\x-.15) -- (#2-.1,\x-.15) ;
		\draw (#2-.1,\x+.15) arc (0:-180:1.5mm and .75mm);}
	\foreach\x in {#1}{
		\draw (#3,\x+.15) -- (#3+.1,\x+.15) (#3,\x-.15) -- (#3+.1,\x-.15) ;
		\draw (#3+.1,\x+.15) arc (0:360:1.5mm and .75mm);}
	}
\def\PoleTwist[#1,#2]{
	\foreach \x/\y in {-1/1L, -.7/1R, 0/2L, .3/2R}{\coordinate(T\y) at (#2,\x); \coordinate(B\y) at (#1,\x);}
	\draw[thin] (B1R) .. controls ++(#2*.5-#1*.5-.1,0) and ++(-#2*.5+#1*.5-.1,0) ..  (T2R)
			(B1L)   .. controls ++(#2*.5-#1*.5+.1,0) and ++(-#2*.5+#1*.5+.1,0) ..    (T2L) ;
	\draw[line width=2pt, white]
			(#1,.15)  .. controls +(#2*.5-#1*.5,0) and +(-#2*.5+#1*.5,0) ..   (#2,-.85) ;
	\draw[thin,over] 
		(B2R) .. controls ++(#2*.5-#1*.5+.1,0) and ++(-#2*.5+#1*.5+.1,0) ..  (T1R) 
			(B2L)  .. controls +(#2*.5-#1*.5-.1,0) and +(-#2*.5+#1*.5-.1,0) ..   (T1L) ;
			}

\def\SymPolesCaps[#1,#2][#3]{
	\draw (#1,.3) -- (#1-.1,.3) (#1,.15) -- (#1-.1, .15) ;
	\draw (#1-.1, .3) arc (0:-180:2pt and 1.5pt);
	\draw (#1,#3+.7) -- (#1-.1,#3+.7) (#1,#3+.85) -- (#1-.1,#3+.85) ;
	\draw (#1-.1,#3+.85)  arc (0:-180:2pt and 1.5pt);
	\draw (#2,.3) -- (#2+.1, .3) (#2, .15) -- (#2+.1, .15) ;
	\draw (#2+.1, .3) arc (0:360:2pt and 1.5pt);
	\draw (#2, #3+.7) -- (#2+.1, #3+.7) (#2, #3+.85) -- (#2+.1, #3+.85) ;
	\draw (#2+.1, #3+.85) arc (0:360:2pt and 1.5pt);}

\newcommand{\posleq}[1]{
	\hspace{0.1cm}
	\begin{tikzpicture}
	\draw (-0.8ex, -0.5ex) -- (0.8ex, -0.5ex);
	\draw (-0.8ex, 0.4ex) -- (0.7ex, -0.2ex);
	\draw (-0.8ex, 0.4ex) -- (0.7ex, 1ex);
	\draw (0.4ex,0.4ex) --(1.1ex, 0.4ex);
	\draw (0.75ex,0.75ex) --(0.75ex, 0.05ex);
	\end{tikzpicture}
	\hspace{0.1cm}
	}
\newcommand{\negleq}[1]{
	\hspace{0.1cm}
	\begin{tikzpicture}
	\draw (-0.8ex, -0.5ex) -- (0.8ex, -0.5ex);
	\draw (-0.8ex, 0.4ex) -- (0.7ex, -0.2ex);
	\draw (-0.8ex, 0.4ex) -- (0.7ex, 1ex);
	\draw (0.4ex,0.4ex) --(1.1ex, 0.4ex);
	\end{tikzpicture}
	\hspace{0.1cm}
	}
	
\newcommand{\zeroleq}[1]{
	\hspace{0.1cm}
	\begin{tikzpicture}
	\draw (-0.8ex, -0.5ex) -- (0.8ex, -0.5ex);
	\draw (-0.8ex, 0.4ex) -- (0.7ex, -0.2ex);
	\draw (-0.8ex, 0.4ex) -- (0.7ex, 1ex);
	\draw  (0.75ex,0.4ex) ellipse (0.2ex and 0.35ex);
	\end{tikzpicture}
	\hspace{0.1cm}
	}
	
\newcommand{\posgeq}[1]{
	\hspace{0.1cm}
	\begin{tikzpicture}
	\draw (-0.8ex, -0.5ex) -- (0.8ex, -0.5ex);
	\draw (0.8ex, 0.4ex) -- (-0.7ex, -0.2ex);
	\draw (0.8ex, 0.4ex) -- (-0.7ex, 1ex);
	\draw (-0.4ex,0.4ex) --(-1.1ex, 0.4ex);
	\draw (-0.75ex,0.75ex) --(-0.75ex, 0.05ex);
	\end{tikzpicture}
	\hspace{0.1cm}
	}
\newcommand{\neggeq}[1]{
	\hspace{0.1cm}
	\begin{tikzpicture}
	\draw (-0.8ex, -0.5ex) -- (0.8ex, -0.5ex);
	\draw (0.8ex, 0.4ex) -- (-0.7ex, -0.2ex);
	\draw (0.8ex, 0.4ex) -- (-0.7ex, 1ex);
	\draw (-0.4ex,0.4ex) --(-1.1ex, 0.4ex);
	\end{tikzpicture}
	\hspace{0.1cm}
	}
	
\newcommand{\zerogeq}[1]{
	\hspace{0.1cm}
	\begin{tikzpicture}
	\draw (-0.8ex, -0.5ex) -- (0.8ex, -0.5ex);
	\draw (0.8ex, 0.4ex) -- (-0.7ex, -0.2ex);
	\draw (0.8ex, 0.4ex) -- (-0.7ex, 1ex);
	\draw  (-0.75ex,0.4ex) ellipse (0.2ex and 0.35ex);
	\end{tikzpicture}
	\hspace{0.1cm}
	}

\newcommand{\posl}[1]{
	\hspace{0.1cm}
	\begin{tikzpicture}
	\draw (-0.8ex, 0.4ex) -- (0.7ex, -0.2ex);
	\draw (-0.8ex, 0.4ex) -- (0.7ex, 1ex);
	\draw (0.4ex,0.4ex) --(1.1ex, 0.4ex);
	\draw (0.75ex,0.75ex) --(0.75ex, 0.05ex);
	\end{tikzpicture}
	\hspace{0.1cm}
	}
\newcommand{\negl}[1]{
	\hspace{0.1cm}
	\begin{tikzpicture}
	\draw (-0.8ex, 0.4ex) -- (0.7ex, -0.2ex);
	\draw (-0.8ex, 0.4ex) -- (0.7ex, 1ex);
	\draw (0.4ex,0.4ex) --(1.1ex, 0.4ex);
	\end{tikzpicture}
	\hspace{0.1cm}
	}
	
\newcommand{\zerol}[1]{
	\hspace{0.1cm}
	\begin{tikzpicture}
	\draw (-0.8ex, 0.4ex) -- (0.7ex, -0.2ex);
	\draw (-0.8ex, 0.4ex) -- (0.7ex, 1ex);
	\draw  (0.75ex,0.4ex) ellipse (0.2ex and 0.35ex);
	\end{tikzpicture}
	\hspace{0.1cm}
	}
	
\newcommand{\posg}[1]{
	\hspace{0.1cm}
	\begin{tikzpicture}
	\draw (0.8ex, 0.4ex) -- (-0.7ex, 1ex);
	\draw (0.8ex, 0.4ex) -- (-0.7ex, -0.2ex);
	\draw (-0.4ex,0.4ex) --(-1.1ex, 0.4ex);
	\draw (-0.75ex,0.75ex) --(-0.75ex, 0.05ex);
	\end{tikzpicture}
	\hspace{0.1cm}
	}
\newcommand{\negg}[1]{
	\hspace{0.1cm}
	\begin{tikzpicture}
	\draw (0.8ex, 0.4ex) -- (-0.7ex, -0.2ex);
	\draw (0.8ex, 0.4ex) -- (-0.7ex, 1ex);
	\draw (-0.4ex,0.4ex) --(-1.1ex, 0.4ex);
	\end{tikzpicture}
	\hspace{0.1cm}
	}
	
\newcommand{\zerog}[1]{
	\hspace{0.1cm}
	\begin{tikzpicture}
	\draw (0.8ex, 0.4ex) -- (-0.7ex, -0.2ex);
	\draw (0.8ex, 0.4ex) -- (-0.7ex, 1ex);
	\draw  (-0.75ex,0.4ex) ellipse (0.2ex and 0.35ex);
	\end{tikzpicture}
	\hspace{0.1cm}
	}

\makeatletter
\renewcommand{\@makefnmark}{\mbox{\textsuperscript{}}}
\makeatother

\title{$c$-functions and Macdonald polynomials}
\author{
Laura Colmenarejo\quad\ \ email:\ lcolmen@ncsu.edu \\
Arun Ram\quad\quad\ \,\ email:\ aram@unimelb.edu.au \\
\\
}
\date{}

\usetikzlibrary{arrows.meta}

\begin{document}

\maketitle

\vspace{-3em}
\begin{center}
{\sl In memory of Georgia Benkart}
\end{center}


\begin{abstract}
\noindent
This is a paper about $c$-functions and Macdonald polynomials.  There are 
$c$-function formulas for $E$-expansions of $P_\lambda$ and $A_{\lambda+\rho}$,
principal specializations of $P_\lambda$ and $E_\mu$, for Macdonald's constant term formulas,
and for the norms of Macdonald polynomials.  Most of these follow from the creation
formulas for Macdonald polynomials, providing alternative proofs to several results from
~\cite{Mac03}.  In addition, we prove the Boson-Fermion correspondence in the Macdonald polynomial setting
and the Weyl character formula for Macdonald polynomials.
\end{abstract}


\setcounter{section}{0}

\tableofcontents

\section*{Introduction}
\addcontentsline{toc}{section}{Introduction}

There is a wonderful article by S. Helgason entitled
\emph{Harish-Chandra's $c$-function. A Mathematical Jewel}~\cite{Hel94}.
Helgason describes how the $c$-function is at the core of 
spherical functions, eigenfunctions of invariant differential operators,
hypergeometric functions, Gamma functions, Plancherel measures, 
Fourier transforms, Radon transforms, orbital integrals, and symmetric spaces.
In Macdonald's monograph 
on spherical functions on $p$-adic groups he pointed
to an analogue of the Harish-Chandra's $c$-function which plays a similar role
for spherical functions on $p$-adic groups and provides for $q$-analogues of the
Gamma functions and the other topics in Helgason's article (see~\cite[Ch. IV (4.1)]{Mac71}).  
In Macdonald's 2003 monograph \textsl{Double Affine Hecke algebras and Orthogonal polynomials}~\cite{Mac03} these
analogues of Harish-Chandra's $c$-function appear everywhere,  without explicit mention.  
In this paper we wish to make these appearances of the $c$-function more visible and explain how
they provide an understanding of combinatorial formulas in Macdonald polynomial theory.

We begin this paper with some philosophical ruminations. In Section~\ref{Sec:Why},  we present four broad perspectives indicating why Macdonald polynomials are fascinating objects for continued research:
\begin{enumerate}
\itemsep=-0.2em
\item[(a)] the Macdonald polynomials are eigenfunctions (``wave functions'') for a class of operators
which play a role analogous to the role played by the Laplacian and other Hamiltonians 
in classical harmonic analysis and mathematical physics\,;
\item[(b)] the principal specializations of Macdonald polynomials point to `elliptic generalizations'
of Weyl's dimension formula for irreducible representations of compact Lie groups\,;
\item[(c)] there is a $(q,t)$-generalization of a Boson-Fermion type correspondence which hints
at an ``elliptic'' generalization of ``geometric Satake'';
\item[(d)] the recursive construction of Macdonald polynomials by intertwining operators
has relations to the construction of
Schubert polynomials and Grothendieck polynomials by divided-difference operators and Demazure operators.
\end{enumerate}
Section 2 introduces the main objects of study:
\begin{enumerate}[itemsep=-0.2em]
\item[(a)] the electronic Macdonald polynomials (known as \emph{nonsymmetric} in the literature);
\item[(b)] the bosonic Macdonald polynomials $P_\lambda$ (known as \emph{symmetric} in the literature);
\item[(c)] the fermionic Macdonald polynomials $A_{\lambda+\rho}$.
\end{enumerate}
\begin{rk}
The first author gave a series of lectures on Macdonald polynomials at the University of Melbourne. After explaining
the $(q,t)$-version of the Boson-Fermion correspondence and how it relates to the polynomials
$P_\lambda$ and $A_{\lambda+\rho}$ it started to feel natural to call the $P_\lambda$ bosonic Macdonald
polynomials and the $A_{\lambda+\rho}$ fermionic Macdonald polynomials.  Soon we began exploring the 
pleasing analogies between the Cherednik-Dunkl operators and the Hamiltonian for the quantum harmonic
oscillator and there was a suggestion to call the $E_\mu$ electronic Macdonald polynomials.
It was fun and helpful for keeping these three families straight in one's head, and we've decided to adopt it here.
In previous literature, the $E_\mu$ have been called the ``nonsymmetric Macdonald polynomials'' and the
$P_\lambda$ have been called the ``symmetric Macdonald polynomials''.
\qed
\end{rk}

In Section 3 we develop the powerful operator calculus for handling Macdonald polynomials.  There
are four primary families of operators that are employed:
\begin{enumerate}[itemsep=-0.2em]
\item[(a)] the Hecke algebra operators $T_1, \ldots, T_{n-1}$, and the promotion operator $T_\pi$;
\item[(b)] the Cherednik-Dunkl operators $Y_1, \ldots, Y_n$;
\item[(c)] the intertwiner operators $\tau^\vee_\pi$ and $\tau^\vee_1, \ldots, \tau^\vee_{n-1}$;
\item[(d)] the symmetrizers $\mathbf{1}_0$ and $\varepsilon_0$.
\end{enumerate}
This material is exposited in the books of Macdonald~\cite{Mac03} and Cherednik~\cite{Che05}.
We have tried to make an efficient and accessible treatment of these results, in the
type $GL_n$ case, in a continuing effort to make these amazing and powerful methods more and more
broadly available.  In particular, we have found the symmetrizer expressions in Propositions~\ref{propslicksymmA}
and~\ref{propsymwparabA}  to be
of great utility and, although they have their roots in the seminal work of Harish-Chandra and Macdonald 
and many others (see, in particular,~\cite[(5.5.14)]{Mac03}),
we hope our treatment might help others find further uses for these identities.

In Section 4, we study the action of various operators on polynomials and prove important results.  In particular, we look at:
\begin{enumerate}
\itemsep=-0.2em
\item[(a)] the creation formulas for Macdonald polynomials via intertwiners and symmetrizers;
\item[(b)] the $(q,t)$ Boson-Fermion correspondence and its relation to the symmetrizers;
\item[(c)] the Poincar\'{e} polynomial and its relation to the symmetrizers;
\item[(d)] the $E$-expansions of $P_\lambda$ and $A_{\lambda+\rho}$.
\end{enumerate}
The Boson-Fermion correspondence identifies two different avatars of the ring of symmetric polynomials by
relating the $P_\lambda$ and $A_{\lambda+\rho}$. Moreover, it provides a point of view that connects
mathematical physics, geometric representation theory, and the Langlands program.  The
$E$-expansions are wonderfully explicit combinatorial expressions that generalize the 
expressions for monomial symmetric polynomials as sums over permutations.  Understanding
the coefficients in these expansions in terms of $c$-functions makes more explicit the relation between these expansions and formulas
like the Gindikin-Karpelevi\v c formula in representation theory (see, for example,~\cite{Kn03}
and~\cite[(1)]{BN10}).

Section 5 is devoted to principal specializations of Macdonald polynomials.  We give two kinds of formulas:
\begin{enumerate}[itemsep=-0.2em]
\item[(a)] $c$-function formulas for principal specializations of $E_\mu$, $P_\lambda$ and $A_{\lambda+\rho}$;
\item[(b)] hook formulas for principal specializations of $P_\lambda$ and $E_\mu$.
\end{enumerate}
The $c$-function formulas are reformulations of~\cite[(5.3.9) and (5.2.14)]{Mac03} which put the focus on type $GL_n$
and the corresponding $c$-functions.  The hook formula (Theorem~\ref{hookforP}) 
for the principal specialization of $P_\lambda$ 
is exactly that of~\cite[Ch. VI ($6.11'$)]{Mac}.  The proof we give is different -- it uses the intertwiner operators
to derive the $c$-functions and then the combinatorial argument of~\cite{AGY22}.

Section 6 introduces the inner product $(,)_{q,t}$
with respect to which the Macdonald polynomials are orthogonal polynomials.
We prove that the inner product is sesquilinear, nondegenerate, and normalized Hermitian, that the characterization of
the Macdonald polynomials in terms of the inner product, and two amazing formulas:
\begin{enumerate}
\item[(a)] the ``going up a level'' formula relating $(f,g)_{q,qt}$ to $(A_\rho f, A_\rho g)_{q,t}$;
\item[(b)] the Weyl character formula for Macdonald polynomials.
\end{enumerate}
This section follows closely the exposition in~\cite[Ch.\ 5]{Mac03} except for the changes of notation to focus on 
type $GL_n$.  In particular, the ``going up a level'' formula (Proposition~\ref{ktokp1comp}) and the
Weyl character formula for Macdonald polynomials (Theorem~\ref{WCF}) are~\cite[(5.8.6)]{Mac03} and
~\cite[(5.8.12)]{Mac03}, respectively.

Section 7 gives an exposition (with proofs) of 
\begin{enumerate}[itemsep=-0.2em]
\item[(a)] the $c$-function formulas for $(E_\mu, E_\mu)_{q,t}$, $(P_\lambda,P_\lambda)_{q,t}$ and 
$(A_{\lambda+\rho},A_{\lambda+\rho})_{q,t}$; and
\item[(b)] the formulas for ``Macdonald's constant term''.
\end{enumerate}
Our exposition is for type $GL_n$,
although the proof follows the same ideas and pattern of the general type proof exposited in~\cite[\S 5.8]{Mac03}
(based on the amazing tools developed by Heckman, Opdam, Cherednik, and Macdonald).  We have made a special effort to
streamline the proof and make it accessible.

\bigskip\noindent
\textbf{Acknowledgements.}
A.\ Ram thanks the organizers of the MATRIX/RIMS conference
\textsl{Integrability, combinatorics and representation theory}, 
the Stanford University \textsl{Solvable Lattice Models seminar},
and the Indian Institute of Science, Bangalore \textsl{Workshop on Macdonald polynomials}
for invitations to give lectures on Macdonald polynomials which provided invaluable motivation and
feedback.  We are very grateful to the hardy students and participants of the series of lectures
on Macdonald polynomials given at the University of Melbourne.  Their questions, probing, and avid interest
had a huge impact on this paper.

\newpage

\section{Why these things are so incredibly interesting}\label{Sec:Why}
As mentioned in the introduction, this section aims to give a broad perspective of why Macdonald polynomials are interesting. Therefore, we warn the reader that this section is not too \emph{formal} and that the formal details are presented in the rest of the article. 

\subsection{Eigenvalues and eigenvectors}\label{sec:eig}

Let $n\in \ZZ_{>0}$ and let $q,t^{\frac12}\in \CC^\times:=\CC\setminus\{0\}$.
The symmetric group $S_n$ acts on $\CC[X]:=\CC[x^{\pm1}_1, \ldots, x^{\pm1}_n]$ by permuting $x_1, \ldots, x_n$
so that
$$(s_if)(x_1, \ldots, x_n) = f(x_1, \ldots, x_{i-1}, x_{i+1}, x_i, x_{i+2}, \ldots, x_n),
$$
where $s_i$ denotes the transposition in $S_n$ that switches $i$ and $i+1$.
Define operators $T_1, \ldots, T_{n-1}$ and $T_\pi$ on $\CC[X]$ by
$$T_if = -t^{-\frac12}f+(1+s_i)\frac{t^{-\frac12}-t^{\frac12}x_i^{-1}x_{i+1}}{1-x_i^{-1}x_{i+1}}f
\quad\hbox{and}\quad
(T_\pi f)(x_1, \ldots, x_n) = f(q^{-1}x_n, x_1, \ldots, x_{n-1}).
$$
The \emph{Cherednik-Dunkl operators} are 
$$Y_1 = T_\pi T_{n-1}\cdots T_1, \quad Y_2 = T_1^{-1}Y_1T_1^{-1}, 
\quad Y_3 = T_2^{-1}Y_2T_2^{-1},\quad \ldots,\quad
Y_n = T_{n-1}^{-1}Y_{n-1}T_{n-1}^{-1}.
$$

If $\mu = (\mu_1, \ldots, \mu_n) \in \ZZ^n$ then
the minimal length (with respect to the Bruhat order) permutation $v_\mu$ such that $v_\mu \mu$ is weakly increasing
is given by
$$v_\mu(r) = 1+\#\{r'\in \{1, \ldots, r-1\}\ |\ \mu_{r'}\le \mu_r\}
+\#\{r'\in \{r+1, \ldots, n\}\ |\ \mu_{r'}<\mu_r\}, \text{ for } r\in \{1, \ldots, n\}.$$

\begin{thm} \label{Eeigenvalue} 
There is a unique basis
$\{ E_\mu\ |\ \mu\in \ZZ^n\}$ of $\CC[X]$ such that
$$Y_i E_\mu = q^{-\mu_i}t^{-(v_\mu(i)-1)+\frac12(n-1)} E_\mu,
\qquad\hbox{for $i\in \{1, \ldots, n\}$,}
$$
and the coefficient of $x^\mu = x_1^{\mu_1}\cdots x_n^{\mu_n}$ in $E_\mu$ is $1$.
\end{thm}
\begin{rk}
    In order to make the notation lighter and easier to read, we do not include the variables $x_1,\dots, x_n$ and the parameters $q, t$ in the polynomials. 
\end{rk}
This is an incredible statement!  It says that the operators $Y_1, \ldots, Y_n$ all commute,
and that their simultaneous eigenvectors form an orthogonal basis with respect to an
appropriate inner product and that the eigenvalues are all explicitly determined.  
This special basis of simultaneous eigenvectors, the electronic Macdonald polynomials $E_\mu$,
is the primary object of study in this paper.  The inner product $(,)_{q,t}$ with respect to which they
are orthogonal will be studied in Sections~\ref{section:WCF} and~\ref{section:Orth}.

\subsection{Elliptic, quantum, and ordinary dimension formulas}\label{ssec:elldim}
Let $(\ZZ^n)^+ = \{ \lambda = (\lambda_1, \ldots, \lambda_n)\in \ZZ^n\ |\ \lambda_1\ge \cdots \ge \lambda_n\}$, and consider $\lambda\in (\ZZ^n)^+$. 
The \emph{bosonic Macdonald polynomial}
$P_\lambda$ is a symmetric version of the electronic Macdonald polynomial,
\begin{equation*}
P_\lambda = P_\lambda(x;q,t) = \frac{1}{W_\lambda(t)}
\sum_{w\in S_n} w\Big( E_\lambda \prod_{i<j} \frac{x_i-tx_j}{x_i-x_j}\Big),
\end{equation*}
where $W_\lambda(t)$ is the appropriate constant which makes the coefficient
of $x^\lambda$ equal to $1$ in $P_\lambda(q,t)$.  

Theorem~\ref{hookforP} and Corollary~\ref{princspecB}
say that
\begin{align}
P_{\lambda}(1,t,t^2, \ldots, t^{n-1};q,t)
&= t^{n(\lambda)} \prod_{1\le i<j\le n} \prod_{\ell=0}^{\lambda_i-\lambda_j-1}
\frac{1 - q^\ell t^{j-i +1} }{1- q^\ell t^{j-i}}
\nonumber  \\
&= t^{n(\lambda)} \prod_{b\in \lambda} 
\frac{1-q^{\mathrm{coarm}_\lambda(b)}t^{n-\mathrm{coleg}_\lambda(b)}}
{1-q^{\mathrm{arm}_\lambda(b)}t^{\mathrm{leg}_\lambda(b)+1}}
\label{elldim}
\end{align}
where 
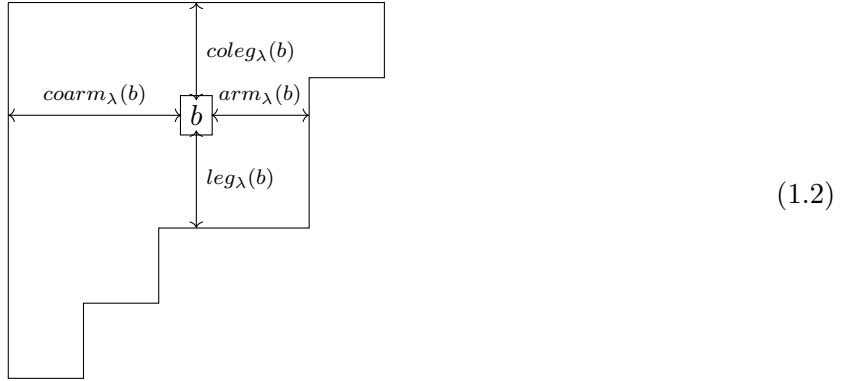
\begin{equation}
\begin{matrix}
\begin{tikzpicture}[yscale=-1]
\draw (0,0) -- (5,0) -- (5,1) -- (4,1) -- (4,3)  -- (2,3) -- (2,4) -- (1,4) -- (1,5) -- (0,5) -- (0,0);
\draw[<->] (2.5,0) -- node[anchor=west]{$\scriptstyle{coleg_\lambda(b)}$} (2.5,1.3) ;
\draw[<->] (2.5,1.7) -- node[anchor=west]{$\scriptstyle{leg_\lambda(b)}$} (2.5,3) ;
\draw[<->] (0,1.5) -- node[anchor=south]{$\scriptstyle{coarm_\lambda(b)}$} (2.3,1.5) ;
\draw[<->] (2.7,1.5) -- node[anchor=south]{$\scriptstyle{arm_\lambda(b)}$} (4,1.5) ;
\draw (2.5,1.5)  node[shape=rectangle,draw]{$b$};
\end{tikzpicture}
\end{matrix}
\label{ellhook}
\end{equation}

The Schur polynomial $s_\lambda$ is the specialization of $P_\lambda$ at $q=t$,
$s_\lambda(x_1, \ldots, x_n) = P_\lambda(x_1, \ldots, x_n;t,t)$.
Note that specializing~\eqref{ellhook} at $q=t$ gives
\begin{equation}
s_\lambda(1,t,t^2, \ldots, t^{n-1}) = t^{n(\lambda)}\prod_{1\le i<j\le n} \frac{1-t^{\lambda_i-\lambda_j+(j-i)}}{1-t^{j-i}}
= t^{n(\lambda)}\prod_{b\in \lambda} \frac{1-t^{n+c(b)}}{1-t^{h(b)}},
\label{qdimLlambda}
\end{equation}
where $h(b)$ is the hook length of the box $b$ and $c(b)$ is the content of the box $b$.
Moreover, setting $t=1$ in~\eqref{qdimLlambda} gives
\begin{equation}
s_\lambda(1,1,\ldots, 1) 
= \prod_{1\le i<j\le n} \frac{\lambda_i-\lambda_j+j-i}{j-i}
= \prod_{b\in \lambda} \frac{n+c(b)}{h(b)}.
\label{dimLlambda}
\end{equation}

These are special cases of Weyl's integral formula and Weyl's dimension formula (see~\cite[Ch.\ VI (1.7)]{BrtD}).

There is also a connection with the representation theory of $GL_n(\CC)$. Let $\mathrm{char}(L(\lambda))$ denote the character of the irreducible polynomial representation of
$GL_n(\CC)$ indexed by $\lambda$ (see~\cite[Ch.\ I App.\ A (8.4)]{Mac}).  
Letting $e^x = \diag(x_1, \ldots, x_n)$ in $GL_n(\CC)$, the Weyl character formula (see~\cite[Theorem 10.4]{Kac})
says that
$$
s_\lambda(x_1, \ldots, x_n)
= \mathrm{char}(L(\lambda)) = \Tr(L(\lambda), e^x),$$
where $\Tr(L(\lambda),g)$ denotes the trace of the action of $g$ on the vector space $L(\lambda)$.
Letting $e^{\rho^\vee} = \diag(1,t, \ldots, t^{n-1})$, the specialization
$$s_\lambda(1,t,t^2, \ldots, t^{n-1}) = \Tr(L(\lambda), e^{\rho^\vee}) = \mathrm{qdim}(L(\lambda))$$
is the \emph{quantum dimension} of $L(\lambda)$ (see~\cite[Prop,\ 10.10]{Kac}) and the specialization
$$s_\lambda(1,1, \ldots, 1) = \Tr(L(\lambda), 1) = \mathrm{dim}(L(\lambda))$$
is the dimension of $L(\lambda)$ (see~\cite[Cor.\ 10.10]{Kac}).  It would be interesting to give an interpretation of the formula
from~\eqref{elldim}
as an ``elliptic dimension'' formula for $L(\lambda)$.

\subsection{Geometric Satake}

Let us outline the context of the Boson-Fermion correspondence for symmetric polynomials
and the Weyl character formula, and describe some amazing relations between these structural features
and the geometric and  representation-theoretic settings.

\textbf{The case $q=0$ and $t=0$.}  
Consider the simple reflections in $S_n$, $s_i = (i,i+1)$ for $1\leq i\leq n-1$, and the spaces
\begin{align*}
\CC[X]^{S_n} &= \{ f\in \CC[X]\ |\ \hbox{if $i\in \{1, \ldots, n-1\}$ then $s_if=f$}\}
\qquad\hbox{and}
\\
\CC[X]^{\det} &= \{ f\in \CC[X]\ |\ \hbox{if $i\in \{1, \ldots, n-1\}$ then $s_if=-f$}\}.
\end{align*}
Let $w_0=(n,n-1,\dots, 1)\in S_n$, with $\ell(w_0 ) = \frac12n(n-1)$, and define the following operators
$$p_0 = \sum_{w\in S_n} w
\qquad\hbox{and}\qquad e_0 = \sum_{w\in S_n} (-1)^{\ell(w) - \ell(w_0)}w,
$$
For $\mu\in \ZZ^n$, let $x^\mu = x_1^{\mu_1}\cdots x_n^{\mu_n}$. The \emph{monomial symmetric polynomial} is 
$$m_\mu = \frac{1}{W_\mu(1)}p_0\, x^\mu = \frac{1}{W_\mu(1)} \sum_{w\in S_n} wx^\mu,$$
where the coefficient $\frac{1}{W_\mu (1)}$ makes the coefficient of $x^\mu$ in $m_\mu$ equal to 1.
The \emph{skew orbit sum} is
$$a_\mu = e_0x^\mu = \sum_{w\in S_n} (-1)^{\ell(w_0)-\ell(w)} x^{w\mu} = \det(x_i^{\mu_j}) .$$

The special case where $\rho = (n-1, n-2, \ldots, 2,1,0)$ gives the \emph{Vandermonde determinant},
$$
a_\rho = (-1)^{\ell(w_0)}\det(x_i^{n-j}) = \prod_{1\le i<j\le n} (x_j-x_i).
$$

Let $(\ZZ^n)^{++} = \{ (\mu_1, \ldots, \mu_n)\in \ZZ^n\ |\ \mu_1>\cdots >\mu_n\}$, and recall that $(\ZZ^n)^+$ is the set of weakly decreasing sequences of integers.
Note that the following map gives a bijection between these two sets:
$$
\begin{matrix}
(\ZZ^n)^+ &\stackrel{\sim}\longrightarrow &(\ZZ^n)^{++} \\
\lambda &\longmapsto &\lambda+\rho
\end{matrix}
$$

Given $\lambda\in (\ZZ^n)^+$ and $\mu \in (\ZZ^n)^{++}$, we have that for 
$i\in \{1, \ldots, n-1\}$, $m_{s_i\lambda} = m_\lambda$ and $a_\mu = -a_{s_i\mu}$. Thus, 
\begin{align*}
\{ m_\lambda\ |\ \lambda\in (\ZZ^n)^+\}\qquad\hbox{is a basis of} \quad \CC[X]^{S_n} = p_0\CC[X],
\\
\{a_\mu\ |\ \lambda\in (\ZZ^n)^{++}\}
\qquad\hbox{is a basis of}\quad
\CC[X]^{\mathrm{det}} = e_0 \CC[X],
\end{align*}
Moreover, for $\lambda\in (\ZZ^n)^+$, the \emph{Schur polynomial} is
\begin{equation}
s_\lambda = \frac{a_{\lambda+\rho}}{a_\rho}.
\label{Schur}
\end{equation}

Schur definitively recognized the polynomial
$s_\lambda$ as the character of a finite-dimensional irreducible representation of the group $GL_n(\CC)$.
A way of making the Schur polynomial very natural is to recognize that the
following diagram of vector space isomorphisms tells us that $\CC[X]^{\det}$ is a free (rank 1)
$\CC[X]^{S_n}$-module with basis vector $a_\rho$.
\begin{equation}
\begin{matrix}
\CC[X]^{S_n} 
&\stackrel{\sim}\longrightarrow
&\CC[X]^{\mathrm{det}} = a_\rho \CC[X]^{S_n} \\[0.1in]
f &\longmapsto  &a_\rho f \\[0.1in]
s_\lambda &\longmapsto 
&a_{\lambda+\rho} = e_0x^{\lambda+\rho} \\[0.1in]
m_\lambda = p_0x^\lambda &\longmapsto &a_\rho m_\lambda
\end{matrix}
\label{HWeyl}
\end{equation}

This isomorphism can be thought of as a version of the Boson-Fermion correspondence for symmetric polynomials.
Hermann Weyl used this isomorphism in his generalization of Schur's result which recognized that
the analogues of the $s_\lambda$ for crystallographic reflection groups (Weyl groups) provide
the characters of the finite-dimensional irreducible representations of compact Lie groups.

\medskip\noindent

\textbf{The case of $q=0$ and general $t$.}  
In view of the operators $T_1, \ldots, T_{n-1}$
from Section~\ref{sec:eig},
the $t$-analogues of the elements $p_0$ and $e_0$ are given by
$$
\mathbf{1}_0 = \sum_{z\in S_n} t^{\frac12(\ell(z)-\ell(w_0))}T_z
\qquad\hbox{and}\qquad
\varepsilon_0 = \sum_{w\in S_n} (-t^{-\frac12})^{\ell(z)-\ell(w_0)} T_z.
$$
It is fruitful to think of the polynomial ring
$\CC[X] = \CC[X]$ 
as generated by a single element $\mathbf{1}_0$
via multiplication by the variables/operators $x_1^{\pm1}, \ldots, x_n^{\pm1}$. 
With this point of view, 
the polynomial ring $\CC[X] = H\mathbf{1}_0$ is an induced representation of the affine Hecke algebra $H$,
where $H$ is the algebra generated by $T_1, \ldots, T_{n-1}$ and $x_1^{\pm1}, \ldots, x_n^{\pm1}$\,: 
$$\CC[X] \cong H\mathbf{1}_0 = \hbox{span}\{ x^\mu\mathbf{1}_0\ |\ \mu\in \ZZ^n\}.$$

For $\mu\in \ZZ^n$, the \emph{Whittaker function}
$$A_\mu(0,t)\mathbf{1}_0 \in \varepsilon_0 H \mathbf{1}_0
\qquad\hbox{is defined by}\qquad
A_\mu(0,t) = \varepsilon_0 x^\mu \mathbf{1}_0.$$
See, for example, [HKP, \S 6] for the connection between $p$-adic groups and the 
affine Hecke algebra and the explanation of why $A_\mu$ is equivalent to 
the data of a (spherical) Whittaker function for a $p$-adic group. 
As proved carefully in~\cite[Theorem 2.7]{NR04}, 
$$\varepsilon_0 H \mathbf{1}_0
\quad\hbox{has $\CC$-basis}
\quad
\{ A_{\lambda+\rho}(0,t) \ |\ \lambda\in (\ZZ^n)^+\}.$$
Following~\cite{Lu83} (see~\cite[Theorem 2.4]{NR04} for another exposition),
$$\begin{array}{ll}
\hbox{the Satake isomorphism,}  &\CC[X]^{S_n} \cong \mathbf{1}_0H\mathbf{1}_0,
\qquad\hbox{and} \\
\hbox{the Casselman-Shalika formula,}  \quad &A_{\lambda+\rho}(0,t) = s_\lambda A_\rho,
\end{array}
$$
can be formulated by the following diagram of vector space isomorphisms:
\begin{equation}
\begin{matrix}
Z(H) = \CC[X]^{S_n} 
&\stackrel{\sim}\longrightarrow
&\mathbf{1}_0 H \mathbf{1}_0 
&\stackrel{\sim}\longrightarrow
&\varepsilon_0 H \mathbf{1}_0  \\[0.1in]
f &\longmapsto &f\mathbf{1}_0 &\longmapsto &A_\rho f \mathbf{1}_0 \\[0.1in]
s_\lambda &\longmapsto &s_\lambda \mathbf{1}_0 &\longmapsto 
&A_{\lambda+\rho}(0,t)=\varepsilon_0 x^{\lambda+\rho} \mathbf{1}_0 \\[0.1in]
P_\lambda(0,t) &\longmapsto 
&P_\lambda(0,t)\mathbf{1}_0=\mathbf{1}_0 x^\lambda \mathbf{1}_0
\end{matrix}
\label{GeomLang}
\end{equation}
where $A_\rho = A_\rho(0,t)$.
As explained by Lusztig~\cite{Lu83},
in this diagram
\begin{enumerate}[itemsep=1mm]
\item[] $\mathbf{1}_0 H \mathbf{1}_0$ is the \emph{spherical Hecke algebra},
\item[] $s_\lambda$ is the \emph{Schur polynomial},
\item[] $P_\lambda(0,t)$ is the \emph{Hall-Littlewood polynomial}, and
\item[] $\{P_\lambda(0,t)\mathbf{1}_0\ |\ \lambda\in (\ZZ^n)^+\}$ 
is the \emph{Kazhdan-Lusztig basis} of $\mathbf{1}_0 H \mathbf{1}_0$.
\end{enumerate}
The spherical Hecke algebra $\mathbf{1}_0 H \mathbf{1}_0$ is the
Iwahori-Hecke algebra corresponding to the \emph{loop Grassmanian} 
$GL_n(\CC((\epsilon)))/GL_n(\CC[[t]])$.  The statement that
$P_\lambda(0,t)\mathbf{1}_0$ is a Kazhdan-Luszitg basis element in $\mathbf{1}_0 H \mathbf{1}_0$
indicates that $P_\lambda(0,t)\mathbf{1}_0$ corresponds to the intersection homology
of a Schubert variety in the loop Grassmannian (amazing!).

The diagram~\eqref{GeomLang} has particular importance due to the fact that $\CC[X]^{S_n} = \CC[X]^{W_0}$ 
(where $W_0$ is the Weyl group) is an avatar of the Grothendieck group of the category $\mathrm{Rep}(G)$ 
of finite dimensional representations of $G$, 
the spherical Hecke algebra $\mathbf{1}_0 H \mathbf{1}_0$ is a form of the 
Grothendieck group of $K$-equivariant perverse sheaves on the loop Grassmannian 
$Gr$ for the Langlands dual group $G^\vee$, 
and $\varepsilon_0 H \mathbf{1}_0$ is isomorphic to the Grothendieck group of Whittaker sheaves (appropriately formulated $N$-equivariant sheaves on $Gr$); see [FGV].

\medskip\noindent
\textbf{An analogous picture for general $q$ and general $t$.}
The results in Theorem~\ref{BosFerCorr}   and Theorem~\ref{WCF} provide an
analogous diagram for Macdonald polynomials.  Letting $\widetilde{H}$ be the affine Hecke algebra
and writing the polynomial
representation of $\widetilde H$ as $\CC[X]\cong \widetilde{H}\mathbf{1}_Y$ as in 
~\eqref{CXasIndHY},
then we have the following diagram:
\begin{equation}
\begin{array}{ccccl}
\CC[X]^{S_n} &\longrightarrow &\CC[X]^{S_n}\mathbf{1}_Y =\mathbf{1}_0 \widetilde{H} \mathbf{1}_Y 
&\longrightarrow &A_{\rho}\CC[X]^{S_n}=\varepsilon_0 \widetilde{H} \mathbf{1}_Y \\[0.1in]
f &\longmapsto &f\mathbf{1}_Y 
&\longmapsto &A_\rho f\mathbf{1}_Y \\[0.1in]
P_{\lambda}(q,qt) &\longmapsto &P_{\lambda}(q,qt)\mathbf{1}_Y
&\longmapsto &A_{\lambda + \rho}(q,t) \mathbf{1}_Y
= \varepsilon_0 E_{\lambda + \rho}(q,t)\mathbf{1}_Y \\[0.1in]
P_{\lambda}(q,t) &\longmapsto 
&P_{\lambda}(q,t) \mathbf{1}_Y = \mathbf{1}_0 E_{\lambda}(q,t)\mathbf{1}_Y 
\end{array}
\label{qtGL}
\end{equation}
In this diagram $A_\rho = A_{0+\rho}(q,t) = A_\rho(0,t)$ and the statement that $P_\lambda(q,qt)$ on the left
maps to $A_{\lambda+\rho}(q,t)$ on the right is the \emph{Weyl character formula for Macdonald polynomials},
\begin{equation}
P_\lambda(q,qt) = \frac{A_{\lambda+\rho}(q,t)}{A_\rho}.
\label{WCFforMac}
\end{equation}
It would be interesting to understand this diagram in terms of geometric contexts analogous
to those which exist for the $q=0$ case.  Some progress in this direction is found, for example,
in Ginzburg-Kapranov-Vasserot~\cite{GKV95} and Oblomkov-Yun~\cite{OY14}.

\subsection{Demazure operators, Hecke operators, and intertwiners}

We start defining several operators in $\CC[X]$. 
\begin{itemize}
\item []For $i\in \{1,\ldots, n\}$, 
$$(y_if)(x_1, \ldots, x_n) = f(x_1, \ldots, x_{i-1}, q^{-1}x_i, x_{i+1}, \ldots, x_n). $$
\item[] For $i\in \{1, \ldots, n-1\}$,
\begin{align}
\partial_ i &= (1+s_i)\frac{1}{x_i-x_{i+1}},\quad
&C_{s_i} 
&=(1+s_i)\frac{t^{-\frac12}-t^{\frac12}x_i^{-1}x_{i+1}}{1-x_i^{-1}x_{i+1}},
\label{H0actiononCX}
\\
D_{i, i+1}  &= (1+s_i)\frac{1}{1-x_ix^{-1}_{i+1}} ,
&D_{i+1,i}  &= (1+s_i)\frac{1}{1-x^{-1}_ix_{i+1}}.
\label{Demazureops}
\end{align}

\item[]
\end{itemize}
For $i\in\{1,\ldots, n-1\}$, $C_{s_i} = T_i +t^{-\frac{1}{2}} = T_i^{-1}+t^{\frac{1}{2}}$.

The \emph{intertwiners or creation operators} $\tau^\vee_\pi$, $\tau^\vee_1, \ldots, \tau^\vee_{n-1}$ are given by
\begin{equation}
\tau^\vee_\pi = X_1T_1\cdots T_{n-1},
\qquad\hbox{and}\qquad
\tau^\vee_i 
= C_{s_i}-\frac{t^{\frac12} - t^{-\frac12}Y^{-1}_iY_{i+1}}{1-Y^{-1}_iY_{i+1}}.
\label{intwnrops}
\end{equation}
These interwiners are used to construct the electronic Macdonald polynomials $E_\mu$.

In the study of Schubert calculus and Demazure characters 
one encounters the Demazure operators given in~\eqref{Demazureops}.
The relation between the interwiners $\tau^\vee_i$ and Demazure operators $D_{i,i+1}$ and $D_{i+1, i}$ is
\begin{align*}
t^{-\frac12}\tau^\vee_i 
&= D_{i,i+1} +t^{-1}(D_{i+1,i}-1)+\frac{(1-t^{-1})Y^{-1}_{i+1}Y_i}{1-Y^{-1}_{i+1}Y_i}
\qquad\hbox{and}
\\
t^{\frac12}\tau^\vee_i 
&= t (D_{i,i+1}-1) + D_{i+1,i}
+ \frac{(1-t)Y^{-1}_i Y_{i+1}}{1-Y^{-1}_iY_{i+1}}.
\end{align*}
The explicit eigenvalues in Theorem~\ref{Eeigenvalue},
and the above formulas for $\tau^\vee_i$ then give
\begin{align*}
t^{\frac12} \tau^\vee_i E_\mu(q,0) &= D_{i+1,i}E_\mu(q,0),
\quad\hbox{when 
$\mu_i>\mu_{i+1}$, and}
\\
t^{-\frac12} \tau^\vee_i E_\nu(q,\infty) &= D_{i,i+1}E_\nu(q,\infty),
\quad\hbox{when $\nu_i<\nu_{i+1}$.}
\end{align*}
This is the reason why the specializations of Macdonald polynomials at $t=0$ and $t=\infty$ 
contain information about Demazure characters (see~\cite{Ion01}
and~\cite[Theorem 1.3]{MRY19}).

In the study of (principal series) representations of $GL(\QQ_p)$ one encounters
the Iwahori-Hecke algebra operators $T_1, \ldots, T_{n-1}$ and $T_\pi$ given by
\begin{equation}
T_i = C_{s_i}-t^{-\frac12} = t^{\frac12}s_i + \frac{t^{\frac12}-t^{-\frac12}}{1- X_iX^{-1}_{i+1}}(1-s_i),
\qquad\hbox{for $i\in \{1, \ldots, n-1\}$.}
\label{Heckeops}
\end{equation}
The relation between the interwiners $\tau^\vee_i$ and the Iwahori-Hecke algebra operators $T_i$ is
\begin{equation}\label{interwinersrels}
t^{-\frac12}\tau^\vee_i = t^{-\frac12}T_i + \frac{(1-t^{-1})Y_iY^{-1}_{i+1}}{1-Y_iY^{-1}_{i+1}}
\qquad\hbox{and}\qquad
t^{\frac12}\tau^\vee_i = t^{\frac12}T^{-1}_i + \frac{(1-t)Y^{-1}_iY_{i+1}}{1-Y^{-1}_iY_{i+1}}.
\end{equation}

The explicit eigenvalues in Theorem~\ref{Eeigenvalue},
and the above formulas for $\tau^\vee_i$ then give
\begin{align*}
t^{-\frac12} \tau^\vee_i E_\mu(0,t) &= t^{-\frac12}T_i^{-1}E_\mu(0,t),
\quad\hbox{when $q=0$ and $\mu_i>\mu_{i+1}$, and}
\\
t^{\frac12} \tau^\vee_i E_\nu(\infty,t) &= t^{-\frac12}T_i E_\nu(\infty,t),
\quad\hbox{when $q=\infty$ and $\nu_i<\nu_{i+1}$.}
\end{align*}
If $\lambda$ is weakly decreasing then $E_\lambda(0,t) = x^\lambda$ and
$E_{w_0\lambda}(\infty,t) = x^{w_0\lambda}$. 
It is because of these relationships that the specializations of Macdonald polynomials at 
$q=0$ and $q=\infty$ contain information 
about principal series representations of $GL_n(\QQ_p)$ (see~\cite{Ion04}).

\section{Macdonald polynomials}\label{sec:Macpolys}

In this section, we introduce the three types of Macdonald polynomials that are the focus of this paper,
the electronic Macdonald polynomials, the bosonic Macdonald polynomials, and the fermionic Macdonald polynomials.

\subsection{Electronic Macdonald polynomials}\label{ssec:Epolys}

The \emph{electronic Macdonald polynomial} $E_\mu$, 
for $\mu\in \ZZ^n$, is determined by the following recursive process:
\begin{enumerate}
\item[(E0)] $E_{(0,\ldots, 0)} = 1$, 
\item[(E1)] $E_{(\mu_n+1, \mu_1, \ldots, \mu_{n-1})} 
= q^{\mu_n} x_1 E_\mu(x_2, \ldots, x_n, q^{-1}x_1),$
\item[(E2)] If $(\mu_1, \ldots, \mu_n)\in \ZZ^n_{\ge 0}$ and $\mu_i>\mu_{i+1}$ then
$$
E_{s_i\mu} 
= \Big( \partial_i x_i - t x_i\partial_i 
+ \frac{(1-t)q^{\mu_i-\mu_{i+1}} t^{v_\mu(i)-v_\mu(i+1)} }{1-q^{\mu_i-\mu_{i+1}} t^{v_\mu(i)-v_\mu(i+1)} } \Big)E_\mu,
$$
where $v_\mu\in S_n$ is the minimal length permutation such that $v_\mu\mu$ is weakly increasing,
\item[(E3)] $E_{(\mu_1-1, \ldots, \mu_n-1)} = x^{-1}_1\cdots x^{-1}_n E_{(\mu_1, \ldots, \mu_n)}$.
\end{enumerate}

Proposition~\ref{Emutopterm} says that $E_\mu$ is a homogeneous polynomial and that
$E_\mu$ is $x^\mu$ plus a linear combination of lower terms.    The appropriate ordering determining
``lower terms'' is the DB-lex ordering, defined as follows.

The \emph{dominance partial order} on weakly decreasing elements of $\ZZ^n$ is given by 
$\lambda^+\le \mu^+$ if $\lambda_1+\cdots +\lambda_i \le \mu_1+\cdots+\mu_i$ for $i\in \{1, \ldots, n\}$
(where $\lambda^+ = (\lambda_1, \ldots, \lambda_n)$ and $\mu^+ = (\mu_1, \ldots, \mu_n)$ with
$\lambda \ge \cdots \ge\lambda_n$ and $\mu_1 \ge \cdots \ge \mu_n$).  The \emph{Bruhat order} on the
symmetric group $S_n$ is determined by setting $v< vs_{ij}$ if $\ell(v)<\ell(vs_{ij})$ (where
$s_{ij}$ denotes the transposition switching $i$ and $j$).  If $\mu\in \ZZ^n$ let $\mu^+$ be the weakly
decreasing rearrangement of $\mu$ and let $z_\mu\in S_n$ be minimal length such that
$\mu = z_\mu \mu^+$.
The \emph{DBlex order
(dominance-Bruhat-lexicographic order)} is the order on $\ZZ^n$ given by
$$\lambda\le \mu 
\quad\hbox{if }\qquad
\begin{array}{c}
\hbox{$\lambda^+< \mu^+$ in dominance order} \\
\hbox{or} \\
\hbox{$\lambda^+=\mu^+$ and $z_\lambda<z_\mu$ in Bruhat order.}
\end{array}
$$

\begin{prop}  \label{Emutopterm}
Let $\mu\in \ZZ^n$.  Then
$$E_\mu = x^\mu + \sum_{\nu<\mu\atop \vert\nu\vert = \vert \mu \vert} a_{\mu\nu}(q,t) x^\nu,
\qquad\hbox{with $a_{\mu\nu}(q,t)\in \CC(q,t)$.}$$
\end{prop}
This result follows by verifying
that if this property holds before applying one of the operations (E1), (E2), or (E3), then
it also holds after the application of (E1), (E2), or (E3).

\noindent
It follows from Proposition~\ref{Emutopterm} that
$\{ E_\mu \ |\ \mu\in \ZZ^n\}$ is basis of $\CC[X]$.

\subsection{Bosonic Macdonald polynomials}

The \emph{bosonic Macdonald polynomial}
$P_\lambda$, for $\lambda \in (\ZZ^n)^+$, is defined by
\begin{equation}
P_\lambda = P_\lambda(x;q,t) = \frac{1}{W_\lambda(t)}
\sum_{w\in S_n} w\Big( E_\lambda \prod_{i<j} \frac{x_i-tx_j}{x_i-x_j}\Big),
\label{Plambdadefn}
\end{equation}
where $W_\lambda(t)$ is the appropriate constant which makes the coefficient
of $x^\lambda$ equal to $1$ in $P_\lambda(q,t)$.  The constant $W_\lambda(t)$ 
is determined explicitly in Proposition~\ref{paraPoin}.

Various specializations of the $P_\lambda(x;q,t)$ have their own names.
\begin{align*}
&m_\lambda = P_\lambda(x;0,1) &&\hbox{are the \emph{monomial symmetric polynomials}}, \\
&s_\lambda = P_\lambda(x;0,0) &&\hbox{are the \emph{Schur polynomials}}, \\
&P_\lambda(x;0,t) &&\hbox{are the \emph{Hall-Littlewood polynomials}}.
\end{align*}

\begin{prop} \label{E0t} If $\lambda = (\lambda_1, \ldots, \lambda_n)\in \ZZ^n$ 
with $\lambda_1\ge \cdots \ge \lambda_n$ then $E_\lambda(0,t) = x^\lambda$.  
\end{prop}

\noindent
Using Proposition~\ref{E0t} gives formulas
\begin{equation*}
m_\lambda = \sum_{\gamma\in S_n\lambda} x^\gamma,
\qquad\qquad
s_\lambda = \sum_{w\in S_n} w\Big( x^\lambda \prod_{i<j} \frac{x_i}{x_i-x_j} \Big)
\end{equation*}
$$\hbox{and}
\qquad
P_\lambda(0,t) 
= \frac{1}{W_\lambda(t)} \sum_{w\in S_n} w\Big( x^\lambda \prod_{i<j} \frac{x_i-tx_j}{x_i-x_j} \Big),
$$
which are familiar formulas for the monomial symmetric polynomials, the Schur polynomials and the
Hall-Littlewood polynomials (see~\cite[Ch.\ 1 (2.1)]{Mac},~\cite[Ch.\ I (3.1)]{Mac} and~\cite[Ch.\ III (2.1)]{Mac}).
We do not provide a proof of Proposition~\ref{E0t} in this paper.  It is an easy consequence of the 
alcove walks formula for Macdonald polynomials given in~\cite[Theorem 2.2]{RY08}.

\subsection{Fermionic Macdonald polynomials}

The \emph{fermionic Macdonald polynomial} $A_{\lambda+\rho}$, for $\lambda \in (\ZZ^n)^+$, is
\begin{equation}
A_{\lambda+\rho}(q,t) 
= \Big(\prod_{i<j} \frac{x_j-tx_i}{x_i-x_j}\Big) \sum_{w\in S_n} (-1)^{\ell(w)} wE_{\lambda+\rho}(q,t).
\label{Alambdadefn}
\end{equation}

\begin{thm} (Weyl character formula)
Let  $\lambda \in \ZZ^n$ with $\lambda_1\ge \cdots \ge \lambda_n$.  Then
$$A_\rho(q,t) = \prod_{i<j} (x_j-tx_i)
\qquad\hbox{and}\qquad
P_\lambda(q,qt) = \frac{A_{\lambda+\rho}(q,t)} {A_\rho(q,t)}.$$
\end{thm}

\noindent
Note that even though $E_\rho(q,t)$ depends on $q$, the denominator $A_\rho(q,t) = A_\rho(t)$ depends
only on $t$.

\section{The DAHA, symmetrizers and $c$-functions}

This section sets up the \emph{operator calculus} which will be exploited in later sections to obtain
explicit combinatorial results for Macdonald polynomials.  We have defined so far six kinds of operators:
\begin{itemize}
\item  the multiplication by $x_i$ operators $X_1, \ldots, X_n$, 
\item the divided difference operators $\partial_1,\ldots, \partial_{n-1}$,
\item the Hecke operators $T_1, \ldots, T_{n-1}$, and the promotion operator $T_\pi$, 
\item the Cherednik-Dunkl operators $Y_1, \ldots, Y_n$, 
\item the creation/intertwiner operators
$\tau^\vee_1, \ldots, \tau^\vee_{n-1}$ and $\tau^\vee_\pi$, and 
\item the symmetrizers $\mathbf{1}_0$
and $\varepsilon_0$. 
\end{itemize}
We also define
\begin{equation}
T_\pi^\vee = X_1 T_1\cdots T_{n-1}.
\label{Tpiveedefn}
\end{equation}
In this
section we establish the relations that these operators satisfy and develop the handy calculus for working
with these tools.  It is in this section that we first see the \emph{$c$-functions} appear, particularly in the
notable formulas for the symmetrizers $\mathbf{1}_0$ and $\varepsilon_0$ which are given in 
Proposition~\ref{propslicksymmA}  and Proposition~\ref{propsymwparabA}. At this point, we consider the operators as \emph{symbols}, and Section~\ref{sec:actions} is dedicated to the study of these as operators on polynomials.  

\subsection{The double affine Hecke algebra (DAHA)} \label{subsec:DAHA}

Let $q, t^{\frac12} \in \CC^\times$.
The \emph{double affine Hecke algebra (of type $GL_n$)} is the algebra $\widetilde{H}$ generated by symbols
$T_\pi$ and $X_k$ and  $T_i$ for $i,k\in \ZZ$ with relations
\begin{equation}
T_{i+n} = T_i, \qquad X_{i+n} = q^{-1}X_i,
\qquad X_k X_\ell = X_\ell X_k,
\qquad\hbox{for $i,k,\ell\in \ZZ$;}
\label{periodicityrelsF}
\end{equation}
\begin{equation}
T_iT_{i+1}T_i = T_{i+1}T_i T_{i+1},
\qquad
T_iT_j = T_jT_i,
\qquad
T_i^2 = (t^{\frac12}-t^{-\frac12})T_i +1,
\label{HeckerelsF}
\end{equation}
for $i,j\in \ZZ$ with $j\not\in\{ i\pm 1\}$;
\begin{equation}
X_{i+1} = T_i X_i T_i,
\qquad\hbox{and}\qquad
T_i X_j = X_j T_i,
\label{XaffHeckerelsF}
\end{equation}
for $i\in \{1, \ldots, n-1\}$ and $j\in \{1, \ldots, n\}$ with $j\not\in \{i, i+1\}$; and
\begin{equation}
T_\pi X_i = X_{i+1}T_\pi
\qquad\hbox{and}\qquad
T_\pi T_i  = T_{i+1}T_\pi
\qquad\hbox{for $i\in \ZZ$.}
\label{DAHArels2F}
\end{equation}
It follows from the last relation in~\eqref{HeckerelsF} and the relations in~\eqref{XaffHeckerelsF} that
for $i\in \{1, \ldots, n-1\}$,
\begin{equation}
T_i X_i = X_{i+1}T_i - (t^{\frac12}-t^{-\frac12}) X_{i+1} \qquad\hbox{and}\qquad
T_i X_{i+1} = X_i T_i + (t^{\frac12}-t^{-\frac12}) X_{i+1}.
\label{TpastXF}
\end{equation}

\begin{prop} \label{gluerelations} (The glue relations)
The operators $X_1$, $T_i$, and $T_\pi^\vee$ satisfy the relations
\begin{equation}
T_1^{-1} T_\pi T_\pi^\vee = T_\pi^\vee T_\pi T_{n-1}
\qquad\hbox{and}\qquad
T^{-1}_{n-1}\cdots T_1^{-1} T_\pi  (T_\pi^\vee)^{-1} 
= q(T_\pi^\vee)^{-1} T_\pi T_{n-1}\cdots T_1.
\label{glueF}
\end{equation}
\end{prop}
\begin{proof}
We prove the first relation using~\eqref{XaffHeckerelsF} and~\eqref{DAHArels2F} in the following way
\begin{align*}
T_\pi^\vee T_\pi T_{n-1}
&= X_1 T_1 \cdots T_{n-1} T_\pi T_{n-1} 
=T_1^{-1}\cdots T_{n-1}^{-1}T_{n-1}\cdots T_1 X_1 T_1\cdots T_{n-1} T_\pi T_{n-1} \\
&= T_1^{-1}\cdots T_{n-1}^{-1} X_n T_\pi T_{n-1} 
= T_1^{-1} \cdots T_{n-1}^{-1}T_\pi X_{n-1} T_{n-1} \\
&= T_1^{-1} T_\pi T_1^{-1}\cdots T_{n-2}^{-1} X_{n-1} T_{n-1} 
= T_1^{-1} T_\pi T_1^{-1}\cdots T_{n-2}^{-1} T_{n-2}\cdots T_1 X_1 T_1\cdots T_{n-2} T_{n-1} \\
&= T_1^{-1} T_\pi X_1T_1\cdots T_{n-1} = T_1^{-1}T_\pi T_\pi^\vee.
\end{align*}

We prove the second relation by showing the equivalent relation obtained by taking inverses on both sides. That is, we need to show that 
$$
q^{-1}T_1^{-1}\cdots T_{n-1}^{-1}T_\pi^{-1}T_\pi^\vee = T_\pi^\vee T_\pi^{-1} T_1\cdots T_{n-1}.
$$
In this case we use~\eqref{periodicityrelsF} and~\eqref{XaffHeckerelsF}, so that
\begin{align*}
q^{-1}T_1^{-1}\cdots T_{n-1}^{-1}T_\pi^{-1}T_\pi^\vee
&= T_1^{-1}\cdots T_{n-1}^{-1} q^{-1}T_\pi^{-1} X_1 T_1\cdots T_{n-1} \\
&= T_1^{-1}\cdots T^{-1}_{n-1} X_n T_\pi^{-1} T_1\cdots T_{n-1} \\
&= T_1^{-1}\cdots T^{-1}_{n-1} T_{n-1}\cdots T_1 X_1 T_1\cdots T_{n-1} T_\pi^{-1} T_1\cdots T_{n-1} \\
&=  X_1 T_1\cdots T_{n-1} T_\pi^{-1} T_1\cdots T_{n-1} 
=T_\pi^\vee T_\pi^{-1} T_1\cdots T_{n-1}.
\end{align*}
\end{proof}

\subsection{A family of commuting elements}

The \emph{Cherednik-Dunkl operators} $Y_1, \ldots, Y_n$ are analogues of Murphy elements in the DAHA, and we want to show that they commute. To do so, we introduce the 
following pictorial representation:
\begin{equation*}
{ 
\def\TOP{2}\def\K{6}
T_\pi =
\TikZ{[scale=.5]
		\Under[0-0.5,1][6,0]
		\Pole[0-.05][0,1]
		 \foreach \x in {1,...,5} {\draw[thin, style=over] (0,\x) -- (2,\x+1);}
		\Caps[0-.05][0,\TOP][\K]
		\Pole[0-.05][1,2]
		\Over[0-0.5,1][1,2]
}}
\qquad\hbox{and}\qquad
{\def\TOP{2} \def\K{6}
T_i =
\TikZ{[scale=.5]
	\Pole[0-.05][0,2][\K]
	\Under[4,0][3,2]
	\Over[3,0][4,2]
	 \foreach \x in {1,2,5,\K} {
		 \draw[thin] (0,\x) -- (\TOP,\x);
		 }
	\Caps[0-.05][0,\TOP][\K]
	\Label[0,\TOP][3][\footnotesize $i+1$]
	\Label[0,\TOP][4][\footnotesize$i$]
}
\qquad\quad \text{for $i=1, \dots, n-1$.}
}
\end{equation*}
Using the definition of the Cherednik-Dunkl operators in terms of the $T_i$'s operators, we have that 
\begin{equation*}
{ 
\def\TOP{2}\def\K{6}
Y_j = T_{j-1}^{-1}\cdots T_1^{-1} T_\pi T_{n-1}\cdots T_j
=
\TikZ{[scale=.5]
		\Under[0-0.5,1][3,0]
		\Pole[0-.05][0,1]
		 \foreach \x in {1,2} {\draw[thin, style=over] (0,\x) -- (1,\x);}
		\foreach \x in {4,...,\K} {\draw[thin, style=over] (0,\x) -- (1,\x);}
		\foreach \x in {1,2} { \draw[thin, style] (1,\x) -- (2,\x); }
		\foreach \x in {4,...,\K} {\draw[thin, style] (1,\x) -- (\TOP,\x);}
		\Caps[0-.05][0,\TOP][\K]
		\Pole[0-.05][1,2]
		\Over[0-0.5,1][3,2]
		\Label[0,\TOP][3][{\footnotesize $j$}]
}\qquad\quad\hbox{for $j\in \{1, \ldots,n\}$.}
}
\end{equation*}

\begin{prop} \label{theYscommute} For $i,j\in \{1, \ldots, n\}$,
\begin{align}
&Y_iY_j = Y_j Y_i,
\qquad
T_\pi^n = Y_1\cdots Y_n
\quad\hbox{and}\quad
Y_1Y_n^{-1}
= T_0T_{n-1}\cdots T_1\cdots T_{n-1}.
\label{Ycommutation}\\
&\begin{array}{l}
T_iY_i = Y_{i+1}T_i+(t^{\frac12}-t^{-\frac12})Y_i, \\
T_i Y_{i+1} =Y_iT_i - (t^{\frac12}-t^{-\frac12})Y_i, 
\end{array}
\qquad\hbox{and}\qquad
T_i Y_j  = Y_j T_i\quad\hbox{if $j\not\in\{i,i+1\}$.}
\label{TpastYrelsB}
\end{align}
\end{prop}
\begin{proof}
    The pictorial representation provides an easy check of the relations in~\eqref{Ycommutation} and the second relation in~\eqref{TpastYrelsB}. 

    The relations $Y_{i+1} = T^{-1}_iY_iT^{-1}_i$ and $T_i^2= (t^{\frac12}-t^{-\frac12})T_i + 1$
and $T_i-T_i^{-1} = t^{\frac12}-t^{-\frac12}$ give the other two relations:
\begin{align*}
T_iY_i &= ((t^{\frac12}-t^{-\frac12})+T^{-1}_i)Y_i
=Y_{i+1}T_i + (t^{\frac12}-t^{-\frac12})Y_i, \\
T_i Y_{i+1} &= Y_i T^{-1}_i = Y_i (T_i - (t^{\frac12}-t^{-\frac12}))
=Y_iT_i - (t^{\frac12}-t^{-\frac12})Y_i.
\end{align*}
\end{proof}

\subsection{Creation operators}

Next, we look at the relation between the creation operators (or intertwiners) and the Cherednick-Dunkl operators. 

\begin{prop}  For $i\in \{1, \ldots, n-1\}$ and $j\in \{1, \ldots, n\}$ with $j\not\in\{i, i+1\}$,
\begin{equation}
\tau^\vee_i Y_i = Y_{i+1}\tau^\vee_i,
\qquad \tau_i^\vee Y_{i+1} = Y_i \tau^\vee_i, 
\qquad\hbox{and}\qquad
\tau^\vee_i Y_j = Y_j \tau^\vee_i.
\label{taupastYrels1}
\end{equation}
Moreover, for $j\in \{1, \ldots, n-1\}$,
\begin{equation}
\tau^\vee_\pi Y_j = Y_{j+1}\tau^\vee_\pi
\qquad\hbox{and}\qquad \tau^\vee_\pi Y_n = q^{-1}Y_1 \tau^\vee_\pi.
\label{taupastYrels2}
\end{equation}
\end{prop}
\begin{proof}
The statements in~\eqref{taupastYrels1} follow from the relations~\eqref{TpastYrelsB} and the fact that the $Y_i$'s all commute with each other.

For $i\in \{1, \ldots, n-2\}$,
\begin{align*}
T^\vee_\pi T_i 
&= X_1T_1\cdots T_{n-1}T_i = X_1 T_1\cdots T_i T_{i+1}T_i T_{i+2}\cdots T_{n-1} 
= X_1 T_1\cdots T_{i-1}T_{i+1}T_iT_{i+1}T_{i+2}\cdots T_{n-1} \\
&= X_1 T_{i+1} T_1 \cdots T_{i-1}T_iT_{i+1}\cdots T_{n-1} = T_{i+1}X_1T_1\cdots T_{n-1} = T_{i+1}T^\vee_\pi.
\end{align*}
Using these relations and the relations from~\eqref{glueF}, 
\begin{align*}
\tau^\vee_\pi Y_1
&= T_\pi^\vee T_\pi T_{n-1}\cdots T_1
=T_1^{-1} T_\pi T_\pi^\vee T_{n-2}\cdots T_1
=T_1^{-1} T_\pi T_{n-1}\cdots T_2T_\pi^\vee \\
&= T_1^{-1}T_\pi  T_{n-1}\cdots T_2T_1T_1^{-1}T_\pi^\vee
= T_1^{-1}Y_1 T_1^{-1}T_\pi^\vee = Y_2T_\pi^\vee = Y_2 \tau^\vee_\pi,
\quad\hbox{and}  
\\
\tau^\vee_\pi Y_n
&= T_\pi^\vee T^{-1}_{n-1}\cdots T_1^{-1} Y_1 T_1^{-1} \cdots T_{n-1}^{-1} \\
&= T_\pi^\vee T^{-1}_{n-1}\cdots T_1^{-1}  T_\pi T_{n-1}\cdots T_1 T_1^{-1} \cdots T_{n-1}^{-1}
= T_\pi^\vee T^{-1}_{n-1}\cdots T_1^{-1} T_\pi  (T_\pi^\vee)^{-1} T_\pi^\vee \\
&= T_\pi^\vee q(T_\pi^\vee)^{-1} T_\pi T_{n-1}\cdots T_1 T_\pi^\vee 
= qY_1 T_\pi^\vee = qY_1 \tau^\vee_\pi.
\end{align*}
Now, for $i\in \{2, \ldots, n-1\}$, we use induction, so that
$$\tau^\vee_\pi Y_i 
= T_\pi^\vee T_{i-1}^{-1}\cdots T_1^{-1} Y_1 T_1^{-1} \cdots T_{i-1}^{-1} 
= T_i^{-1}\cdots T_2^{-1} Y_2 T_2^{-1}\cdots T_i^{-1} T_\pi^\vee = Y_{i+1}T_\pi^\vee 
= Y_{i+1}\tau^\vee_\pi.
$$
\end{proof}

\subsection{XY-parallelism}

In this section, we highlight the parallelism between the $X_i$ operators and the $Y_i$ operators. 
To do so, we define the normalized intertwiners, which are defined in terms of the $c$-functions. 
For $i,j\in \ZZ$ with $i\ne j$, define the $c$-fuctions
\begin{equation*}
c^X_{ij} = \frac{t^{-\frac12}-t^{\frac12}X_iX_j^{-1}}{1-X_iX_j^{-1}}
\qquad\hbox{and}\qquad
c^Y_{ij} = \frac{t^{-\frac12}-t^{\frac12}Y_iY_j^{-1}}{1-Y_iY_j^{-1}}.
\end{equation*}

For $i\in \{1,\ldots, n-1\}$ define the normalized intertwiners $\eta_{s_i}$ and $\xi_{s_i}$ by
\begin{equation*}
\eta_{s_i}
= \frac{1}{c^Y_{i+1,i}}(C_{s_i} - c^Y_{i,i+1})
\qquad\hbox{and}\qquad
\xi_{s_i}
= \frac{1}{c^X_{i,i+1}}(C_{s_i}-c^X_{i+1,i}).
\end{equation*}

\begin{prop}  \label{normintwn}
For $i, j,k\in \{1, \ldots, n-1\}$ with $k\not\in \{j \pm 1\}$,
\begin{align}
\eta_{s_i}\eta_{s_{i+1}}\eta_{s_i} &= \eta_{s_i}\eta_{s_{i+1}}\eta_{s_i},
&\eta_{s_j}\eta_{s_k} &= \eta_{s_k}\eta_{s_j},
&\eta_{s_i}C_{s_i} &= C_{s_i},
\label{nintbdrels}
\\
\xi_{s_i}\xi_{s_{i+1}}\xi_{s_i} &= \xi_{s_i}\xi_{s_{i+1}}\xi_{s_i},
&\xi_{s_j}\xi_{s_k} &= \xi_{s_k}\xi_{s_j},
&\xi_{s_i}C_{s_i} &= C_{s_i}.
\nonumber
\end{align}
Moreover, for $i\in \{1, \ldots, n-1\}$ and $j\in \{1, \ldots, n\}$,
\begin{equation}
\begin{array}{l}
\eta_{s_i} Y_i = Y_{i+1}\eta_{s_i}, \\
\eta_{s_i} Y_{i+1} = Y_i\eta_{s_i}, \\
\eta_{s_i} Y_j = Y_j\eta_{s_i}, \quad\hbox{if $j\not\in \{i, i+1\}$,}
\end{array}
\qquad\hbox{and}\qquad
\begin{array}{l}
\xi_{s_i} X_i = X_{i+1}\xi_{s_i}, \\
\xi_{s_i} X_{i+1} = X_i\xi_{s_i}, \\
\xi_{s_i} X_j = X_j\xi_{s_i}, \quad\hbox{if $j\not\in \{i, i+1\}$.}
\end{array}
\label{nintpastY}
\end{equation}
\end{prop}

\begin{proof} The first set of relations in~\eqref{nintpastY} follow from the relations in~\eqref{taupastYrels1}.

Using the relations
$$C_{s_i}^2 = (t^{\frac12}+t^{-\frac12})C_{s_i},
\qquad
C_{s_i} = \tau^\vee_i +c^Y_{i,i+1},
\qquad
c^Y_{i,i+1}+c^Y_{i+1,i} = t^{\frac12}+t^{-\frac12},
$$
and the relations in~\eqref{taupastYrels1} gives
\begin{equation}
(\tau^\vee_i)^2 = c^Y_{i,i+1}c^Y_{i+1,i}, \qquad
(\eta_{s_i}+1)c_{i,i+1}^Y = C_{s_i}
\qquad\hbox{and}\qquad
\eta_{s_i}C_{s_i} = C_{s_i}.
\label{etasqcomp}
\end{equation}
A direct computation together with the relations in~\eqref{taupastYrels1} produces
\begin{align}
C_{s_1s_2s_1} &=C_{s_1}C_{s_2}C_{s_1}-C_{s_1} = C_{s_2}C_{s_1}C_{s_2}-C_{s_2}
\nonumber \\
&= T_1T_2T_1+t^{-\frac12}T_1T_2+t^{-\frac12}T_2T_1
+t^{-\frac22}T_1+t^{-\frac22}T_2+t^{-\frac32},  \nonumber \\
&= T_2T_1T_2+t^{-\frac12}T_1T_2+t^{-\frac12}T_2T_1
+t^{-\frac22}T_1+t^{-\frac22}T_2+t^{-\frac32}, \nonumber  \\
&= (\eta_{s_1}\eta_{s_2}\eta_{s_1} + \eta_{s_1}\eta_{s_2}+\eta_{s_2}\eta_{s_1}+\eta_{s_2}+\eta_{s_1}+1)
c^Y_{12}c^Y_{13}c^Y_{23} \nonumber  \\
&= (\eta_{s_2}\eta_{s_1}\eta_{s_2} + \eta_{s_1}\eta_{s_2}+\eta_{s_2}\eta_{s_1}+\eta_{s_2}+\eta_{s_1}+1)
c^Y_{12}c^Y_{13}c^Y_{23}.
\label{Cs1s2s1}
\end{align}
This establishes that $\eta_{s_1}\eta_{s_2}\eta_{s_1} = \eta_{s_2}\eta_{s_1}\eta_{s_2}$ is
equivalent to $T_1T_2T_1 = T_2T_1T_2$ which is equivalent to
$C_{s_1}C_{s_2}C_{s_1}-C_{s_1} = C_{s_2}C_{s_1}C_{s_2}-C_{s_2}$.

The relations for $\xi_{s_i}$ are
the same as the relations for $\eta_{s_i}$ because the relations in~\eqref{TpastXF} are the same as the relations~\eqref{TpastYrelsB} except for a replacing
$Y_i$ with $X_{i+1}$ and replacing $Y_{i+1}$ with $X_i$. Thus, the same proof also works for the relations for $\xi_{s_i}$. 
\end{proof}

The XY-parallelism in this section is the core of the duality in double affine Artin groups and 
double affine Hecke algebras (see~\cite[\S3.5]{Mac03}).

\begin{remark}
The reader should be warned to be careful with denominators when working with the normalized interwiners.
When using the relations in Proposition~\ref{normintwn} 
it is wise to always write any denominators as the leftmost factors so that
when acting on (left) $\tilde H$-modules the expressions are 0 as appropriate. 
As a concrete example, from the first identity
in~\eqref{etasqcomp} it follows that
\begin{equation}
\eta_{s_i}^2 = \frac{1}{c_{i,i+1}^Yc_{i+1,i}^Y}c^Y_{i,i+1}c_{i+1,i}^Y\,,
\qquad\hbox{which \emph{does not always} reduce to}\quad \eta_{s_i}^2 = 1.
\label{etasqwarning}
\end{equation}
Even on the polynomial representation considered in Section~\ref{sec:actions},
the operator $\eta_{s_i}$ can act  by $0$ on some vectors (for example $\eta_{s_i}\cdot 1 = 0$).  
However the identity  for $\eta_{s_i}^2$ in~\eqref{etasqwarning} is still valid since $\eta_{s_i}^2\cdot 1=0$ and
$$
\frac{1}{c_{i,i+1}^Yc_{i+1,i}^Y}c^Y_{i,i+1}c_{i+1,i}^Y\cdot 1 = 0
\quad\hbox{(because $c_{i,i+1}^Y\cdot 1 = 0$),}
$$
In the language of localizations of rings, the elements $\eta_{s_i}$ and $\eta_w$ (defined in~\eqref{etawxiv})
do not live in the double affine Hecke algebra $\widetilde{H}$ proper but in a localization at the multiplicative set generated
by the elements $(1-Y_iY^{-1}_j)$ and $(1-tY_iY^{-1}_j)$.  
\end{remark}

\subsection{Symmetrizers}

Given $w\in S_n$ with reduced word $w = s_{i_1}\cdots s_{i_\ell}$, we define
\begin{equation}
\xi_w = \xi_{s_{i_1}}\cdots \xi_{s_{i_\ell}},
\qquad 
\eta_w = \eta_{s_{i_1}}\cdots \eta_{s_{i_\ell}},
\qquad\hbox{and}\qquad 
T_w = T_{s_{i_1}}\cdots T_{s_{i_\ell}},
\label{etawxiv}
\end{equation}
together with the following symmetrizers
\begin{align*}
\hbox{$X$-symmetrizer}\quad &p^X_0 = \sum_{w\in S_n} \xi_w\ , 
&\hbox{$X$-antisymmetrizer}\quad 
&e^X_0 = \sum_{w\in S_n} (-1)^{\ell(w)-\ell(w_0)} \xi_w,
\\
\hbox{$Y$-symmetrizer}\quad &p_0^Y = \sum_{w\in S_n} \eta_w\ ,
&\hbox{$Y$-antisymmetrizer}\quad 
&e_0^Y = \sum_{w\in S_n} (-1)^{\ell(w)-\ell(w_0)} \eta_w.
\end{align*}
The \emph{bosonic symmetrizer} and the \emph{fermionic symmetrizer} are defined by 
\begin{equation}
\mathbf{1}_0 = \sum_{w\in S_n} t^{\frac12(\ell(w)-\ell(w_0))}T_w
\qquad\hbox{and}\qquad
\varepsilon_0 = \sum_{w\in S_n} (-t^{-\frac12})^{\ell(w)-\ell(w_0)} T_w.
\label{bosfersymm}
\end{equation}
The bosonic symmetrizer $\mathbf{1}_0$ is a $t$-analogue of $p^X_0$ and $p^Y_0$
and the fermionic symmetrizer $\varepsilon_0$ is a $t$-analogue of 
$e^X_0$ and $e^Y_0$.
Using the relation $T_i^2 = (t^{\frac12}-t^{-\frac12})T_i+1$ and the last relation in~\eqref{etasqcomp}
gives
$$T_i\mathbf{1}_0 = t^{\frac12}\mathbf{1}_0,
\qquad C_{s_i}\mathbf{1}_0 = (t^{\frac12}+t^{-\frac12})\mathbf{1}_0
\quad\hbox{and}\qquad
\eta_{s_i}\mathbf{1}_0 = \mathbf{1}_0,
$$
since $\eta_{s_i}\mathbf{1}_0 
= \eta_{s_i}(t^{\frac12}+t^{-\frac12})^{-1}C_{s_i}\mathbf{1}_0
= (t^{\frac12}+t^{-\frac12})^{-1}C_{s_i}\mathbf{1}_0 = \mathbf{1}_0$.

Let us denote by $w_0$ the longest element in $S_n$. That is, 
$w_0(i) = n-i+1$, for $i\in \{1, \ldots, n\}$. It has maximal length, $\ell(w_0) =\frac{n(n-1)}{2} = \binom{n}{2}$, since its set of inversions is  
$$\mathrm{Inv}(w_0) = \{ (i,j)\ |\ \hbox{$i,j\in \{1, \ldots, n\}$ and $i<j$} \}.$$ 
Then, we define
$$c_{w_0}^X =  \prod_{1\le i<j\le n} c_{ij}^X
= \prod_{1\le i<j\le n} \frac{t^{-\frac12}-t^{\frac12}X_iX_j^{-1}}{1-X_iX_j^{-1}}
\qquad\hbox{and}\qquad
c_{w_0}^{X^{-1}} 
= \prod_{1\le i<j\le n} \frac{t^{-\frac12}-t^{\frac12}X^{-1}_iX_j}{1-X^{-1}_iX_j }.
$$

The following result rewrites the bosonic and fermionic symmetrizers in terms of the $X$-symmetrizers and the $Y$-symmetrizers.

\begin{prop} \label{propslicksymmA}  
\begin{equation}
\mathbf{1}_0 = p_0^X c_{w_0}^{X^{-1}} = p_0^Y c_{w_0}^Y \qquad\hbox{and}\qquad
\varepsilon_0 = c_{w_0}^X e_0^X = c_{w_0}^{Y^{-1}} e_0^Y.
\label{slicksymmA}
\end{equation}
\end{prop}
\begin{proof}
Let $w\in S_n$ with reduced word $w = s_{i_1}\cdots s_{i_\ell}$.
Using that $T_i = \xi_{s_i} c_{i,i+1}^{X^{-1}} + (t^{\frac12}-c_{i,i+1}^{X})$ and expanding the result, we have that 
$$T_w = T_{i_1}\cdots T_{i_\ell} = T_w = \xi_w c_w^{X^{-1}} + \sum_{v<w} \xi_v b_v(X),
\qquad\hbox{with $b_v(X)\in \CC(X_1, \ldots, X_n)$.}
$$
Thus there are 
$a_v(X)\in \CC(X_1, \ldots, X_n)$ such that
\begin{equation}
\mathbf{1}_0 = \sum_{w\in S_n} t^{-\frac12\ell(w) - \ell(w_0)}T_w
= \xi_{w_0}c_{w_0}^{X^{-1}} + \sum_{w<w_0} \xi_v a_v(X),
\label{topterm}
\end{equation}

Since $p^X_0 = \sum_{w\in S_n} \xi_w$ and
$\xi_{s_i} p^X_0 c^{X^{-1}}_{w_0} = p^X_0 c^{X^{-1}}_{w_0}$ then
\begin{align*}
T_i (p^X_0 c^{X^{-1}}_{w_0})
&=  \big(c^X_{i,i+1} \xi_{s_i} + (t^{\frac12}-c^X_{i,i+1})\big) p^X_0  c_{w_0}^{X^{-1}} \\
&=  \big(c^X_{i,i+1} + (t^{\frac12}-c^X_{i,i+1})\big) p^X_0  c_{w_0}^{X^{-1}} 
=  t^{\frac12}\big(p^X_0  c_{w_0}^{X^{-1}}\big).
\end{align*}
Now, $\mathbf{1}_0$ is determined, up to multiplication by a constant, by the fact that 
$T_i \mathbf{1}_0 = t^{\frac12}\mathbf{1}_0$ for $i\in \{1, \ldots, n-1\}$, and so 
it follows from~\eqref{topterm} that
$\mathbf{1}_0 = p^X_0 c_{w_0}^{X^{-1}}$ (see the example for $n=3$ in~\eqref{Cs1s2s1},
where $\mathbf{1}_0 = C_{s_1s_2s_1}$).
The other relation in~\eqref{slicksymmA} is proven similarly. 
\end{proof}

\subsubsection{Symmetrizers and stabilizers}

Let $\lambda \in (\ZZ^n)^+$.
The stabilizer of $\lambda$ under the action of $S_n$ is
$W_\lambda = \{ v\in S_n\ |\ v\lambda= \lambda\}$, and denote by $w_\lambda$ its longest element.
Let  $W^\lambda$ be the set of minimal length representatives of the cosets in $S_n/W_\lambda$, and denote by $v_\lambda$ its longest element. Then, 
$w_0 = v_\lambda w_\lambda$ with $\ell(w_0) = \ell(v_\lambda)+\ell(w_\lambda)$.
Consider the following symmetrizers and antisymmetrizers for $W^\lambda$ and $W_\lambda$.
\begin{equation}
\begin{array}{lclll}
\displaystyle{
p^\lambda_X = \sum_{u\in W^\lambda} \xi_u } 
&\hbox{and}
&\displaystyle{
p^X_\lambda = \sum_{v\in W_\lambda} \xi_v }
&\hbox{so that}\qquad
&p_0^X = p^\lambda_X p^X_\lambda, 
\\
\displaystyle{
e^\lambda_X = \sum_{u\in W^\lambda} \det(v_\lambda u) \xi_u } 
&\hbox{and}
&\displaystyle{
e^X_\lambda = \sum_{v\in W_\lambda} \det(w_\lambda v) \xi_v }
&\hbox{so that}\qquad
&e_0^X = e^\lambda_X e^X_\lambda, 
\\
\displaystyle{
p^\lambda_Y = \sum_{u\in W^\lambda} \eta_u } 
&\hbox{and}
&\displaystyle{
p^Y_\lambda = \sum_{v\in W_\lambda} \eta_v, }
&\hbox{so that}\qquad
&p_0^Y = p^\lambda_Y p^Y_\lambda,
\\
\displaystyle{
e^\lambda_Y = \sum_{u\in W^\lambda} \det(v_\lambda u) \eta_u } 
&\hbox{and}
&\displaystyle{
e^Y_\lambda = \sum_{v\in W_\lambda} \det(w_\lambda v) \eta_v }
&\hbox{so that}\qquad
&e_0^Y = e^\lambda_Y e^Y_\lambda.
\\
\end{array}
\label{pWs}
\end{equation}
We also define the $t$-analogues of the elements in~\eqref{pWs}:
\begin{equation}
\begin{array}{lcl}
\displaystyle{
\mathbf{1}^\lambda 
= \sum_{u\in W^\lambda} (t^{\frac12})^{\ell(u)-\ell(v_\lambda)} T_u }
&\hbox{and}
&\displaystyle{
\mathbf{1}_\lambda 
= \sum_{v\in W_\lambda} (t^{\frac12})^{(\ell(v)-\ell(w_\lambda)} T_v, }
\\
\\
\displaystyle{
\varepsilon^\lambda 
= \sum_{u\in W^\lambda} (-t^{-\frac12})^{\ell(u)-\ell(v_\lambda)} T_u }
&\hbox{and}
&\displaystyle{
\varepsilon_\lambda 
= \sum_{v\in W_\lambda} (-t^{-\frac12})^{\ell(v)-\ell(w_\lambda)} T_v. }
\end{array}
\label{pHs}
\end{equation}
Then 
$\mathbf{1}_0 = \mathbf{1}^\lambda \mathbf{1}_\lambda$ and 
$\varepsilon_0 = \varepsilon^\lambda \varepsilon_\lambda$. 

For $\lambda \in (\ZZ^n)^+$, consider the sets of inversions of the elements $w_\lambda$ and $v_\lambda$. 
\begin{align*}
\mathrm{Inv}(w_\lambda) &= \{ (i,j)\ |\ \hbox{$i<j$ and $\lambda_i = \lambda_j$ } \}
\qquad\hbox{and}\qquad
\mathrm{Inv}(v_\lambda) = \{ (i,j)\ |\ \hbox{$i<j$ and $\lambda_i>\lambda_j$ } \}.
\end{align*}
Thus, the associated $c$-functions are defined by 
$$c^{X^{-1}}_{w_\lambda} 
= \prod_{1\le i< j\le n\atop \lambda_i = \lambda_j} c_{ij}^{X^{-1}}
\qquad\hbox{and}\qquad
c^{X^{-1}}_{v_\lambda} 
= \prod_{1\le i< j\le n\atop \lambda_i > \lambda_j} c_{ij}^{X^{-1}}.
$$
The following is a generalization of Proposition~\ref{propslicksymmA}.

\begin{prop} \label{propsymwparabA} 
For $\lambda \in (\ZZ^n)^+$,
\begin{equation*}
\mathbf{1}_0   
= p^\lambda_X c_{v_\lambda}^{X^{-1}} \mathbf{1}_\lambda
= p^\lambda_Y c_{v_\lambda}^Y \mathbf{1}_\lambda
\qquad\hbox{and} \qquad
\varepsilon_0 
= c_{v_\lambda}^X e^\lambda_X \varepsilon_\lambda 
= c_{v_\lambda}^{Y^{-1}} e^\lambda_Y \varepsilon_\lambda.
\end{equation*}
\end{prop}
\begin{proof} 
If $u\in W_\lambda$ and $\lambda_i>\lambda_j$  then $\lambda_{u(i)}>\lambda_{u(j)}$. Therefore, 
$$u\cdot \mathrm{Inv}(v_\lambda)
= \{ (u(i), u(j))\ |\ \hbox{$i<j$ and $\lambda_i>\lambda_j$}\}
= \mathrm{Inv}(v_\lambda),$$ 
and so, $w^{-1}_\lambda c^X_{v_\lambda} = u c^X_{v_\lambda} = c^X_{v_\lambda}$
for $u\in W_\lambda$.  This shows that 
\begin{equation}
c^{X^{-1}}_{w_0} = c^{X^{-1}}_{v_\lambda w_\lambda} 
= (w^{-1}_\lambda c^{X^{-1}}_{v_\lambda}) c^{X^{-1}}_{w_\lambda}
= c^{X^{-1}}_{v_\lambda} c^{X^{-1}}_{w_\lambda}
\qquad\hbox{and}\qquad
p_\lambda^X c^{X^{-1}}_{v_\lambda} = c^{X^{-1}}_{v_\lambda} p^X_\lambda.
\label{cfcnsplit}
\end{equation}
Replacing $S_n$ by the group $W_\lambda$ in the proof of Proposition~\ref{propslicksymmA} gives
$\mathbf{1}_\lambda = p^X_\lambda c_{w_\lambda}^{X^{-1}}$.
Using the relations in~\eqref{cfcnsplit} and the identity
$\mathbf{1}_\lambda = p^X_\lambda c_{w_\lambda}^{X^{-1}}$ gives
\begin{align*}
\mathbf{1}_0 
&= p^X_0 c_{w_0}^{X^{-1}}
= p^\lambda_X p^X_\lambda (w^{-1}_\lambda c_{v_\lambda}^{X^{-1}}) c_{w_\lambda}^{X^{-1}}
= p^\lambda_X  c_{v_\lambda}^{X^{-1}} p^X_\lambda c_{w_\lambda}^{X^{-1}}
= p^\lambda_X  c_{v_\lambda}^{X^{-1}} \mathbf{1}_\lambda.
\end{align*}
The proof for $\varepsilon_0$ is similar.
\end{proof}

\section{The action on polynomials}\label{sec:actions}

In this section we see the operators of the last section acting on polynomials.
Using this action we can reinterpret the constructions of the Macdonald polynomials
given in Section~\ref{sec:Macpolys} as the result of operators acting on the initial polynomial 1.
These are the `creation formulas' for Macdonald polynomials.  The symmetrizer formulas from Proposition~\ref{propslicksymmA} and Proposition~\ref{propsymwparabA}  
are then used to set up and prove the Boson-Fermion correspondence in the
Macdonald polynomial setting and to provide explicit combinatorial formulas for the
expansion of bosonic and fermionic Macdonald polynomials in terms of the basis
of electronic Macdonald polynomials.  These symmetrizer formulas are 
also the key to product formulas for the Poincar\'e polynomials which arise as the
normalizing constants for the bosonic Macdonald polynomials $P_\lambda$.

We have structured this section so that the proofs are postponed to the last subsection.
This allows us to first focus on the results and the ``big picture'' structure that relates the
bosonic, fermionic, and electronic Macdonald polynomials and leaves the necessary computations
to the last subsection.

\subsection{DAHA acts on polynomials}

Recall the definition of the \emph{divided difference operators}
$\partial_i\colon \CC[X] \to \CC[X]$, the \emph{Hecke algebra operators}
$T_i\colon \CC[X]\to \CC[X]$ and the \emph{promotion operator} $T_\pi\colon \CC[X]\to \CC[X]$ by
\begin{equation}
\partial_i f = \frac{f-s_if}{x_i-x_{i+1}},
\qquad
T_i = t^{-\frac12}x_{i+1}\partial_i - t^{\frac12}\partial_i x_{i+1}
\qquad\hbox{and}\qquad
T_\pi = s_1\cdots s_{n-1}y_n.
\label{divdiffops}
\end{equation}

Letting
\begin{equation*}
s_0 = T_\pi s_{n-1} T_\pi^{-1}
= y_1y^{-1}_n s_{n-1}\cdots s_2s_1s_2\cdots s_{n-1},
\end{equation*}
then
\begin{equation*}
T_\pi s_0 T_\pi^{-1} = s_1,
\qquad
T_\pi X_n T^{-1}_\pi = q^{-1}X_1,
\qquad
T_\pi s_i T^{-1}_\pi = s_{i+1},
\quad\hbox{and}\quad
T_\pi X_i T^{-1}_\pi = X_{i+1},
\end{equation*}
for $i\in \{1, \ldots, n-1\}$.

\begin{thm} \label{PolyrepThm}  Let $\widetilde{H}$ be the double affine Hecke algebra
as defined by generators and relations in Section~\ref{subsec:DAHA}.
The formulas~\eqref{divdiffops} define an action of $\widetilde{H}$
on $\CC[X]$.
\end{thm}

A way of deriving the formulas in~\eqref{divdiffops} is to consider the induced representation
\begin{align*}
\CC[X] \cong \Ind_{H_Y}^{\widetilde{H}}(\mathbf{1}_Y)  = \widetilde{H}\mathbf{1}_Y
= \hbox{$\CC$-span}
\{ X_1^{\mu_1}\cdots X_n^{\mu_n}\mathbf{1}_Y \ |\ \mu=(\mu_1, \ldots, \mu_n) \in \ZZ^n\}
\end{align*}
determined by
\begin{equation}
T_\pi \mathbf{1}_Y = \mathbf{1}_Y
\qquad\hbox{and}\qquad
T_i \mathbf{1}_Y = t^{\frac12}\mathbf{1}_Y, \qquad \text{ for } i\in \{1,\ldots, n\}.
\label{HYtriv}
\end{equation}
Then the formulas in~\eqref{divdiffops} are consequences of the relations in 
~\eqref{XaffHeckerelsF} and~\eqref{DAHArels2F}.  In other words, the map
\begin{equation}
\begin{matrix}
\CC[X] &\longrightarrow &\widetilde{H}\mathbf{1}_Y \\
x^\mu &\longmapsto &X^\mu\mathbf{1}_Y
\end{matrix}
\qquad\hbox{is a $\widetilde H$-module isomorphism.}
\label{CXasIndHY}
\end{equation}

\subsection{Creation formulas}

Let $\mu = (\mu_1, \ldots, \mu_n) \in \ZZ_{\ge 0}^n$.  
The \emph{minimal length permutation} $v_\mu$ such that $v_\mu \mu$ is weakly increasing
is given by
$$v_\mu(r) = 1+\#\{r'\in \{1, \ldots, r-1\}\ |\ \mu_{r'}\le \mu_r\}
+\#\{r'\in \{r+1, \ldots, n\}\ |\ \mu_{r'}<\mu_r\},$$
for $r\in \{1, \ldots, n\}$.  A \emph{box in $\mu$} is a pair $(r,c)$ with 
$r\in \{1, \ldots, n\}$ and $c\in \{1, \ldots, \mu_r\}$.
If $b = (r,c)$ is a box in $\mu$ then define
$$u_\mu(r,c) = \#\{ r'\in \{1, \ldots, r-1\}\ |\ \mu_{r'}\le c-1\} + \#\{ r'\in \{r+1, \ldots, n\}\ |\ \mu_{r'}<c-1\}.$$
Let us take the following creation formula for $E_\mu$
from~\cite[ Proposition 5.5 and Proposition 2.2(a)]{GR21}, where a complete proof is included.

\begin{thm}
\label{creationformulathm}  
Let $\mu\in \ZZ_{\ge 0}^n$.  Letting
\begin{equation*}
\tau^\vee_{u_\mu} 
= \prod_{(r,c)\in \mu} 
(\tau^\vee_{u_\mu(r,c)}\cdots \tau^\vee_2\tau^\vee_1\tau^\vee_\pi),
\qquad\hbox{then}\qquad
E_\mu = t^{-\frac12\ell(v^{-1}_\mu)}\tau^\vee_{u_\mu} 1,
\end{equation*}
where $1$ is the polynomial $1\in \CC[X]$.
\end{thm}

The creation formulas for $P_\lambda$ and $A_{\lambda+\rho}$ are given in terms of the bosonic and fermionic symmetrizers defined in~\eqref{bosfersymm}.
\begin{thm}
Let $\lambda\in \left(\ZZ^n\right)^{++}$, and recall that $\rho=(n-1,n-2,\ldots, 1)$. Then
\begin{equation}
P_\lambda = \frac{t^{\frac12\ell(w_0)}}{W_\lambda(t)} \mathbf{1}_0 E_\lambda
\qquad\hbox{and}\qquad
A_{\lambda+\rho} = t^{\frac12\ell(w_0)}\varepsilon_0 E_{\lambda+\rho},
\label{creationPA}
\end{equation}
where $W_\lambda(t)$ is a normalizing constant which makes the coefficient of $x^\lambda$ in $P_\lambda$ equal to $1$.  
\end{thm}
By~\eqref{slicksymmA}, the formulas in~\eqref{creationPA} match the formulas in 
~\eqref{Plambdadefn} and~\eqref{Alambdadefn}.
The constant
$W_\lambda(t)$ is determined explicitly in Proposition~\ref{paraPoin}.

\subsection{The {B}oson-{F}ermion correspondence}

Define the following vector subspaces of $\CC[X]$
\begin{align*}
\CC[X]^{S_n} &= \{ f\in \CC[X]\ |\ s_if = f\ \text{ for all } i\in\{1,\ldots,n-1\} \},
\\
\CC[X]^{\mathrm{det}} &= \{ f\in \CC[X]\ |\ s_if = -f \text{ for all } i\in\{1,\ldots,n-1\} \},
\\
\CC[X]^{\mathrm{Bos}} &= \{ f\in \CC[X]\ |\ T_{s_i}f = t^{\frac12} f \text{ for all } i\in\{1,\ldots,n-1\} \},
\\
\CC[X]^{\mathrm{Fer}} &= \{ f\in \CC[X]\ |\ T_{s_i}f = -t^{-\frac12} f \text{ for all } i\in\{1,\ldots,n-1\} \}.
\end{align*}
The Boson-Fermion correspondence in Theorem~\ref{BosFerCorr} puts the formulas in~\eqref{creationPA} into a structural context.

\begin{thm} \label{BosFerCorr} 
Let $\mathbf{1}_0$ and $\varepsilon_0$ be as defined in~\eqref{bosfersymm}.
Then 
$$\mathbf{1}_0 \CC[X] = \CC[X]^{\mathrm{Bos}} = \CC[X]^{S_n}
\qquad\hbox{and}\qquad
\varepsilon_0 \CC[X] = \CC[X]^{\mathrm{Fer}} = A_\rho \CC[X]^{S_n}.$$
Moreover, there are $\CC[X]^{S_n}$-module isomorphisms
\begin{equation}
\begin{matrix}
\CC[X]^{S_n} &\to &\CC[X]^{\mathrm{det}} \\
f &\mapsto &a_\rho f
\end{matrix}
\qquad\hbox{and}\qquad
\begin{matrix}
\CC[X]^{\mathrm{Bos}} &\to &\CC[X]^{\mathrm{Fer}} \\
f &\mapsto &A_\rho f
\end{matrix}
\label{BosFermmaps}
\end{equation}
where
\begin{equation}
a_\rho = \prod_{1\le i<j\le n} (x_j-x_i)
\qquad\hbox{and}\qquad
A_\rho = \prod_{1\le i<j\le n} (x_j-tx_i),
\label{arhoArhodefn}
\end{equation}
\end{thm}

\subsection{The Poincar\'e polynomial $W_0(t)$}

The \emph{Poincar\'e polynomial} for the symmetric group $S_n$ is 
$$W_0(t) = \sum_{w\in S_n} t^{\ell(w)}.$$

The following result writes the Poincar\'e polynomial in terms of the $c$-functions
$c_{w_0}^Y$ and
$c_{w_0}^{X^{-1}}$.

\begin{prop} \label{Poinbysymmprop}
We have that
\begin{equation}
t^{-\frac12\ell(w_0)} W_0(t) =  \ev^t_0(c_{w_0}^Y) = \sum_{w\in S_n} wc_{w_0}^{X^{-1}}
\label{Poinbysymm}
\end{equation}
and
$$W_0(t) = \prod_{1\le i<j\le n} \frac{ 1-t^{j-i+1}}{1-t^{j-i}}
= [n]!,
\qquad\hbox{where}\qquad
[n]! = \prod_{d=1}^n \frac{1-t^d}{1-t}.
$$
\end{prop}

By Proposition~\ref{Poinbysymmprop}, 
$$W_0(t) = \sum_{w\in S_n} w t^{\frac12\ell(w_0)}c_{w_0}^{X^{-1}}
= \sum_{w\in S_n} w \Big(\prod_{i<j} \frac{x_i-tx_j}{x_i-x_j}\Big).$$
This coincides with the definition of $W_0(t)$ (see~\eqref{Plambdadefn}) as the 
appropriate constant that makes the coefficient of  $1= x^0 = x_1^0\cdots x_n^0$ in 
$$P_0 = \frac{1}{W_0(t)} \sum_{w\in S_n}  w\Big(\prod_{i<j} \frac{x_i-tx_j}{x_i-x_j}\Big).
$$

The following result is the analog for the subgroup $W_\lambda$, and its proof is the same as for 
Proposition~\ref{Poinbysymmprop}, 
except restricted to the subgroup $W_\lambda$.  It provides explicit formulas for the
normalizing factor $W_\lambda(t)$ which makes the coefficient of $x^\lambda$ in the
bosonic Macdonald polynomial $P_\lambda(x_1, \ldots, x_n;q,t)$ equal to $1$.

\begin{prop} \label{paraPoin} 
Let $\lambda \in \left(\ZZ^n\right)^+$. Let 
$W_\lambda = \{ v\in S_n\ |\ v\lambda = \lambda\}$ be the stabilizer of $\lambda$ in $S_n$ and 
$W_\lambda(t)$ the length generating
function for $W_\lambda$.  Then
$$t^{-\frac12\ell(w_\lambda)} W_\lambda(t)  = \ev^t_0(c_{w_\lambda}^Y)
= \sum_{w\in W_\lambda} w c_{w_\lambda}^{X^{-1}},$$
where $w_\lambda$ is the longest element of $W_\lambda$.  Alternatively,
$$W_\lambda(t) 
= \sum_{w\in W_\lambda} w\Big(\prod_{1\le i<j\le n\atop \lambda_i = \lambda_j} \frac{1-tx^{-1}_ix_j}{1-x^{-1}_ix_j}\Big)
= \prod_{1\le i<j\le n\atop \lambda_i=\lambda_j} \frac{1-t^{j-i+1}}{1-t^{j-i}}
= \prod_{i\in \ZZ} [m_i]!,
$$
where $m_i = \#\{ k\in \{1, \ldots, n\}\ |\ \lambda_k = i\}$.
\end{prop}

\begin{example}
For $\lambda = (5,5,3,2,2,2,2,2, -1, -1, -1)$, we have that  $m_5 = 2$, $m_3=1$, $m_2 = 5$,
$m_{-1} = 3$. By Proposition~\ref{paraPoin}, 
$$W_\lambda = S_5 \times S_1\times S_5 \times S_3 \subseteq S_{11}
\qquad\hbox{and}\qquad W_\lambda(t) = [2]!\cdot [1]! \cdot [5]! \cdot [3]!.
$$
\end{example}

\subsection{$H_Y$-decomposition of the polynomial representation}

Let $H_Y$ be the algebra generated by the operators $T_1, \ldots, T_{n-1}$
and $Y_1, \ldots, Y_n$ (so that $H_Y$ is an affine Hecke algebra).
As $H_Y$-modules
$$\CC[X] = \bigoplus_{\lambda} \CC[X]^\lambda
\qquad\hbox{where}\qquad
\CC[X]^\lambda 
= \hbox{span}\{ E_\mu\ |\ \mu\in S_n\lambda\},
$$
where the direct sum is over decreasing $\lambda = (\lambda_1\ge \cdots \ge \lambda_n)\in \ZZ^n$, and
$S_n\lambda$ denotes the set of distinct rearrangements of $\lambda$.

A description of the action of $H_Y$ on $\CC[X]^\lambda$ is given by the following.
Let $\mu\in \ZZ^n$ and $v_\mu\in S_n$ be the minimal length permutation such that $v_\mu\mu$ is weakly increasing. Fix $i\in \{1, \ldots, n-1\}$, and define
\begin{equation*}
\begin{array}{l}
a_\mu =q^{\mu_i-\mu_{i+1}}t^{v_\mu(i)-v_\mu(i+1)} , \\
a_{s_i\mu} =q^{\mu_{i+1}-\mu_i} t^{v_\mu(i+1)-v_\mu(i)} ,
\end{array}
\qquad\hbox{and}\qquad
D_\mu = \frac{(1-ta_\mu)(1-ta_{s_i\mu})}{(1-a_\mu)(1-a_{s_i\mu})}.
\end{equation*}
We have two cases depending on the relationship between $\mu_i$ and $\mu_{i+1}$.
\begin{itemize}
    \item[] \underline{Case 1:} $\mu_i>\mu_{i+1}$. Using the identity $E_{s_i\mu} = t^{\frac12}\tau^\vee_iE_\mu$ if $\mu_i > \mu_{i+1}$ from (E2), the eigenvalue  from Theorem~\ref{Eeigenvalue} and the formulas in~\eqref{interwinersrels} for $\tau_i^\vee$,
\begin{equation}
\begin{array}{l}
Y_i^{-1}Y_{i+1} E_\mu = a_\mu E_\mu, \\
Y_i^{-1}Y_{i+1} E_{s_i\mu} = a_{s_i\mu}E_{s_i\mu},
\end{array}
\quad
\begin{array}{l}
t^{\frac12}\tau^\vee_i E_\mu = E_{s_i\mu}, \\
t^{\frac12}\tau^\vee_i E_{s_i\mu} = D_\mu E_\mu,
\end{array}
\quad\hbox{and}
\quad
\begin{array}{l}
t^{\frac12}T_i E_\mu = -\frac{1-t}{1-a_\mu}E_\mu + E_{s_i\mu}, \\
t^{\frac12}T_i E_{s_i\mu} = D_\mu E_\mu + \frac{1-t}{1-a_{s_i\mu}} E_{s_i\mu}.
\end{array}
\label{CXlambdaaction}
\end{equation}

    \item[] \underline{Case 2:} If $\mu_i=\mu_{i+1}$ then
$v_\mu(i+1) = v_\mu(i)+1$ and $a_\mu = t^{-1}$, so that
\begin{equation}
Y_i^{-1}Y_{i+1} E_\mu = t^{-1} E_\mu, 
\qquad (t^{\frac12}\tau^\vee_i) E_\mu = 0,
\qquad\hbox{and}\qquad (t^{\frac12}T_i) E_\mu = t E_\mu.
\label{Tigivest}
\end{equation}
\end{itemize}
These formulas make explicit the action of $H_Y$ on $\CC[X]^\lambda$
in the basis $\{ E_\mu\ |\ \mu\in S_n \lambda\}$.

\subsection{E-expansions}

Let $\CC(Y)$ be the field of fractions of $\CC[Y]$.  
For $\mu = (\mu_1, \ldots, \mu_n)\in \ZZ^n$ define homomorphisms
$\ev^t_\mu\colon \CC[Y] \to \CC[Y]$ by
\begin{equation*}
\ev^t_\mu(Y_i) = q^{-\mu_i}t^{-(v_\mu(i)-1)+\frac12(n-1)},
\qquad\hbox{for $i\in \{1, \ldots, n\}$.}
\end{equation*}
Extend $\ev^t_\mu$ to those elements of the field $\CC(Y)$ for which the evaluated denominator
is not $0$.  
We refer to these homomorphisms as \emph{evaluations}. 
Then Theorem~\ref{Eeigenvalue} gives that
\begin{equation}
f E_\mu = \ev^t_\mu(f)E_\mu,\qquad \text{ for } f\in \CC[Y] \text{ and }\mu\in \ZZ^n.
\label{EeigenvalueB} 
\end{equation}

The following result expressed the bosonic and fermionic Macdonald polynomials in terms of the electronic Macdonald polynomials using the evaluations and $c$-functions. 
\begin{prop}  \label{Eexpansion}
Let $\lambda \in \left(\ZZ^n\right)^+$. Then
\begin{align*}
P_\lambda 
&=  \sum_{z\in W^\lambda} t^{\frac12\ell(v_\lambda z)} 
\ev^t_{z\lambda}(c^Y_{v_\lambda z}) E_{z\lambda}
\qquad\hbox{and}
\\
A_{\lambda+\rho} 
&= \sum_{z\in W_0} (-t^{\frac12})^{\ell(w_0z)}
\ev^t_{z(\lambda+\rho)}(c^{Y^{-1}}_{w_0z})
E_{z(\lambda+\rho)}.
\end{align*}
Alternatively, 
\begin{align*}
P_\lambda 
&= \sum_{\mu\in S_n\lambda} 
\Big(\prod_{1\le i<j\le n\atop \mu_i > \mu_j} t
\Big(\frac{1-q^{\mu_i-\mu_j}t^{v_\mu(i)-v_\mu(j)-1}}{1-q^{\mu_i-\mu_j}t^{v_\mu(i)-v_\mu(j)}}\Big)\Big) E_\mu
\qquad\hbox{and}
\\
A_{\lambda+\rho} 
&= \sum_{\mu\in S_n(\lambda+\rho)} 
\Big(\prod_{1\le i<j\le n\atop \mu_i > \mu_j} 
(-1)\Big(\frac{1-q^{\mu_i-\mu_j}t^{v_\mu(i)-v_\mu(j)+1}}{1-q^{\mu_i-\mu_j}t^{v_\mu(i)-v_\mu(j)}}\Big) \Big) E_\mu.
\end{align*}
\end{prop}

\begin{example}
For $n=2$ and $m\in \ZZ_{>0}$,
\begin{align*}
P_{(m,0)}
&= E_{(0,m)} + t^{\frac12} \ev^t_{(m,0)}(c^Y_{12}) E_{(m,0)} 
= E_{(0,m)} + t\Big(\frac{1-q^m }{1-q^m t}\Big) E_{(m,0)}, \\
A_{(m,0)} 
&= E_{(0,m)} - t^{\frac12}\ev^t_{(m,0)}(c^Y_{21}) E_{(m,0)} 
= E_{(0,m)} - \frac{1-q^m t^2}{1-q^m t} E_{(m,0)}.
\end{align*}
Note that these expressions relate to $c$-functions since 
$t\Big(\frac{1-q^m }{1-q^m t}\Big) 
= t^{\frac12}\Big(\frac{t^{-\frac12}-t^{\frac12}q^{-m}t^{-1} }{1-q^{-m}t^{-1}}\Big)$.

For $n=3$, 
\begin{equation*}
P_{(1,0,0)} = E_{(0,0,1)} + t\Big(\frac{1-q}{1-qt}\Big)E_{(0,1,0)}
+ t^2 \Big(\frac{1-q}{1-qt}\Big)\Big(\frac{1-qt}{1-qt^2}\Big)E_{(1,0,0)}.
\end{equation*}
For general $n$ and $m\in \ZZ_{>0}$, denote $\varepsilon_i = (0, \ldots, 0, 1,0, \ldots, 0)$, that is the sequence of 
length $n$ with $1$ in the $i$th spot and $0$ elsewhere. Then
\begin{align*}
P_{(m,0,\ldots, 0)} 
&= \sum_{i=1}^n 
t^{n-i} \Big(\frac{1-q^m}{1-q^mt}\Big)
\Big(\frac{1-q^m t}{1-q^m t^2}\Big)
\cdots \Big( \frac{1-q^mt^{n-i-1}}{1-q^mt^{n-i}}\Big)
E_{m\varepsilon_i}
= \sum_{i=1}^n 
t^{n-i} \Big(\frac{1-q^m}{1-q^mt^{n-i}}\Big) E_{m\varepsilon_i}.
\end{align*}
\end{example}

\begin{example}
\begin{align*}
    P_{(2,1,0)} &= E_{(0,1,2)} 
+ t\Big(\frac{1-q}{1-qt}\Big)E_{(1,0,2)}+ t\Big(\frac{1-q}{1-qt}\Big)E_{(0,2,1)}
+ t^2\Big(\frac{1-qt}{1-qt^2}\Big)\Big(\frac{1-q^2}{1-q^2t}\Big)E_{(2,0,1)} \\
&\qquad+ t^2\Big(\frac{1-qt}{1-qt^2}\Big)\Big(\frac{1-q^2}{1-q^2t}\Big)E_{(1,2,0)}
+ t^3\Big(\frac{1-q}{1-qt}\Big)\Big(\frac{1-q^2t}{1-q^2t^2}\Big)\Big(\frac{1-q}{1-qt}\Big)E_{(2,1,0)}.
\end{align*}
\end{example}

\subsection{Proofs}

\subsubsection{Proof that the action on polynomials is a representation of $\widetilde{H}$ }

\begin{dupThm}[~\ref{PolyrepThm}]
    Let $\widetilde{H}$ be the double affine Hecke algebra
as defined by generators and relations in Section~\ref{subsec:DAHA}.
The formulas~\eqref{divdiffops} define an action of $\widetilde{H}$
on $\CC[X]$.
\end{dupThm}

\begin{proof} 
Let $i\in \{1, \ldots, n-1\}$.
Since $y_n X_i = X_i y_n$ and $s_1\cdots s_{n-1} X_i = X_{i+1} s_1\cdots s_{n-1}$ 
then
$$T_\pi X_i = s_1\cdots s_{n-1}y_n X_i = X_{i+1} s_1\cdots s_{n-1}y_n = X_{i+1}T_\pi.$$
As operators on $\CC[X]$,
$$T_\pi X_n = s_1\cdots s_{n-1} y_n X_n = s_1\cdots s_{n-1}q^{-1}X_n y_n = q^{-1}X_1s_1\cdots s_{n-1} y_n
= q^{-1}X_1 T_\pi = X_{n+1} T_\pi.$$
Recall that we define $s_0 = T_\pi s_{n-1}T_\pi^{-1}$ so that
\begin{align*}
s_0 
& = T_\pi s_{n-1}T_\pi^{-1}
= (y_1s_1\cdots s_{n-1})s_{n-1}(s_{n-1}\cdots s_1y^{-1}_1)
=y_1 s_1\cdots s_{n-2}s_{n-1}s_{n-2}\cdots s_1 y^{-1}_1 
\\ &= y_1y^{-1}_n s_1\cdots s_{n-2}s_{n-1}s_{n-2}\cdots s_1= y_1y^{-1}_n s_{n-1}\cdots s_2s_1s_2 \cdots s_{n-1}.
\end{align*}
Then 
\begin{align*}
T_\pi s_0T_\pi^{-1} = (s_1\cdots s_{n-1}y_n)
(y_1y^{-1}_n s_{n-1}\cdots s_2s_1s_2 \cdots s_{n-1})
(s_{n-1}\cdots s_1y^{-1}_1)
=s_1y_2 y_1 y_2^{-1} y_1^{-1} = s_1.
\end{align*}
Define $\partial_0 = T_\pi \partial_{n-1} T_\pi^{-1}$, so that
\begin{align*}
\partial_0 = T_\pi \partial_{n-1} T_\pi^{-1} = T_\pi\frac{1}{x_{n-1}-x_n}(1-s_{n-1})T_\pi^{-1}
=\frac{1}{x_n-q^{-1}x_1}(1-s_0).
\end{align*}
Then
\begin{align*}
T_\pi \partial_0 T^{-1}_\pi = T_\pi \frac{1}{x_n-q^{-1}x_1}(1-s_0)T^{-1}_\pi
=\frac{1}{q^{-1}x_1-q^{-1}x_2}(1-s_1) = q\partial_1.
\end{align*}
Finally, 
\begin{align*}
T_0 
&= T_\pi T_{n-1}T^{-1}_\pi 
= T_\pi (t^{-\frac12}x_n\partial_{n-1}-t^{\frac12}\partial_{n-1}x_n)T^{-1}_\pi
= t^{-\frac12}q^{-1}x_1\partial_0-t^{\frac12}\partial_0 q^{-1}x_n,
\end{align*}
and
\begin{align*}
T_\pi T_0T^{-1}_\pi = T_\pi (t^{-\frac12}q^{-1}x_1\partial_0-t^{\frac12}\partial_0 q^{-1}x_1)T^{-1}_\pi
= t^{-\frac12}q^{-1} x_2q\partial_1 -t^{\frac12}q\partial_1 q^{-1}x_2
= T_1.
\end{align*}
These computations show that the relations in~\eqref{DAHArels2F} hold. 

Let $i\in \{1, \ldots, n-1\}$.  Then, as operators on $\CC[X]$,
\begin{align*}
\xi_{s_i} 
&= \frac{1}{c^X_{i,i+1}}(C_{s_i}-c^X_{i+1,i})
= \frac{1}{c^X_{i,i+1}}(T_i+t^{-\frac12}-c^X_{i+1,i})
= \frac{1}{c^X_{i,i+1}}(t^{-\frac12}x_{i+1}\partial_i-t^{\frac12}\partial_i x_{i+1}+t^{-\frac12}-c^X_{i+1,i})
\\
&= \frac{1}{c^X_{i,i+1}}\Big(
t^{-\frac12}\frac{x_{i+1}}{x_i-x_{i+1}}(1-s_i)-t^{\frac12}\frac{1}{x_i-x_{i+1}}(1-s_i)x_{i+1}
+\frac{t^{-\frac12}x_i-t^{-\frac12}x_{i+1}}{x_i-x_{i+1}} - \frac{t^{-\frac12}x_i-t^{\frac12}x_{i+1}}{x_i-x_{i+1}}\Big)
\\
&= \frac{1}{c^X_{i,i+1}}\Big( \frac{t^{-\frac12}x_{i+1}-t^{\frac12}x_{i+1}
+t^{-\frac12}x_i-t^{-\frac12}x_{i+1} -t^{-\frac12}x_i+t^{\frac12}x_{i+1}}{x_i-x_{i+1}}
+\frac{ - t^{-\frac12}x_{i+1}+t^{\frac12}x_i}{x_i-x_{i+1}}s_i\Big)
\\
&= \frac{1}{c^X_{i,i+1}}\Big( 0+\frac{t^{-\frac12}-t^{\frac12}x_i x^{-1}_{i+1}}{1-x_ix^{-1}_{i+1}} s_i\Big)
=s_i.
\end{align*}
Then, as in the proof of Proposition~\ref{normintwn}, the relations
$$s_i^2 = 1, \qquad s_is_{i+1}s_i = s_{i+1}s_i s_{i+1}, \qquad s_j s_k = s_k s_j,$$
for $i,j,k\in \{1, \ldots, n-1\}$ with $k\not\in \{j-1, j+1\}$,
imply the relations in~\eqref{HeckerelsF}.
Similarly, as in the proof of Proposition~\ref{normintwn}, the relations
$$s_i X_i = X_{i+1}s_i, \qquad s_i X_{i+1} = x_i s_i, \qquad s_i X_j = X_j s_i,$$
for $i\in \{1, \ldots, n-1\}$ with $j\not\in \{i,i+1\}$, imply the relations in~\eqref{TpastXF},
which, in the presence of the last relation in~\eqref{HeckerelsF}, are equivalent to the relations in~\eqref{XaffHeckerelsF}.
\end{proof}

\subsubsection{Proof of the Boson-Fermion equalities}
The following result is a restatement of Theorem~\ref{BosFerCorr}.
\begin{prop} \label{BFequalitiesP}  
$$
\begin{array}{llll}
p_0 \CC[X] = \CC[X]^{S_n} = \CC[X]^{W_0},
&e_0 \CC[X] = \CC[X]^{\det} = a_\rho \CC[X]^{W_0}
&\hbox{and}
&a_\rho = e_0x^\rho,
\\
\\
\mathbf{1}_0 \CC[X] = \CC[X]^{\mathrm{Bos}} = \CC[X]^{W_0},\quad
&
\varepsilon_0 \CC[X] = \CC[X]^{\mathrm{Fer}} = A_\rho \CC[X]^{W_0}
&\hbox{and}
&A_\rho = \varepsilon_0 x^\rho.
\end{array}
$$
\end{prop}
\begin{proof} \hspace{0.3cm}
\begin{itemize}[leftmargin=0.5cm]
\item[$\circ$] $p_0\CC[X] = \CC[X]^{S_n}$ 
\begin{itemize}[leftmargin=0.5cm]
    \item If $f\in \CC[X]^{S_n}$ then $f = p_0\Big(\dfrac{1}{n!}   f\Big)$ so that $f\in p_0\CC[X]$. Thus $\CC[X]^{S_n} \subseteq p_0\CC[X]$.

    \item If $f\in p_0\CC[X]$ then $f = p_0g$. We also know that if $w\in S_n$ then
$wf = wp_0g = p_0g$, and so $f\in \CC[X]^{S_n}$. Thus $p_0\CC[X] \subseteq \CC[X]^{S_n}$.
\end{itemize}

\item[$\circ$] $a_\rho \CC[X]^{W_0} = \CC[X]^{\det}$
\begin{itemize}[leftmargin=0.5cm]
    \item 
If $f\in \CC[X]^{\det}$ then $(1-s_{ij})f = 0$, where $s_{ij}$ denotes the transposition in $S_n$ that switches $i$ and $j$.  Thus $f$ is divisible by each $x_j-x_i$, which implies that  
$$\hbox{$f$ is divisible by}\quad a_\rho = \prod_{1\le i<j\le n} x_j-x_i.$$
Then $\dfrac{1}{a_\rho} f\in \CC[X]^{S_n}$ and $f\in a_\rho\CC[X]^{W_0}$, so that $\CC[X]^{\det} \subseteq a_\rho\CC[X]^{S_n}$.
\end{itemize}

\item[$\circ$] $e_0\CC[X] = \CC[X]^{\det}$
\begin{itemize}[leftmargin=0.5cm]
    \item  If $f\in e_0\CC[X]$ then $s_\alpha f = s_\alpha e_0 g = -e_0g = -f$.  Thus $f\in \CC[X]^{\det}$ and $e_0\CC[X] \subseteq \CC[X]^{\det}$.

    \item If $f\in \CC[X]^{\det}$ then $e_0 f = \Card(W_0)f$. Thus  $f\in e_0\CC[X]$ and $\CC[X]^{\det} \subseteq e_0\CC[X]$.
\end{itemize}

\item[$\circ$] $\mathbf{1}_0 \CC[X] = \CC[X]^{\mathrm{Bos}} = \CC[X]^{S_n}$
\begin{itemize}[leftmargin=0.5cm]
    \item Let $h\in \mathbf{1_0}\CC[X]$, we can write $h = \mathbf{1}_0f$ with $f\in \CC[X]$ and then
$$T_{s_i}h = T_{s_i}\mathbf{1}_0 f 
=  t^{\frac12}\mathbf{1}_0 f 
= t^{\frac12}h.
\qquad\quad
\hbox{Thus $h\in \CC[X]^{\mathrm{Bos}}$ and $\mathbf{1}_0\CC[X] \subseteq \CC[X]^{\mathrm{Bos}}$.}
$$

\item 
Let $f\in \CC[X]^{\mathrm{Bos}}$.  Then, by Proposition~\ref{propslicksymmA} and
~\eqref{Poinbysymm}, 
$$f = \frac{t^{\frac12\ell(w_0}}{W_0(t)} \mathbf{1}_0 f 
= \frac{1}{[n]!} \sum_{w\in W_0} w \Big(f\ \prod_{1\le i<j\le n} \frac{x_i-tx_j}{x_i-x_j}\Big)
\quad \in \CC[X]^{W_0}.
$$
Thus $\CC[X]^{\mathrm{Bos}} \subseteq \CC[X]^{W_0}$.

\item 
Let $f\in \CC[X]^{W_0}$.  Then,  by Proposition~\ref{propslicksymmA} and
~\eqref{Poinbysymm},
$$
\mathbf{1}_0 \frac{t^{\frac12\ell(w_0}}{W_0(t)} f
=\frac{1}{[n]!} \sum_{w\in W_0} w \Big(f\ \prod_{1\le i<j\le n} \frac{x_i-tx_j}{x_i-x_j}\Big)
= f \frac{1}{[n]!} \sum_{w\in W_0} w \Big(\prod_{1\le i<j\le n} \frac{x_i-tx_j}{x_i-x_j}\Big)=f.
$$
Thus $f\in \mathbf{1}_0\CC[X]$ and $\CC[X]^{W_0} \subseteq \mathbf{1}_0\CC[X]$. 
\end{itemize}

\item[$\circ$]  $\varepsilon_0 \CC[X] = \CC[X]^{\mathrm{Fer}} = A_\rho \CC[X]^{W_0}$
    
    \begin{itemize}[leftmargin=0.5cm]
    \item 
Let $h = \varepsilon_0 \CC[X]$ and let $f\in \CC[X]$ such that $h=\varepsilon_0 f$.  Then
$$T_{s_i}h = T_{s_i}\varepsilon_0 f 
=  -t^{-\frac12}\varepsilon_0 f 
= -t^{-\frac12}h.
$$
Thus $h\in \CC[X]^{\mathrm{Fer}}$ and $\varepsilon_0\CC[X] \subseteq \CC[X]^{\mathrm{Fer}}$.

\item  
Let $f\in \CC[X]^{\mathrm{Fer}}$.
Then $T_i f = -t^{-\frac12}f$ gives
$$f = \frac{t^{\frac12\ell(w_0)}} {W_0(t)} \varepsilon_0 f
= \frac{1}{W_0(t)} \frac{A_\rho}{a_\rho} \sum_{w\in W_0} \det(w) wf
\quad \in A_\rho \CC[X]^{W_0}.
$$
Thus $\CC[X]^{\mathrm{Fer}}\subseteq A_\rho \CC[X]^{W_0}$.

\item  
Let $f \in A_\rho \CC[X]^{W_0}$ and 
let $g\in \CC[X]^{S_n}$ be such that $f = A_\rho g$. Write $g$ as a linear combination, $g = \sum c_\lambda s_\lambda$, where
$s_\lambda$ are Schur polynomials.
Then
$$f = A_\rho g =  \sum_\lambda c_\lambda A_\rho s_\lambda
= \sum_{\lambda} c_\lambda \frac{A_\rho}{a_\rho} 
\sum_{w\in W_0} \det(w_0w) wx^{\lambda+\rho}
= \varepsilon_0 \Big(\sum_\lambda c_\lambda x^{\lambda+\rho}\Big).
$$
By~\eqref{slicksymmA} and the fact that 
$$t^{-\frac12\ell(w_0)}\frac{A_\rho}{a_\rho} 
= t^{-\frac12\ell(w_0)} \prod_{1\le i<j\le n} \frac{x_j - tx_i}{x_j-x_i}
= \prod_{1\le i<j\le n} \frac{t^{-\frac12} - t^{\frac12}x_ix^{-1}_j}{1-x_ix^{-1}_j}
=c_{w_0}(x)
$$
conclude that $f\in \varepsilon_0 \CC[X]$.  Thus $A_\rho \CC[X]^{W_0} \subseteq \varepsilon_0 \CC[X]$. 

\end{itemize}

\end{itemize}

\end{proof}

\subsubsection{Proof of the formulas for Poincar\'e polynomials}

\begin{dupProp}[~\ref{Poinbysymmprop}]
We have that
\begin{equation*}
t^{-\frac12\ell(w_0)} W_0(t) =  \ev^t_0(c_{w_0}^Y) = \sum_{w\in S_n} wc_{w_0}^{X^{-1}}
\end{equation*}
and
$$W_0(t) = \prod_{1\le i<j\le n} \frac{ 1-t^{j-i+1}}{1-t^{j-i}}
= [n]!,
\qquad\hbox{where}\qquad
[n]! = \prod_{d=1}^n \frac{1-t^d}{1-t}.
$$
\end{dupProp}
\begin{proof}  Identify $\CC[X]$ and $\widetilde{H}\mathbf{1}_Y$ via the isomorphism in 
~\eqref{CXasIndHY}.
Applying the identity $\mathbf{1}_0 = p_0^Xc_{w_0}^{X^{-1}}$ from Proposition~\ref{slicksymmA} 
to the polynomial $1=x^0 = X^0\mathbf{1}_Y = \mathbf{1}_Y$,
\begin{align*}
t^{-\frac12\ell(w_0)} &\sum_{w\in S_n}  w\Big(\prod_{i<j} \frac{x_i-tx_j}{x_i-x_j}\Big) 
= \Big(\sum_{w\in S_n}  w\Big) c_{w_0}^{X^{-1}} \mathbf{1}_Y 
= p_0^Xc^{X^{-1}}_{w_0}\mathbf{1}_Y = \mathbf{1}_0\mathbf{1}_Y \quad\hbox{and} \\
\mathbf{1}_0\mathbf{1}_Y &= \sum_{w\in S_n} (t^{\frac12})^{\ell(w)-\ell(w_0)}T_w \mathbf{1}_Y
= t^{-\frac12\ell(w_0)} \Big(\sum_{w\in S_n} t^{\ell(w)} \Big)\mathbf{1}_Y
= t^{-\frac12\ell(w_0)} W_0(t)\mathbf{1}_Y.
\end{align*}
Applying the identity $\mathbf{1}_0 = p_0^Yc^Y_{w_0}$ from 
Proposition~\ref{slicksymmA} to $\mathbf{1}_Y$ and using that $\eta_w \mathbf{1}_Y=0$ if
$w\in S_n$ and $w\ne 1$ (see the lest identity in~\eqref{nintbdrels}), gives
\begin{align*}
\mathbf{1}_0\mathbf{1}_Y &= p_0^Y c_{w_0}^Y
=\Big(\sum_{w\in W_0}  \eta_w\Big) c_{w_0}^Y \mathbf{1}_Y 
= \ev^t_0\Big(c_{w_0}^Y\Big) \Big(1+\sum_{w\in W_0, w\ne 1}  \eta_w\Big) \mathbf{1}_Y  \\
&= \ev^t_0\Big(c_{w_0}^Y\Big) \Big(1+0 \Big) \mathbf{1}_Y 
= \ev^t_0\Big(c_{w_0}^Y\Big) \mathbf{1}_Y
= \ev^t_0\Big(\prod_{i<j} \frac{t^{-\frac12}-t^{\frac12}Y_iY^{-1}_j}{1-Y_iY^{-1}_j}\Big)\mathbf{1}_Y
\\
&= \Big(\prod_{i<j} \frac{t^{-\frac12}-t^{\frac12}Y_iY^{-1}_j}{1-Y_iY^{-1}_j}\Big)\mathbf{1}_Y
= t^{-\frac12\ell(w_0)} \Big(\prod_{i<j} \frac{ 1-t^{j-i+1}}{1-t^{j-i}}\Big)\mathbf{1}_Y.
\end{align*}
Finally,
\begin{align*}
\prod_{1\le i<j\le n} &\frac{ 1-t^{j-i+1}}{1-t^{j-i}}
=\prod_{d=1}^{n-1} \prod_{i<j\atop j-i = d} \frac{1- t^{d+1}}{1-t^d}
=\prod_{d=1}^{n-1} \Big(\frac{1- t^{d+1}}{1-t^d}\Big)^{n-d} 
\\
&= \Big(\frac{1- t^n}{1-t^{n-1}}\Big)
\Big(\frac{1- t^{n-1}}{1-t^{n-2}}\Big)^2
\cdots \Big(\frac{1- t^2}{1-t}\Big)^{n-1}
= \frac{(1-t^n)(1-t^{n-1})\cdots (1-t^2)}{(1-t)^{n-1}} = [n]!.
\end{align*}
\end{proof}

\begin{dupProp}[~\ref{paraPoin}]
Let $\lambda \in \left(\ZZ^n\right)^+$. Let 
$W_\lambda = \{ v\in S_n\ |\ v\lambda = \lambda\}$ be the stabilizer of $\lambda$ in $S_n$ and 
$W_\lambda(t)$ is the length generating
function for $W_\lambda$.  Then
$$t^{-\frac12\ell(w_\lambda)} W_\lambda(t)  = \ev^t_0(c_{w_\lambda}^Y)
= \sum_{w\in W_\lambda} w c_{w_\lambda}^{X^{-1}},$$
where $w_\lambda$ is the longest element of $W_\lambda$.  Alternatively,
$$W_\lambda(t) 
= \sum_{w\in W_\lambda} w\Big(\prod_{1\le i<j\le n\atop \lambda_i = \lambda_j} \frac{1-tx^{-1}_ix_j}{1-x^{-1}_ix_j}\Big)
= \prod_{1\le i<j\le n\atop \lambda_i=\lambda_j} \frac{1-t^{j-i+1}}{1-t^{j-i}}
= \prod_{i\in \ZZ} [m_i]!,
$$
where $m_i = \#\{ k\in \{1, \ldots, n\}\ |\ \lambda_k = i\}$.
\end{dupProp}
\begin{proof}
Identify $\CC[X]$ and $\widetilde{H}\mathbf{1}_Y$ via the isomorphism in 
~\eqref{CXasIndHY}.
By the definition of $\mathbf{1}_\lambda$ and $\mathbf{1}_Y$ from~\eqref{pHs} and~\eqref{HYtriv},
\begin{align*}
\mathbf{1}_\lambda \mathbf{1}_Y 
&= \sum_{w\in W_\lambda} (t^{\frac12})^{\ell(w)-\ell(w_\lambda)} T_w
\mathbf{1}_Y
= t^{-\frac12\ell(w_\lambda)}\Big(\sum_{w\in W_\lambda} t^{\ell(w)}\Big)\mathbf{1}_Y
= t^{-\frac12\ell(w_\lambda)}W_\lambda(t)\mathbf{1}_Y.
\end{align*}
Let $v_\lambda = w_0w^{-1}_\lambda$ so that $v_\lambda\in S_n$ is minimal length 
such that $v_\lambda \lambda$ is weakly increasing.
Applying the identities $\mathbf{1}_0 = p_0^Xc_{w_0}^{X^{-1}}$ and
$\mathbf{1}_0 =  p_\lambda^X c_{v_\lambda}^{X^{-1}} \mathbf{1}_\lambda$ from~\eqref{propsymwparabA},
and using the facts that
$\ell(v_\lambda)+\ell(w_\lambda)=\ell(w_0)$ and
the coefficient of $1$ in $c_{v_\lambda}^{X^{-1}}$ is $t^{-\frac12\ell(v_\lambda)}$, then
the expression
\begin{align*}
\Big(\sum_{w\in S_n} w\Big) c_{w_0}^{X^{-1}}E_\lambda
&=\mathbf{1}_0 E_\lambda \mathbf{1}_Y 
=p^\lambda_X c_{v_\lambda}^{X^{-1}} \mathbf{1}_\lambda E_\lambda \mathbf{1}_Y 
=p^\lambda_X c_{v_\lambda}^{X^{-1}} E_\lambda \mathbf{1}_\lambda  \mathbf{1}_Y 
=t^{-\frac12\ell(w_\lambda)} W_\lambda(t)
p^\lambda_X c_{v_\lambda}^{X^{-1}} E_\lambda\mathbf{1}_Y
\end{align*}
has coefficient of $x^\lambda$ equal to $t^{-\frac12\ell(w_0)}W_\lambda(t)$.
Since $\lambda_1\ge \cdots \ge \lambda_n$ and 
the coefficient of $x^\lambda = x_1^{\lambda_1}\cdots x_n^{\lambda_n}$ in $P_\lambda$ is $1$ then
$$P_\lambda = \frac{t^{\frac12\ell(w_0)}}{W_\lambda(t)}
\Big(\sum_{w\in S_n} w\Big) c_{w_0}^{X^{-1}}E_\lambda
=\frac{1}{W_\lambda(t)} \sum_{w\in S_n} w\Big( E_\lambda \prod_{i<j}\frac{x_i-tx_j}{x_i-x_j}\Big).
$$
Applying the identity $\mathbf{1}_\lambda = p_\lambda^Y c_{w_\lambda}^Y$ from Proposition~\eqref{propsymwparabA}
to $\mathbf{1}_Y$ and using that 
$\eta_w \mathbf{1}_Y=0$ if $w\in S_n$ and $w\ne 1$ (see the last identity in~\eqref{nintbdrels}), gives
\begin{align*}
\mathbf{1}_\lambda \mathbf{1}_Y 
&= p_\lambda^Y c_{w_\lambda}^Y = \Big( \sum_{w\in W_\lambda} \eta_w\Big)c_{w_\lambda}^Y \mathbf{1}_Y
= \ev_0^t\big(c_{w_\lambda}^Y\big) \Big( 1 + \sum_{w\in W_\lambda, w\ne 1} \eta_w\Big) \mathbf{1}_Y 
= \ev_0^t\big(c_{w_\lambda}^Y\big) (1 + 0)\mathbf{1}_Y.
\end{align*}
so that $t^{-\frac12\ell(w_\lambda)}W_\lambda(t) = \ev_0^t\big(c_{w_\lambda}^Y\big)$.
Explicitly,
\begin{align*}
\ev_0^t\big(c_{w_\lambda}^Y\big)
&= \ev_0^t\Big( \prod_{i<j\atop \lambda_i=\lambda_j} \frac{t^{-\frac12}-t^{\frac12}Y_iY^{-1}_j}{1-Y_iY^{-1}_j}\Big)
\mathbf{1}_Y 
=t^{-\frac12\ell(w_\lambda)} \Big(\prod_{i<j\atop \lambda_i=\lambda_j} \frac{1-t^{j-i+1}}{1-t^{j-i}}\Big)
\mathbf{1}_Y.
\end{align*}
Finally, letting $m_i = \#\{ k\in \{1, \ldots, n\}\ |\ \lambda_k = i\}$,
\begin{align*}
\prod_{1\le i<j\le n\atop \lambda_i = \lambda_j} \frac{ 1-t^{j-i+1}}{1-t^{j-i}}
&=\prod_{i\in \ZZ} \prod_{d_i=1}^{m_i-1} \prod_{i<j\atop j-i = d_i} \frac{1- t^{d_i+1}}{1-t^{d_i}}
=\prod_{i\in \ZZ} \prod_{d_i=1}^{m_i-1} \Big(\frac{1- t^{d_i+1}}{1-t^{d_i}}\Big)^{n-d_i} 
\\
&= \prod_{i\in \ZZ} \Big(\frac{1- t^{m_i}}{1-t^{m_i-1}}\Big)
\Big(\frac{1- t^{m_i-1}}{1-t^{m_i-2}}\Big)^2
\cdots \Big(\frac{1- t^2}{1-t}\Big)^{m_i-1}
\\
&= \prod_{i\in \ZZ}  \frac{(1-t^{m_i})(1-t^{m_i-1})\cdots (1-t^2)}{(1-t)^{{m_i}-1}} = \prod_{i\in \ZZ}  [m_i]!.
\end{align*}
\end{proof}

\subsubsection{Proof of the eigenvalue formula}

The following result establishes Theorem~\ref{Eeigenvalue}.
\begin{prop}  Let $\mu = (\mu_1, \ldots, \mu_n)\in \ZZ^n$.  Then
$$Y_i E_\mu = q^{-\mu_i}t^{-(v_\mu(i)-1) + \frac12(n-1)} E_\mu.$$
\end{prop}
\begin{proof}  First do the base case $\mu=0$ when $E_\mu = 1$.  
Using $T_\pi 1 = 1$  and $T_i 1 = t^{\frac12}1$ and $T_i^{-1}1 = t^{-\frac12}1$
gives
$$Y_i 1  = T^{-1}_{i-1}\cdots T^{-1}_{n-1}T_\pi T_{n-1}\cdots T_i \cdot 1
= t^{\frac12(-(i-1)+(n-i-1)}\cdot 1 = t^{-(i-1)+\frac12(n-1)}\cdot1.
$$
Then, by the creation formula for $E_\mu$ in Theorem~\ref{creationformulathm} and the relations~\eqref{taupastYrels1} and~\eqref{taupastYrels2} for moving $Y_i$ past $\tau^\vee_j$,
\begin{align*}
Y_i E_\mu 
&= t^{-\frac12\ell(v_\mu^{-1})} \tau^\vee_{u_\mu} 1
= t^{-\frac12\ell(v_\mu^{-1})} \tau^\vee_{u_\mu} Y_{u^{-1}_\mu(i)} 1
= t^{-\frac12\ell(v_\mu^{-1})} \tau^\vee_{u_\mu} q^{-\mu_i} Y_{v_\mu(i)}) 1 \\
&= t^{-\frac12\ell(v_\mu^{-1})} \tau^\vee_{u_\mu} q^{-\mu_i} t^{-(v_\mu(i)-1)+\frac12(n-1)} 1
= q^{-\mu_i} t^{-(v_\mu(i)-1)+\frac12(n-1)} E_\mu.
\end{align*}
\end{proof}

\subsubsection{Proof of the E-expansions}

\begin{dupProp}[~\ref{Eexpansion}]
Let $\lambda \in \left(\ZZ^n\right)^+$. Then
\begin{align*}
P_\lambda 
&=  \sum_{z\in W^\lambda} t^{\frac12\ell(v_\lambda z)} 
\ev^t_{z\lambda}(c^Y_{v_\lambda z}) E_{z\lambda}
\qquad\hbox{and}
\\
A_{\lambda+\rho} 
&= \sum_{z\in W_0} (-t^{\frac12})^{\ell(w_0z)}
\ev^t_{z(\lambda+\rho)}(c^{Y^{-1}}_{w_0z})
E_{z(\lambda+\rho)}.
\end{align*}
Alternatively,
\begin{align*}
P_\lambda 
&= \sum_{\mu\in S_n\lambda} 
\Big(\prod_{1\le i<j\le n\atop \mu_i > \mu_j} t
\Big(\frac{1-q^{\mu_i-\mu_j}t^{v_\mu(i)-v_\mu(j)-1}}{1-q^{\mu_i-\mu_j}t^{v_\mu(i)-v_\mu(j)}}\Big)\Big) E_\mu
\qquad\hbox{and}
\\
A_{\lambda+\rho} 
&= \sum_{\mu\in S_n(\lambda+\rho)} 
\Big(\prod_{1\le i<j\le n\atop \mu_i > \mu_j} 
(-1)\Big(\frac{1-q^{\mu_i-\mu_j}t^{v_\mu(i)-v_\mu(j)+1}}{1-q^{\mu_i-\mu_j}t^{v_\mu(i)-v_\mu(j)}}\Big) \Big) E_\mu.
\end{align*}
\end{dupProp}

\begin{proof} 
Note that the coefficient of $E_{w_0\lambda}$ in $P_\lambda$ is 1 and the coefficient
of $E_{w_0(\lambda+\rho)}$ in $A_{\lambda+\rho}$ is 1.

For the first statement,~\eqref{slicksymmA},~\eqref{creationPA} and~\eqref{Poinbysymm} give
\begin{align*}
P_\lambda 
&= \frac{t^{\frac12\ell(w_0)}}{W_\lambda(t)} \mathbf{1}_0 E_\lambda
= \frac{t^{\frac12(\ell(w_0)-\ell(w_\lambda))}}{t^{-\frac12\ell(w_\lambda)}W_\lambda(t)} 
\Big(\sum_{z\in W^\lambda} \eta_z\Big)
c_{v_\lambda}^Y \mathbf{1}_\lambda E_\lambda
=t^{\frac12(\ell(w_0)-\ell(w_\lambda))}\sum_{z\in W^\lambda}  \tau^\vee_z \frac{c_{v_\lambda}^Y}{c_z^Y}E_\lambda
\\
&=t^{\frac12\ell(v_\lambda)}\sum_{z\in W^\lambda} \tau^\vee_z (z^{-1}c_{v_\lambda z}^Y) E_\lambda
=t^{\frac12\ell(v_\lambda)}\sum_{z\in W^\lambda} c_{v_\lambda z}^Y \tau^\vee_z E_\lambda
\\
&=t^{\frac12\ell(v_\lambda)}\sum_{z\in W^\lambda} \ev^t_{z\lambda}(c_{v_\lambda z}^Y) t^{-\frac12\ell(z)}
E_{z\lambda}
=\sum_{z\in W^\lambda} t^{\frac12\ell(v_\lambda z)}\ev^t_{z\lambda}(c_{v_\lambda z}^Y)
E_{z\lambda}.
\end{align*}
If $z\in W^\lambda$ and $\mu = z\lambda$ then
$\mathrm{Inv}(v_\lambda z) = \{ (i,j)\ |\ \hbox{$1\le i<j\le n$ and $\mu_i>\mu_j$} \}
$,
so that
\begin{align}
t^{\frac12\ell(v_\lambda z)} \ev^t_\mu(c_{v_\lambda z}^Y)
&= \ev^t_\mu\Big( \prod_{1\le i<j\le n\atop \mu_i>\mu_j }
t^{\frac12} \frac{t^{-\frac12}-t^{\frac12}Y_iY^{-1}_j } { 1-Y_iY^{-1}_j } \Big) 
= \ev^t_\mu\Big( \prod_{1\le i<j\le n\atop \mu_i>\mu_j }
t \frac{t^{-1}Y^{-1}_iY_j-1 } { Y^{-1}_iY_j-1 } \Big) 
\nonumber \\
&= 
\prod_{1\le i<j\le n\atop \mu_i>\mu_j }
t \Big( \frac{1 - q^{\mu_i-\mu_j} t^{v_\mu(i)-v_\mu(j)-1} }
{1 - q^{\mu_i-\mu_j} t^{v_\mu(i)-v_\mu(j)} }\Big),
\label{Eexpb}
\end{align}
where we have used that
$\ev^t_\mu(Y_i) = q^{-\mu_i}t^{-(v_\mu(i)-1) + \frac12(n-1)}$. Thus
$$
\ev^t_\mu(Y^{-1}_iY_j) 
= q^{\mu_i}t^{(v_\mu(i)-1) - \frac12(n-1)}
q^{-\mu_j}t^{-(v_\mu(j)-1) + \frac12(n-1)}
= q^{\mu_i-\mu_j}t^{v_\mu(i)-v_\mu(j)}.
$$

For the second statement,
\begin{align*}
A_{\lambda+\rho}(q,t)
&= t^{\frac12\ell(w_0)}\varepsilon_0 E_{\lambda+\rho}(q,t)
= t^{\frac12\ell(w_0)} c_{w_0}^{Y^{-1}} \sum_{z\in S_n} \det(w_0z)\eta_z E_{\lambda+\rho}(q,t)
\\
&= t^{\frac12\ell(w_0)}
\sum_{z\in S_n} \det(w_0z) c_{w_0z}^{Y^{-1}} t^{-\frac12\ell(z)} t^{\frac12\ell(z)} \tau^\vee_z 
E_{\lambda+\rho}(q,t) \\
&= 
\sum_{z\in S_n} \det(w_0z) c_{w_0z}^{Y^{-1}} t^{\frac12\ell(w_0z)}  
E_{z(\lambda+\rho)}(q,t)
\\
&= \sum_{z\in S_n} (-1)^{\ell(w_0z)} t^{\frac12\ell(w_0z)}
\ev^t_{z(\lambda+\rho)}(c_{w_0z}^{Y^{-1}})
E_{z(\lambda+\rho)}(q,t).
\end{align*}
If $\mu = z(\lambda+\rho)$ then, in a manner similar to the computation in~\eqref{Eexpb},
\begin{align}
(-t^{\frac12})^{\ell(w_0z)} \ev^t_\mu(c_{w_0z}^{Y^{-1}}) 
&= \prod_{1\le i<j\le n\atop \mu_i >\mu_j}
(-t^{\frac12}) \frac{t^{-\frac12}-t^{\frac12}Y^{-1}_iY_j}{1-Y^{-1}_iY_j}
\nonumber \\
&=
\prod_{1\le i<j\le n\atop \mu_i >\mu_j}
(-1) \frac{1-t q^{\mu_i-\mu_j} t^{v_\mu(i)-v_\mu(j)\rangle} }
{1-q^{\mu_i-\mu_j} t^{v_\mu(i)-v_\mu(j)\rangle}}.
\label{Eexpd}
\end{align}
\end{proof}

\section{Principal specializations and hook formulas}

In Section~\ref{sec:actions} we saw that $c$-functions provide the explicit constants for normalization
and the $E$-expansion of bosonic and fermionic Macdonald polynomials.  In this section we see the role 
of $c$-functions in the amazing product formulas for principal specializations.  
These principal specializations capture many of
the hook type formulas that appear in formulas for dimensions of irreducible representations
of the general linear group and the symmetric group.  As discussed in Section~\ref{ssec:elldim},
the representation theoretic interpretation of the hook formulas that arise from the principal specializations 
of Macdonald polynomials  is still rather mysterious, but we hope
that viewing these results as being sourced from evaluations of $c$-functions will help to provide insight.
Our exposition follows~\cite[(5.2.14) and (5.3.9)]{Mac03} (with some streamlining) for the $c$-function formulas,
and then follows~\cite{AGY22} for rewriting the $c$-function formulas into hook formulas.
This provides an alternative route to the proof of the principal specialization formulas for $P_\lambda$
which are given in~\cite[($6.11'$)]{Mac}.

\subsection{$c$-function formulas}

An \emph{$n$-periodic permutation} is a bijection $w\colon \ZZ\to \ZZ$ such that $w(i+n) = w(i)+n$ for all $i\in \ZZ$.
Given an $n$-periodic permutation $w$, define its \emph{set of inversions} and \emph{length} by
$$\mathrm{Inv}(w) = \left\{ (i,k)\ \Big|\  
\begin{array}{c}
\hbox{$i\in \{1, \ldots, n\}, k\in \ZZ$} \\ \hbox{$i<k$ and $w(i)>w(k)$}
\end{array}\right\}
\qquad\hbox{and}\qquad
\ell(w) = \#\mathrm{Inv}(w).
$$
Define an action of the $n$-periodic permutations on $\ZZ^n$ by setting
$$w(\mu_1, \ldots, \mu_n) = (\mu_{v(1)}+\ell_1, \ldots, \mu_{v(n)}+\ell_n),$$
where $v(i)\in \{1,\ldots, n\}$ and $\ell_i\in \ZZ$ are determined by
$w(i) = v(i)+\ell_i n$.
Given $\mu\in \ZZ^n$, we consider the following $n$-periodic permutations associated to $\mu$ 
$$\begin{array}{l}
\hbox{$u_\mu$ be the minimal length $n$-periodic permutation such that
$u_\mu(0,\ldots, 0)= (\mu_1, \ldots, \mu_n)$ and }
\\
\hbox{$t_\mu$ be the $n$-periodic permutation given by $t_\mu(i) = i+n\mu_i$.}
\end{array}
$$
Recall that $v_\mu\in S_n$ denotes the minimal length permutation such that $v_\mu\mu$ is weakly increasing. The three permutations are related by $u_\mu = t_\mu v_\mu^{-1}$.

We extend the definition of $c$-functions to $n$-periodic permutations. 
For $i,j\in \{1, \ldots, n\}$ and $\ell\in \ZZ$ define
\begin{equation*}
c^{Y^{-1}}_{(i,j+\ell n)} = t^{-\frac12} \frac{1-q^{\ell} tY_i^{-1}Y_j}{1-q^\ell Y_i^{-1}Y_j}
\qquad\hbox{and}\qquad
c^{Y^{-1}}_w = \prod_{(i,k)\in \mathrm{Inv}(w)} c^{Y^{-1}}_{(i,k)}.
\end{equation*}
Define ring homomorphisms $\ev^t_0\colon \CC[Y] \to \CC$ and $\ev^{t^{-1}}_0\colon \CC[Y] \to \CC$ 
by
$$\ev^t_0(Y_i) = t^{-(i-1)+\frac12(n-1)}
\qquad\hbox{and}\qquad
\ev^{t^{-1}}_0(Y_i) = t^{(i-1)-\frac12(n-1)},
\qquad\hbox{for $i\in \{1, \ldots, n\}$.}
$$

\begin{thm}  \label{princspecA}
Let $\mu, \lambda \in \ZZ^n$ with $\lambda_1 \ge \cdots \ge \lambda_n$. 
Then 
\begin{align*}
E_\mu(1,t,t^2, \ldots, t^{n-1};q,t) &= t^{\frac{(n-1)}{2}\vert \lambda\vert}
t^{-\frac12\ell(v^{-1}_\mu)} \ev^t_0(c^{Y^{-1}}_{u_\mu}),
\\
P_\lambda(1,t,t^2, \ldots, t^{n-1};q,t) &= t^{\frac{(n-1)}{2}\vert \lambda\vert}
\ev^{t^{-1}}_0(c^{Y^{-1}}_{t_\lambda})
\qquad
\hbox{and} \\
A_{\lambda+\rho}(1,t,t^2, \ldots, t^{n-1};q,t) &= 0.
\end{align*}
\end{thm}

Let $\mu \in \ZZ^n_{\ge 0}$.
Using the formulas  (see~\cite[Proposition 2.1(d) and Proposition 2.2(b)]{GR21})
\begin{equation}
\mathrm{Inv}(t_\lambda) 
=\{ (i,j+\ell n)\ |\ \hbox{$i,j\in \{1, \ldots, n\}$, $i<j$ and $\ell\in \{0, 1, \ldots, \lambda_j-\lambda_i-1\}$}\}
\label{Invfortlambda}
\end{equation}
and
\begin{equation}
\mathrm{Inv}(u_\mu) 
=\{ (v_{\mu}(r), i+(\mu_r-c+1)n)\ |\ \hbox{$(r,c)\in \mu$ and $i\in \{1, \ldots, u_\mu(r,c)\}$}\}.
\label{Invforumu}
\end{equation}
gives the following corollary.

\begin{cor}  \label{princspecB}
Let $\mu\in \ZZ^n_{\ge 0}$. Denote by $\lambda$ the decreasing rearrangement of $\mu$ and $n(\lambda) = \sum_{i=1}^n (i-1)\lambda_i$.
Then
$$P_\lambda(1,t,t^2, \ldots, t^{n-1};q,t) 
= t^{n(\lambda)} \prod_{1\le i<j\le n} \prod_{\ell=0}^{\lambda_i-\lambda_j-1}
\frac{1 - q^\ell t^{j-i +1} }{1- q^\ell t^{j-i}} 
$$
and
$$E_\mu(1,t,t^2, \ldots, t^{n-1};q,t)
= t^{-\frac12\ell(v^{-1}_\mu)} \prod_{(r,c)\in \mu} \prod_{i=1}^{u_\mu(r,c)}
\frac{1-q^{\mu_r-c+1} t^{v_\mu(r)-i+1} }{1-q^{\mu_r-c+1} t^{v_\mu(r)-i}}.
$$
\end{cor}

\subsection{Hook formulas for the bosonic and electronic cases}

Let $\lambda\in \ZZ^n_{\ge 0}$ with
$\lambda_1\ge \cdots \ge \lambda_n$. 
Let $\lambda'$ denote the conjugate partition to $\lambda$ (i.e. for $c\in \ZZ_{>0}$ let
$\lambda'_c = \#\{ j\in \ZZ_{>0}\ |\ \lambda_j\ge c\}$).  A \emph{box in $\lambda$} is a 
pair $b=(r,c)$ with $r\in \{1, \dots, n\}$ and $c\in \{1, \ldots, \lambda_i\}$.
For a box $b=(r,c)$ in $\lambda$ define
$$
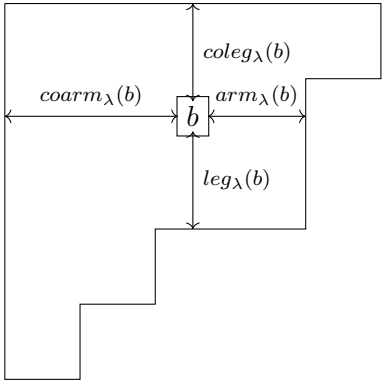
\begin{matrix}
\begin{tikzpicture}[yscale=-1]
\draw (0,0) -- (5,0) -- (5,1) -- (4,1) -- (4,3)  -- (2,3) -- (2,4) -- (1,4) -- (1,5) -- (0,5) -- (0,0);
\draw[<->] (2.5,0) -- node[anchor=west]{$\scriptstyle{coleg_\lambda(b)}$} (2.5,1.3) ;
\draw[<->] (2.5,1.7) -- node[anchor=west]{$\scriptstyle{leg_\lambda(b)}$} (2.5,3) ;
\draw[<->] (0,1.5) -- node[anchor=south]{$\scriptstyle{coarm_\lambda(b)}$} (2.3,1.5) ;
\draw[<->] (2.7,1.5) -- node[anchor=south]{$\scriptstyle{arm_\lambda(b)}$} (4,1.5) ;
\draw (2.5,1.5)  node[shape=rectangle,draw]{$b$};
\end{tikzpicture}
\end{matrix}
\qquad\quad
\begin{matrix}
&\mathrm{coleg}_\lambda(b) = r-1, \\
\\
\ 
\\
\mathrm{coarm}_\lambda(b) = c-1, &b=(r,c),
&\mathrm{arm}_\lambda(b) = \lambda_r-c, \\
\\
\ 
\\
&\mathrm{leg}_\lambda(b) = \lambda_c'-r.
\end{matrix}
$$
The \emph{hook length} $h(b)$ and the \emph{content} $c(b)$ of the box $b$ are defined by
\begin{equation*}
h(b) = \mathrm{arm_\lambda}(b)+\mathrm{leg}_\lambda(b)+1
\qquad\hbox{and}\qquad
c(b) = \mathrm{coarm}_\lambda(b) - \mathrm{coleg}_\lambda(b).
\end{equation*}

\begin{thm} \label{hookforP}  
Let $\lambda = (\lambda_1, \ldots, \lambda_n)\in \ZZ^n_{\ge 0}$ with
$\lambda_1\ge \cdots \ge \lambda_n$. Then
\begin{align*}
P_{\lambda}(1,t,t^2, \ldots, t^{n-1};q,t)
&= t^{n(\lambda)} \prod_{b\in \lambda} 
\frac{1-q^{\mathrm{coarm}_\lambda(b)}t^{n-\mathrm{coleg}_\lambda(b)}}
{1-q^{\mathrm{arm}_\lambda(b)}t^{\mathrm{leg}_\lambda(b)+1}} . 
\end{align*} 
\end{thm}

\begin{thm}\label{hookforE}
Let $\mu=(\mu_1, \ldots, \mu_n)\in \ZZ_{\ge 0}^n$ and let $\lambda$ be the weakly decreasing rearrangment of $\mu$. 
Then
\begin{align*}
E_\mu(1,t,t^2, \ldots, t^{n-1};q,t) 
&= t^{n(\lambda)}
\prod_{(r,c) \in \mu} 
\frac{1-q^c t^{v_\mu(r)} } 
{1-q^{\mu_r-c+1} t^{v_\mu(r)- u_\mu(r,c) } } .
\end{align*} 
\end{thm}

\subsection{Proofs}

\subsubsection{Proof of the $c$-function formula}

\begin{dupThm}[~\ref{princspecA}]
Let $\mu, \lambda \in \ZZ^n$ with $\lambda_1 \ge \cdots \ge \lambda_n$. 
Then 
\begin{align*}
E_\mu(1,t,t^2, \ldots, t^{n-1};q,t) &= t^{\frac{(n-1)}{2}\vert \lambda\vert}
t^{-\frac12\ell(v^{-1}_\mu)} \ev^t_0(c^{Y^{-1}}_{u_\mu}),
\\
P_\lambda(1,t,t^2, \ldots, t^{n-1};q,t) &= t^{\frac{(n-1)}{2}\vert \lambda\vert}
\ev^{t^{-1}}_0(c^{Y^{-1}}_{t_\lambda})
\qquad
\hbox{and} \\
A_{\lambda+\rho}(1,t,t^2, \ldots, t^{n-1};q,t) &= 0.
\end{align*}
\end{dupThm}

\begin{proof}
For this proof use the realization of the polynomial representation $\CC[X]$ as an induced
module $\widetilde{H}\mathbf{1}_Y$ via the $\widetilde{H}$-module isomorphism of~\eqref{CXasIndHY}.  Then the creation formulas for $E_\mu$, $P_\lambda$ and $A_{\lambda+\rho}$
are
$$E_\mu = t^{-\frac12\ell(v^{-1}_\mu)}\tau^\vee_{u_\mu} \mathbf{1}_0,
\qquad
P_\lambda = \frac{t^{\frac12\ell(w_0)}}{W_\lambda(t)} \mathbf{1}_0E_\lambda,
\qquad\hbox{and}\qquad
A_{\lambda+\rho} = t^{\frac12\ell(w_0)}\varepsilon_0 E_{\lambda+\rho}
$$
(see Theorem~\ref{creationformulathm}  and~\eqref{creationPA}).

Let $\mathbf{1}_X$ be a formal symbol which satisfies
$\mathbf{1}_X T_j = t^{\frac12}\mathbf{1}_X$ and $\mathbf{1}_Xg^\vee = \mathbf{1}_X$,
for $j\in \{1, \ldots, n-1\}$.
Since $g^\vee = x_1T_1\cdots T_{n-1}$ and 
$x_{i+1} = T_i x_i T_i$ then
$x_1 = g^\vee T^{-1}_{n-1}\cdots T^{-1}_1$ and
$$\mathbf{1}_X x_i = t^{-\frac12(n-1)} t^{i-1} \mathbf{1}_X,
\qquad\hbox{for $i\in \{1, \ldots, n\}$.}
$$
Thus, if $\mu\in \ZZ^n$ then
$$\mathbf{1}_X E_\mu(x_1, \ldots, x_n;q,t) 
= \mathbf{1}_X t^{-\frac12(n-1)\vert \mu \vert} E_\mu(1,t,t^2, \ldots, t^{n-1};q,t).
$$
For $i\in \{1, \ldots, n-1\}$,
\begin{align*}
\mathbf{1}_X \tau^\vee_i 
&= \mathbf{1}_X \Big(T_i + \frac{t^{-\frac12}-t^{\frac12}}{1-Y^{-1}_i Y_{i+1}}\Big)
= \mathbf{1}_X \Big(t^{\frac12} + \frac{t^{-\frac12}-t^{\frac12}}{1-Y^{-1}_i Y_{i+1}}\Big) 
= \mathbf{1}_X \Big(\frac{t^{-\frac12}-t^{\frac12}Y^{-1}_iY_{i+1}}{1-Y^{-1}_i Y_{i+1}}\Big) 
= \mathbf{1}_X c^{Y^{-1}}_{i,i+1}.
\end{align*}
By~\eqref{EeigenvalueB}, 
$$c^{Y^{-1}}_{i,i+1} \mathbf{1}_Y = \ev^t_0(c^{Y^{-1}}_{i,i+1})\mathbf{1}_Y.$$
If $w\in W$ and $\ell(s_iw)>\ell(w)$ then 
$$\mathbf{1}_X\tau^\vee_i \tau^\vee_w \mathbf{1}_Y
=\mathbf{1}_Xc^{Y^{-1}}_{i,i+1} \tau^\vee_w \mathbf{1}_Y
=\mathbf{1}_X \tau^\vee_w c^{Y^{-1}}_{w^{-1}(i),w^{-1}(i+1)} \mathbf{1}_Y
=\ev^t_0(c^{Y^{-1}}_{w^{-1}(i),w^{-1}(i+1)}) \mathbf{1}_X \tau^\vee_w  \mathbf{1}_Y.
$$
By induction, conclude that 
if $w\in W$ and $w = s_{i_1}\cdots s_{i_\ell}$ is 
a reduced word for $w$ then 
$$
\mathbf{1}_X \tau^\vee_w \mathbf{1}_Y
=\mathbf{1}_X \tau^\vee_{i_1}\cdots \tau^\vee_{i_\ell} \mathbf{1}_Y
= \mathbf{1}_X \ev^t_0(c^{Y^{-1}}_w) \mathbf{1}_Y
= \ev^t_0(c^{Y^{-1}}_w)\mathbf{1}_X\mathbf{1}_Y.
$$
Thus, by the creation formula $E_\mu = t^{-\frac12\ell(v^{-1}_\mu)}\tau^\vee_{u_\mu}\mathbf{1}_Y$, 
\begin{align*}
t^{-\frac12(n-1)\vert \mu \vert}E_\mu(1,t,\ldots, t^{n-1};q,t) \mathbf{1}_X \mathbf{1}_Y
&= \mathbf{1}_X E_\mu \mathbf{1}_Y
= \mathbf{1}_X t^{-\frac12\ell(v^{-1}_\mu)} \tau^\vee_{u_\mu} \mathbf{1}_Y
= t^{-\frac12\ell(v^{-1}_\mu)} \ev^t_0(c^{Y^{-1}}_{u_\mu}) \mathbf{1}_X \mathbf{1}_Y.
\end{align*}
which completes the proof of the first statement.

Using the creation formula 
$P_\lambda = \dfrac{t^{\frac12\ell(w_0)} }{W_\lambda(t)}\mathbf{1}_0 E_\lambda$ gives
\begin{align*}
t^{-\frac12(n-1)\vert \lambda \vert}&P_\lambda(1,t,\ldots, t^{n-1};q,t) \mathbf{1}_X \mathbf{1}_Y
= \mathbf{1}_X P_\lambda \mathbf{1}_Y
= \mathbf{1}_X \frac{t^{\frac12\ell(w_0)}}{W_\lambda(t)} \mathbf{1}_0 E_\lambda \mathbf{1}_Y
\\
&= \mathbf{1}_X \frac{t^{\frac12\ell(w_0)}W_0(t)}{W_\lambda(t)} E_\lambda \mathbf{1}_Y
= t^{\frac12\ell(w_0)}t^{-\frac12\ell(v^{-1}_\lambda)} 
\frac{W_0(t)}{W_\lambda(t)} \ev^t_0(c^{Y^{-1}}_{u_\lambda} ) \mathbf{1}_X\mathbf{1}_Y.
\end{align*}
Since $v^{-1}_\lambda = (w_0w_\lambda)^{-1} = w_\lambda w_0$
and $t_\lambda = u_\lambda v_\lambda$ then
\begin{align*}
\frac{t^{-\frac12\ell(w_0)}W_0(t)}{t^{-\frac12\ell(w_\lambda)}W_\lambda(t)} 
\ev^t_0(c^{Y^{-1}}_{u_\lambda})
&= \ev^{t^{-1}}_0\Big(\frac{c^{Y^{-1}}_{w_0})}{c^{Y^{-1}}_{w_\lambda}}\Big)
\ev^t_0(c^{Y^{-1}}_{u_\lambda})
= \ev^{t^{-1}}_0\Big(\frac{c^{Y^{-1}}_{w_0} )}{c^{Y^{-1}}_{w_\lambda}}\Big)
\ev^{t^{-1}}_0(v^{-1}_\lambda c^{Y^{-1}}_{u_\lambda})
\\
&= \ev^{t^{-1}}_0(c^{Y^{-1}}_{v_\lambda})
\ev^{t^{-1}}_0(v^{-1}_\lambda c^{Y^{-1}}_{u_\lambda})
= \ev^{t^{-1}}_0(c^{Y^{-1}}_{u_\lambda v_\lambda})
= \ev^{t^{-1}}_0(c^{Y^{-1}}_{t_\lambda}).
\end{align*}
The first equality comes from Proposition~\ref{paraPoin} which gives
$$t^{-\frac12\ell(w_\lambda)}W_\lambda(t) = \ev^t_0(c^Y_{w_\lambda})
= \ev^{t^{-1}}_0(c^{Y^{-1}}_{w_\lambda}),$$
and the second equality results from the fact that 
$$\hbox{
if $i,j\in \{1, \ldots, n\}$ and $i<j$ then $t^{j-i} = (t^{-1})^{(n-j)-(n-i)} = (t^{-1})^{w_0(j)-w_0(i)}$,}
$$
which gives
$$\ev^t_0(c^{Y^{-1}}_{u_\lambda}) 
= \ev^{t^{-1}}_0(v^{-1}_\lambda c^{Y^{-1}}_{u_\lambda}).
$$

Using the creation formula $A_{\lambda+\rho} 
= t^{\frac12\ell(w_0)} \varepsilon_0 E_{\lambda+\rho}$ gives
$$
t^{-\frac12(n-1)\vert \lambda \vert} A_{\lambda+\rho}(1,t,\ldots, t^{n-1};q,t) \mathbf{1}_X \mathbf{1}_Y
= \mathbf{1}_X A_{\lambda+\rho} \mathbf{1}_Y
= t^{\frac12\ell(w_0} \mathbf{1}_X \varepsilon_0 E_{\lambda+\rho} \mathbf{1}_Y = 0,
$$
since $\mathbf{1}_X \varepsilon_0 = 0$.
Thus $A_{\lambda+\rho}(1,t,\ldots, t^{n-1};q,t)=0$.
\end{proof}

\subsubsection{Proof of the root product formula}

\begin{dupCor}[~\ref{princspecB}]
Let $\mu\in \ZZ^n_{\ge 0}$. Denote by $\lambda$ the decreasing rearrangement of $\mu$ and $n(\lambda) = \sum_{i=1}^n (i-1)\lambda_i$.
Then
$$P_\lambda(1,t,t^2, \ldots, t^{n-1};q,t) 
= t^{n(\lambda)} \prod_{1\le i<j\le n} \prod_{\ell=0}^{\lambda_i-\lambda_j-1}
\frac{1 - q^\ell t^{j-i +1} }{1- q^\ell t^{j-i}} 
$$
and
$$E_\mu(1,t,t^2, \ldots, t^{n-1};q,t)
= t^{-\frac12\ell(v^{-1}_\mu)} \prod_{(r,c)\in \mu} \prod_{i=1}^{u_\mu(r,c)}
\frac{1-q^{\mu_r-c+1} t^{v_\mu(r)-i+1} }{1-q^{\mu_r-c+1} t^{v_\mu(r)-i}}.
$$
\end{dupCor}

\begin{proof}
Let $\rho^\vee = \frac12(n-1, n-3, \ldots, -(n-3), -(n-1))$.  Then
\begin{align*}
n(\lambda) 
&
= \langle (\lambda_1, \ldots, \lambda_n), (0,1,\ldots, n-1)\rangle
= \langle \lambda, -\rho^\vee+\hbox{$\frac{(n-1)}{2}$}(1,1,\ldots, 1)\ \rangle 
= \hbox{$\frac{(n-1)}{2}$} \vert\lambda\vert - \langle \lambda, \rho^\vee\rangle
\end{align*}
and, since $\ell(t_\lambda) = \sum_{i<j} (\lambda_i - \lambda_j) $ then
\begin{align*}
\ell(t_\lambda) 
&= (n-1)\lambda_1+(n-2)\lambda_2+\cdots+(n-n)\lambda_n
- \big( (n-1)\lambda_n + (n-2)\lambda_{n-1} + \cdots + (n-n)\lambda_1\big) 
\\
&= (n-1)\lambda_1+(n-3)\lambda_2+(n-5)\lambda_3+\cdots -(n-3)\lambda_{n-1}-(n-1)\lambda_n 
= \langle \lambda, 2\rho^\vee\rangle,
\end{align*}
so that 
\begin{equation}
n(\lambda) = \hbox{$\frac12$}(n-1)\vert \lambda \vert - \hbox{$\frac12$}\ell(t_\lambda).
\label{nlambda}
\end{equation}
Since 
\begin{align*}
\ev^{t^{-1}}_0&(c_{i,j+rn}^{Y^{-1}})
= \ev^{t^{-1}}_0\Big(t^{-\frac12} \frac{1-tq^r Y^{-1}_iY_j }{1-q^r Y^{-1}_iY_j }\Big)
\\
&=t^{-\frac12} \frac{ 1- t q^r t^{-(i-1)+\frac12(n-1)}t^{(j-1)-\frac12(n-1)} } 
{1-q^r t^{-(i-1)+\frac12(n-1)} t^{(j-1)-\frac12(n-1)} } 
=t^{-\frac12} \frac{ 1- q^rt^{j-i+1} } {1-q^r t^{j-i} }
\end{align*}
then, using~\eqref{Invfortlambda} and~\eqref{nlambda},
\begin{align*}
t^{\frac{(n-1)}{2}\vert \lambda\vert} \ev^{t^{-1}}_0(c_{t_\lambda}^{Y^{-1}}) 
&= t^{\frac{(n-1)}{2}\vert \lambda\vert} t^{-\frac12\ell(t_\lambda)}
\prod_{i<j} \prod_{r=0}^{\lambda_i-\lambda_j-1} \frac{1-q^r t^{j-i+1}}{1-q^rt^{j-i}}  
= t^{n(\lambda)}
\prod_{i<j} \prod_{r=0}^{\lambda_i-\lambda_j-1} \frac{1-q^r t^{j-i+1}}{1-q^rt^{j-i}}.
\end{align*}
The formula in the first statement follows then from Theorem~\ref{princspecA}.

Again by Theorem~\ref{princspecA},
\begin{align*}
E_\mu(1,t,t^2, \ldots, t^{n-1};q,t) 
&= t^{\frac{(n-1)}2 \vert \mu\vert}t^{-\frac12\ell(v^{-1}_\mu)}
\ev_{k\rho}(c_{u_\mu}^{Y^{-1}}) 
= t^{\frac{(n-1)}2 \vert \mu\vert - \frac12\ell(u_\mu)-\frac12\ell(v_\mu)}
\Big(t^{\frac12\ell(u_\mu)}
\ev_{k\rho}(c_{u_\mu}^{Y^{-1}}) \Big) 
\end{align*}
and, if $\lambda$ is the weakly decreasing rearrangement of $\mu$ then, 
by~\eqref{nlambda},
\begin{align*}
\hbox{$\frac{(n-1)}{2}$} \vert \mu\vert - \hbox{$\frac12$}\ell(u_\mu)-\hbox{$\frac12$}\ell(v_\mu)
&=
\hbox{$\frac{(n-1)}2$} \vert \lambda\vert - \hbox{$\frac12$}\ell(t_\mu)
=
\hbox{$\frac{(n-1)}2$} \vert \lambda\vert - \hbox{$\frac12$}\ell(t_\lambda)
= n(\lambda).
\end{align*}
Then~\eqref{Invforumu} gives that  
$t^{\frac12\ell(u_\mu)} \ev^t_0(c_{u_\mu}^{Y^{-1}})$ is equal to
\begin{equation}
\ev^t_0 \Big(
\prod_{(r,c) \in \mu} \prod_{j=1}^{u_\mu(r,c)}
\frac{1-q^{\mu_r-c+1} t Y^{-1}_{v_\mu(r)} Y_j  } 
{1-q^{\mu_r-c+1} Y^{-1}_{v_\mu(r)} Y_j  }  \Big) 
= 
\prod_{(r,c) \in \mu} \prod_{j=1}^{u_\mu(r,c)}
\frac{1-t q^{\mu_r-c+1} t^{v_\mu(r)-j} } 
{1-q^{\mu_r-c+1} t^{v_\mu(r)-j} } .
\label{cumueval}
\end{equation}
\end{proof}

\subsubsection{Proof of the bosonic hook formula}

\begin{dupThm}[~\ref{hookforP}]  
Let $\lambda = (\lambda_1, \ldots, \lambda_n)\in \ZZ^n_{\ge 0}$ with
$\lambda_1\ge \cdots \ge \lambda_n$. Then
\begin{align*}
P_{\lambda}(1,t,t^2, \ldots, t^{n-1};q,t)
&= t^{n(\lambda)} \prod_{b\in \lambda} 
\frac{1-q^{\mathrm{coarm}_\lambda(b)}t^{n-\mathrm{coleg}_\lambda(b)}}
{1-q^{\mathrm{arm}_\lambda(b)}t^{\mathrm{leg}_\lambda(b)+1}} . 
\end{align*} 
\end{dupThm}
\begin{proof}  
In view of Corollary~\ref{princspecB} the result will follow if we prove that
$$\prod_{i<j} \prod_{\ell=0}^{\lambda_i-\lambda_j-1}
\frac{1 - q^\ell t^{j-i +1} }{1- q^\ell t^{j-i}} 
= \prod_{b\in \lambda} 
\frac{1-q^{\mathrm{coarm}_\lambda(b)}t^{n-\mathrm{coleg}_\lambda(b)}}
{1-q^{\mathrm{arm}_\lambda(b)}t^{\mathrm{leg}_\lambda(b)+1}}.
$$
The left-hand side is
$$
LHS = \prod_{i<j} \prod_{\ell=0}^{\lambda_i-\lambda_j-1}
\frac{1 - q^\ell t^{j-i +1} }{1- q^\ell t^{j-i}} 
= \prod_{r=1}^n \prod_{j=r+1}^n \prod_{\ell=0}^{\lambda_r-\lambda_j-1}
\frac{1 - t q^\ell t^{j-r} }{1- q^\ell t^{j-r}}.
$$
Let $m$ be the number of columns of length $n$ in $\lambda$.  
Let $r\in \{1, \ldots, n\}$.  Switching the products over $j$ and $\ell$ gives
\begin{equation}
\prod_{j=r+1}^n \prod_{\ell=0}^{\lambda_r-\lambda_j-1}
\frac{1 - t q^\ell t^{j-r} }{1- q^\ell t^{j-r}}  
=\prod_{c=m+1}^{\lambda_r} \prod_{j=\lambda_c'+1}^{n} 
\frac{1-tq^{\lambda_r-c}t^{j-r} }{1-q^{\lambda_r-c} t^{j-r}}
=\prod_{c=m+1}^{\lambda_r} 
\frac{1-tq^{\lambda_r-c}t^{n-r} }{1-q^{\lambda_r-c} t^{\lambda'_c+1-r}}
\label{hookmiracle}
\end{equation}
The definitions of arms, legs, coarms and colegs of boxes give that
\begin{align*}
RHS &=
t^{n(\lambda)}\prod_{b=(r,c)\in \lambda} 
\frac{1-q^{\mathrm{coarm}_\lambda(b)} t^{n-\mathrm{coleg}_\lambda(b)} }
{1-q^{\mathrm{arm}_\lambda(b)} t^{\mathrm{leg}_\lambda(b)+1}}
=t^{n(\lambda)} \prod_{r=1}^n  \prod_{c=1}^{\lambda_r}
\frac{1-q^{c-1} t^{n-(r-1)} } {1-q^{\lambda_r-c} t^{\lambda_c'-r+1 } } 
\end{align*}
For $r\in \{1, \ldots, n\}$ let $\ell = \lambda_r$ and write
\begin{align*}
\prod_{c=1}^{\lambda_r} \frac{1-q^{c-1} t^{n-(r-1)} } {1-q^{\lambda_r-c} t^{\lambda_c'-r+1 } } 
= \frac{(1-q^0t^{n-(r-1)})(1-q^1t^{n-(r-1)}) \cdots (1-q^{\lambda_r-1}t^{n-(r-1)})}
{(1-q^{\lambda_r-1}t^{\lambda_1'-r+1})\cdots (1-q^1t^{\lambda'_{\ell-1}-r-1+1})
(1-q^0 t^{\lambda_\ell'-r+1)}}
\end{align*}
to observe that the last $m$ factors in the numerator cancel with the first $m$ terms in
the denominator.  Thus
\begin{align*}
\prod_{c=1}^{\lambda_r} \frac{1-q^{c-1} t^{n-(r-1)} } {1-q^{\lambda_r-c} t^{\lambda_c'-r+1 } } 
&= \frac{(1-q^0t^{n-(r-1)})(1-q^1t^{n-(r-1)}) \cdots (1-q^{\lambda_r-m} t^{n-(r-1)})}
{(1-q^{\lambda_r-m}t^{\lambda_{\ell-m}'-r+1})\cdots (1-q^1t^{\lambda'_{\ell-1}-r-1+1})
(1-q^0 t^{\lambda_\ell'-r+1)}}
\\
&=\prod_{c=m+1}^{\lambda_r} 
\frac{1-q^{\lambda_r-c} t^{n-(r-1)} } {1-q^{\lambda_-c} t^{\lambda_c'-r+1 } }
=\prod_{c=m+1}^{\lambda_r} 
\frac{1-t q^{\lambda_r-c} t^{n-r} } {1-q^{\lambda_-c} t^{\lambda_c'-r+1 } }.
\end{align*}
Since this is equal to the expression in~\eqref{hookmiracle}, the result follows.
\end{proof}

\subsubsection{Proof of the electronic hook formula}

\begin{dupThm}[~\ref{hookforE}]
Let $\mu=(\mu_1, \ldots, \mu_n)\in \ZZ_{\ge 0}^n$ and let $\lambda$ be the weakly decreasing rearrangment of $\mu$. 
Then
\begin{align*}
E_\mu(1,t,t^2, \ldots, t^{n-1};q,t) 
&= t^{n(\lambda)}
\prod_{(r,c) \in \mu} 
\frac{1-q^c t^{v_\mu(r)} } 
{1-q^{\mu_r-c+1} t^{v_\mu(r)- u_\mu(r,c) } } .
\end{align*} 
\end{dupThm}

\begin{proof}
In view of Corollary~\ref{princspecB} the result will follow if we prove that
$$
\prod_{b = (r,c) \in \mu} \prod_{j=1}^{u_\mu(r,c)}
\frac{1- q^{\mu_r-c+1} t^{v_\mu(r)-j+1} } 
{1-q^{\mu_r-c+1} t^{v_\mu(r)-j} }  
=\prod_{(r,c) \in \mu} 
\frac{1- q^c t^{v_\mu(r)} } 
{1-q^{\mu_r-c+1} t^{v_\mu(r)- u_\mu(r,c) } } .
$$
For a single box $b=(r,c)$, the product
$\displaystyle{
\prod_{j=1}^{u_\mu(r,c)}
\frac{1-t q^{\mu_r-c+1} t^{v_\mu(r)-j} } 
{1-q^{\mu_r-c+1} t^{v_\mu(r)-j} }
}$
is equal to
\begin{align*}
\frac{(1- q^{\mu_r-c+1} t^{v_\mu(r)}) } 
{\cancel{(1-q^{\mu_r-c+1} t^{v_\mu(r)-1} )} }
&\frac{\cancel{ (1- q^{\mu_r-c+1} t^{v_\mu(r)-1})  } } 
{\cancel{ (1-q^{\mu_r-c+1} t^{v_\mu(r)-2} ) } }
\cdots
\frac{\cancel{ (1- q^{\mu_r-c+1} t^{v_\mu(r) - (u_\mu(r,c)-1)} ) } } 
{(1-q^{\mu_r-c+1} t^{v_\mu(r)-u_\mu(r,c)} )}
\\
&=
\frac{1- q^{\mu_r-c+1} t^{v_\mu(r)} } 
{1-q^{\mu_r-c+1} t^{v_\mu(r)-u_\mu(r,c)} }.
\end{align*}
For a single fixed row $r$, the product
$\displaystyle{
\prod_{c=1}^{\mu_r} (1- q^{\mu_r-c+1} t^{v_\mu(r)} )
=\prod_{c=1}^{\mu_r} 1-q^c  t^{v_\mu(r)},
}$
and so
\begin{align*}
\prod_{(r,c) \in \mu} \prod_{j=1}^{u_\mu(r,c)}
\frac{1- q^{\mu_r-c+1} t^{v_\mu(r)-j+1} } 
{1-q^{\mu_r-c+1} t^{v_\mu(r)-j} }  
&=
\prod_{(r,c) \in \mu} 
\frac{1- q^{\mu_r-c+1} t^{v_\mu(r)} } 
{1-q^{\mu_r-c+1} t^{v_\mu(r)-u_\mu(r,c)} }
= \prod_{(r,c)\in \mu} 
\frac{1-q^ct^{v_\mu(r)} }{ 1- q^{\mu_r-c+1}t^{v_\mu(r)-u_\mu(r,c)} }.
\end{align*}
\end{proof}

\section{The inner product and the Weyl character formula}\label{section:WCF}

One of the amazing formulas of mathematics is the Weyl character formula which, in the type
$GL_n$ case says that the Schur polynomial is a quotient of two determinants (the denominator is
a Vandermonde determinant).  Miraculously, the Weyl character formula generalizes to a similar quotient formula
for the bosonic Macdonald polynomial $P_\lambda(q,qt)$.  In this section we prove this quotient formula
for $P_\lambda(q,qt)$,
following the same proof as in~\cite[\S5]{Mac03} but with some simplifications in notation.

\subsection{The inner product $(,)_{q,t}$}

Define an involution $\overline{\phantom{T}}\colon \CC[X] \to \CC[X]$ by
\begin{align*}
\overline{f}(x_1, \ldots, x_n;q,t) &= f(x_1^{-1}, \ldots, x_n^{-1}; q^{-1}, t^{-1}),
\end{align*}
Define
$$
\Delta_\infty^X = \Delta^X_{\infty}(t) = \prod_{1\le i<j\le n} \prod_{r\in \ZZ_{\ge 0}} \frac{1-tq^rx_ix^{-1}_j}{1-q^rx_ix^{-1}_j},
\qquad
\Delta_\infty^{X^{-1}} = \Delta_\infty^{X^{-1}}(t) 
= \prod_{1\le i<j\le n} \prod_{r\in \ZZ_{\ge 0}} \frac{1-tq^rx^{-1}_ix_j}{1-q^rx^{-1}_ix_j},
$$
$$
\Delta_{0}^X = \Delta_{0}^X(t) = \prod_{1\le i< j \le n} \frac{1-tx_ix^{-1}_j}{1-x_ix^{-1}_j}
\qquad\hbox{and}\qquad
\Delta_{0}^{X^{-1}} = \Delta_{0}^{X^{-1}}(t) =  \prod_{1\le i< j \le n} \frac{1-tx^{-1}_ix_j}{1-x^{-1}_ix_j}.
$$
Define a scalar product $(\ ,\ )_{q,t}\colon \CC[X]\times \CC[X] \to \CC(q,t)$ by
\begin{equation}
( f_1,f_2 )_{q,t} = \mathrm{ct}\Big(\frac{f_1 \overline{f_2} }{\Delta_\infty^X \Delta_0^X \Delta_\infty^{X^{-1}} }\Big),
\qquad\hbox{where}\quad
\mathrm{ct}(f) = (\hbox{constant term in $f$}),
\label{innproddefnNEW}
\end{equation}
for $f\in \CC[X]$.
Proposition~\ref{nondegherm} shows that, in a suitable sense, the inner product
$(,)_{q,t}$ is nondegenerate and normalized Hermitian. 

\begin{prop} \label{nondegherm}  
\item[(a)] (sesquilinear) If $f,g\in \CC[X]$ and $c\in \CC[q^{\pm1}]$  then 
$$(cf,g)_{q,t} = c(f,g)_{q,t},\qquad\hbox{and}\qquad 
(f,cg)_{q,t} = \overline{c}(f,g)_{q,t}.
$$
\item[(b)] (nonisotropy) If $f\in \CC[X]$ and $f\ne 0$ then $(f,f)_{q,t} \ne 0$.
\item[(c)] (nondegeneracy) If $F$ is a subspace of $\CC[X]$ and $(,)_F\colon F\times F \to \CC$ is the
restriction of $(,)_{q,t}$ to $F$, then $(,)_F$ is nondegenerate.
\item[(d)] (normalized Hermitian)  If $f_1, f_2\in \CC[X]$ then
$$\frac{(f_2,f_1)_{q,t}}{(1,1)_{q,t}} = \overline{ \Big(\frac{(f_1,f_2)_{q,t}}{(1,1)_{q,t}}\Big)}.$$
\end{prop}
\begin{proof}
\begin{itemize}
\itemsep0.75em
\item[(a)] Let $f_1, f_2\in \CC[X]$ and $c\in \CC[q^{\pm1},t^{\pm1}]$.  Then
$(cf_1, f_2)_{q,t} = \mathrm{ct}(cf_1\overline{f_2}) 
= c\cdot \mathrm{ct}(f_1\overline{f_2}) = c(f,g)_{q,t}$
and
$$(f_1, cf_2)_{q,t} = \mathrm{ct}(f_1\overline{cf_2}) = \mathrm{ct}(f_1\bar c\overline{f_2}) 
= \overline{c} \cdot \mathrm{ct}(f_1\overline{f_2}) = \overline{c}(f,g)_{q,t}.$$
\item[(b)]  Let $f\in \CC[X]$ with $f\ne 0$.  By clearing denominators appropriately, renormalize $f$
so that $f$ specializes to something nonzero at $q=1$.  If 
$$f = \sum_\mu f_\mu x^\mu
\qquad\hbox{then}\qquad
(f,f)_{1,1} = (f,f)_{1,1^k} = \sum_\mu \vert f_\mu\vert^2 \in \RR_{>0}.$$
Thus $(f,f)_t \ne 0$.

\item[(c)] Let $f\in F$ with $f\ne 0$.  Since $(f,f)_{q,t}\ne 0$ then there exists $p\in F$ such that
$(f,p)_{q,t}\ne 0$.  Thus the restriciton of $(,)_{q,t}$ to $F$ is nondegenerate.

\item[(d)] Let
$\displaystyle{f_1 = \sum_\lambda a_\lambda x^\lambda }$ and $\displaystyle{ f_2 = \sum_\mu b_\mu x^\mu}$
and
$$\frac{ \Delta_\infty^X \Delta_0^X \Delta_\infty^{X^{-1}} }{(1,1)_{q,t}}
=\frac{ \Delta_\infty^X \Delta_0^X \Delta_\infty^{X^{-1}} }{ \mathrm{ct}(\Delta_\infty^X \Delta_0^X \Delta_\infty^{X^{-1}} )} 
= \sum_{\mu\in \ZZ^n} d_\mu(q,t) x^\mu.$$
Then
$$\frac{(f_2,f_1)_{q,t}}{(1,1)_{q,t} }
=\sum_{\lambda, \mu\in \ZZ^n} \overline{a_\lambda} b_\mu d_{\lambda-\mu}
=\sum_{\lambda, \mu\in \ZZ^n} \overline{a_\lambda \overline{b_\mu} d_{\mu-\lambda} }
=\overline{ \Big(\frac{(f_1,f_2)_{q,t}}{(1,1)_{q,t}}\Big) }.
$$
\end{itemize}
\end{proof}

\subsection{The inner product characterization of $E_\mu$ and $P_\lambda$}

Recall that the elements of $\ZZ^n$ are partially ordered with the DBlex order given by
$$\lambda\le \mu 
\quad\hbox{if }\qquad
\begin{array}{c}
\hbox{$\lambda^+< \mu^+$ in dominance order} \\
\hbox{or} \\
\hbox{$\lambda^+=\mu^+$ and $z_\lambda<z_\mu$ in Bruhat order.}
\end{array}
$$

\begin{prop}  \label{Ebyinnerprod} Let $\mu \in \ZZ^n$.
The nonsymmetric Macdonald polynomial $E_\mu(q,t)$ is the unique element of
$\CC[X]$ such that
\begin{enumerate}
\item[(a)] $E_\mu(q,t) = x^\mu + (\hbox{lower terms in DBlex order})$;
\item[(b)] If $\nu\in \ZZ^n$ and  $\nu<\mu$ then $(E_\mu(q,t), x^\nu)_{q,t} = 0$.
\end{enumerate}
\end{prop}
\begin{proof}  Let
$V = \hbox{span}\{ x^\mu\ |\ \hbox{$\nu\in \ZZ^n$ and $\nu \le\mu$} \}$,
$$S = \hbox{span}\{ x^\nu\ |\ \hbox{$\nu\in \ZZ^n$ and $\nu<\mu$} \}
\qquad\hbox{and}\qquad
S^\perp = \{ f\in \CC[X] \ |\ \hbox{if $p\in S$ then $(f,p)_{q,t} = 0$} \}.
$$
Since the inner product $(,)_{q,t}$ is nonisotropic then the restriction of $(,)_{q,t}$ to $V$
is nondegenerate and so $\dim(S^\perp) = 1$.  Then the normalization of $E_\mu\in S^\perp$ 
is determined by condition (a).
\end{proof}

For $\gamma\in (\ZZ^n)^+$, define the \emph{monomial symmetric polynomial} $m_\gamma$ by
$$m_\gamma = \sum_{\mu\in S_n \gamma} x^\mu,
\qquad
\hbox{where the sum is over all distinct rearrangements of $\gamma$.}
$$

\begin{prop}  \label{Pbyinnerprod} Let $\lambda\in (\ZZ^n)^+$.
The symmetric Macdonald polynomial $P_\lambda$ is the unique element of
$\CC[X]^{S_n}$ such that
\begin{enumerate}
\item[(a)] $P_\lambda(q,t) = m_\lambda + (\hbox{lower terms in dominance order})$;
\item[(b)] If $\gamma\in (\ZZ^n)^+$ and $\gamma < \lambda$ then $(P_\lambda(q,t), m_\gamma)_{q,t} = 0$.
\end{enumerate}
\end{prop}
\begin{proof}
The proof is completed in the same manner as the proof of Proposition~\ref{Ebyinnerprod}.
\end{proof}

\subsection{Going up a level from $t$ to $qt$}

As in~\eqref{arhoArhodefn} and~\eqref{slicksymmA}, let
$$A_\rho = A_\rho(t) = \prod_{1\le i<j\le n} (x_j-tx_i)
\qquad\hbox{and}\qquad
W_0(t) = \sum_{w\in S_n} t^{\ell(w)}.
$$

\begin{prop}   \label{ktokp1comp}
Let $f,g\in \CC[X]^{S_n}$ so that $f$ and $g$ are symmetric polynomials. Then
\begin{equation*}
(f,g)_{q,qt} 
=\frac{W_0(qt)}{ W_0(t^{-1})} (A_{\rho} f, A_{\rho} g )_{q,t}.
\end{equation*}
\end{prop}
\begin{proof}  
Letting
$$a_\rho(x) = \prod_{i<j} (x_j-x_i)
\quad\hbox{and}\quad
a_\rho(x^{-1}) = \prod_{i<j} (x_i-x_j),
\qquad\hbox{then}\qquad
\Delta_{\infty}^X(qt) = \frac{\Delta_{\infty}^X(t)}{\Delta_{0}^X(qt)a_\rho(x)}.
$$
Then
\begin{align*}
(f,g)_{q,qt} 
&= \mathrm{ct}\Big(\frac{f\overline{g}}{\Delta^X_{\infty}(qt) \Delta^{X^{-1}}_{\infty}(qt) \Delta^X_{0}(qt)}\Big) 
= \mathrm{ct}\Big(f\overline{g}\frac{ \Delta^X_{0}(qt) a_\rho(x) \Delta^{X^{-1}}_{0}(qt) a_\rho(x^{-1}) }
{\Delta^X_{\infty}(t) \Delta^{X^{-1}}_{\infty}(t) \Delta^X_{0}(qt)}\Big) \\
&= \mathrm{ct}\Big(f\overline{g} \frac{  a_\rho(x) a_\rho(x^{-1})  }
{\Delta^X_{\infty}(t) \Delta^{X^{-1}}_{\infty}(t) } \Delta^{X^{-1}}_{0}(qt) \Big) 
= \frac{W_0(qt)}{n!} \mathrm{ct}\Big(f\overline{g}\frac{  a_\rho(x) a_\rho(x^{-1}) }
{\Delta^X_{\infty}(t) \Delta^{X^{-1}}_{\infty}(t) } \Big),
\end{align*}
where the last equality follows from the fact that 
if $H$ is symmetric then
\begin{align*}
\mathrm{ct}(H \Delta^{X^{-1}}_{0}(t))
&= \frac{1}{n!} \mathrm{ct}\Big( \sum_{w\in S_n} w(H \Delta^{X^{-1}}_{0}(t))\Big)
= \frac{1}{n!} \mathrm{ct}\Big( H \sum_{w\in S_n} w(\Delta^{X^{-1}}_{0}(t))\Big) \\
&= \frac{1}{n!} \mathrm{ct}\Big( H W_0(t)\Big)
= \frac{W_0(t)}{n!} \mathrm{ct}(H).
\end{align*}
Similarly, using $A_\rho = \Delta^X_{0}(t)a_\rho(x)$ gives
\begin{align*}
(A_\rho f, &A_\rho g)_{q,t} 
= \mathrm{ct}\Big(f\overline{g} \frac{ \Delta^X_{0}(t)a_\rho(x) \Delta^{X^{-1}}_{0}(t) a_\rho(x^{-1}) }
{\Delta^X_{\infty}(t) \Delta^{X^{-1}}_{\infty}(t) \Delta^X_{0}(t)}\Big) \\
&= \mathrm{ct}\Big(f\overline{g} \frac{ a_\rho(x)a_\rho(x^{-1})   }
{\Delta^X_{\infty}(t) \Delta^{X^{-1}}_{\infty}(t) } \Delta^{X^{-1}}_{0}(t) \Big) 
= \frac{W_0(t)}{n!} \mathrm{ct}\Big(f\overline{g} \frac{ a_\rho(x)a_\rho(x^{-1})  }
{\Delta^X_{\infty}(t) \Delta^{X^{-1}}_{\infty}(t) }\Big).
\end{align*}
Comparing these expressions for 
$(f,g)_{q,qt}$ and $(A_\rho f,A_\rho g)_{q,t}$ gives the statement.
\end{proof}

\subsection{Weyl character formula for Macdonald polynomials}

\begin{thm} \label{WCF}
Let $\lambda\in \ZZ^n$ with $\lambda_1\ge \lambda_2\ge \cdots \ge \lambda_n$.  Then
$$P_\lambda(q,qt) = \frac{A_{\lambda+\rho}(q,t)}{A_\rho(t)}.$$
\end{thm}
\begin{proof}
Since $A_{\lambda+\rho} = t^{\frac12\ell(w_0)}\varepsilon_0 E_{\lambda+\rho}$ then
$A_{\lambda+\rho}\in \CC[X]^{\mathrm{Fer}}$.  Thus, by
Proposition~\ref{BFequalitiesP}, 
$$\hbox{there exists}\qquad f \in \CC[X]^{S_n} \qquad\hbox{such that}\qquad
A_{\lambda+\rho} = A_\rho f.$$
If $\mu\in \ZZ^n$ is such that the coefficient of $x^\mu$ in $A_{\lambda+\rho}$ is nonzero then
$\mu \le w_0(\lambda+\rho)$.  Thus
$$f = m_\lambda + \hbox{(lower terms)}.$$
The $E$-expansion for $A_{\lambda+\rho}$ in Proposition~\ref{Eexpansion} gives that
$$
A_{\lambda+\rho} 
= \sum_{\mu\in S_n(\lambda+\rho)} d_{\lambda+\rho}^\mu E_\mu
= E_{w_0(\lambda+\rho)} + (\hbox{lower terms})
$$
and, from the definitions of $A_\rho$ and $m_\nu$,
$$
A_\rho m_\nu = x^{w_0(\nu+\rho)}+ (\hbox{lower terms}).
$$
Since
$(E_{w_0(\lambda+\rho)}, x^\gamma)_{q,t} = 0$ for $\gamma\in \ZZ^n$ with 
$\gamma < w_0(\lambda+\rho)$, 
then
$$(A_\rho f, A_\rho m_\nu)_{q,t} = (A_{\lambda+\rho}, A_\rho m_\nu)_{q,t} = 0,
\qquad\hbox{for $\nu\in (\ZZ^n)^+$ with $\nu<\lambda$.}
$$
Thus, by~\eqref{ktokp1comp}, since $f\in \CC[X]^{S_n}$ and $m_\nu\in \CC[X]^{S_n}$ then
$$(f, m_\nu)_{q,qt} = \frac{W_0(qt)}{ W_0(t^{-1})} (A_\rho f, A_\rho m_\nu)_{q,t} =0,
\qquad\hbox{for $\nu\in (\ZZ^n)^+$ with $\nu<\lambda$.}
$$
Thus, by Proposition~\ref{Pbyinnerprod}, $f = P_\lambda(q,qt)$.
\end{proof}

\section{Norms and $c$-functions}\label{section:Orth}

The Macdonald polynomials form an incredible family of orthogonal polynomials
with respect to the inner product $(,)_{q,t}$ defined in~\eqref{innproddefnNEW}.
These orthogonal polynomials 
generalize the Askey-Wilson polynomials to the multivariate case and to all affine root systems.
In this section we see how the $c$-functions enter in the formulas for the norms of the Macdonald
polynomials.  The special case of the norm of the polynomial $1$ is the celebrated Macdonald
constant term conjecture~\cite{Mac82}. The generalization~\cite{Mac87} of the Macdonald 
constant term conjectures to conjectures for the norms of the $P_\lambda$ was astounding.  
After much work on special cases,
the methods and tools of Heckman, Opdam, Macdonald, and Cherednik yielded a proof of the norm
conjecture (about 1994, see the Notes and References at the end of~\cite[\S5]{Mac03}).  
This proof is exposited beautifully in~\cite[\S5.8 and \S5.9]{Mac03}.  In this section
we provide a type $GL_n$ specific exposition, highlighting the role of the $c$-functions,
and using them to simplify some steps.  The general framework of our exposition follows the proof
found in~\cite[\S5.8]{Mac03}.

\subsection{Adjoints and orthogonality}

For a linear operator $M\colon \CC[X]\to \CC[X]$, the \emph{adjoint of $M$}
is the linear operator $M^*\colon \CC[X] \to \CC[X]$ determined by 
$$(Mf_1, f_2)_{q,t} = (f_1, M^*f_2)_{q,t},
\qquad\hbox{for $f_1,f_2\in \CC[X]$,}
$$
where the inner product on $\CC[X]$ is as defined in~\eqref{innproddefnNEW}. Also, simplify the notation by letting
$$J_{q,t} = \frac{1}{\Delta_\infty^X \Delta_0^X \Delta_{\infty}^{X^{-1}}}
=\prod_{1\le i<j\le n} 
\frac{(x_ix^{-1}_j;q)_\infty}{(tx_ix^{-1}_j;q)_\infty}\cdot
\frac{(qx^{-1}_ix_j;q)_\infty}{(qtx^{-1}_ix_j;q)_\infty}, \quad \text{ and } \quad \Delta_{ij} = \frac{1-tx_ix^{-1}_j}{1-x_ix^{-1}_j}.
$$

\begin{prop} Let $i\in \{1, \ldots, n\}$ and $k\in \{1, \ldots, n-1\}$.  Then,
as operators on $\CC[X]$,
$$x_i^* = x^{-1}_i, \qquad 
s^*_k =\frac{ \Delta_k,k+1}{\Delta_{k+1,k}} s_k, \qquad
T_\pi^* = T_\pi^{-1}, \qquad
T^*_k = T_k^{-1}, \qquad
Y_i^* = Y_i^{-1}.
$$
\end{prop}
\begin{proof} \hspace{0.3cm}
\begin{itemize}[leftmargin=0.5cm]
\item[$\circ$] Adjoint of multiplication by $x_i$:
$$(x_if,g)_{q,t} = \mathrm{ct}(x_if\overline{g} J_{q,t})
= \mathrm{ct}(f\cdot \overline{x_i^{-1}g}\cdot J_{q,t}) = (f, x_i^{-1}g)_{q,t}.$$

\item[$\circ$] Adjoint of $s_k$:
For $i,j\in \{1, \ldots, n\}$ with $i\ne j$, note that
$$\overline{\Delta_{ij}} = \frac{1-t^{-1}x_i^{-1}x_j}{1-x^{-1}_ix_j} = 
\frac{tx_ix^{-1}_j - 1}{t(x_ix^{-1}_j-1)} = t^{-1}\Delta_{ij}.
$$
Therefore
\begin{align*}
s_k J_{q,t} &= s_k\Big(\frac{1}{\Delta_\infty^X\Delta_0^X \Delta_\infty^{X^{-1}} }\Big)
=
s_k\Big(\frac{1}{\Delta_\infty^X\Delta_0^X \Delta_0^{X^{-1}}\Delta_\infty^{X^{-1}} }\cdot \Delta_0^{X^{-1}}\Big)
= \frac{1}{\Delta_\infty^X\Delta_0^X \Delta_0^{X^{-1}}\Delta_\infty^{X^{-1}} }\cdot s_k(\Delta_0^{X^{-1}}) \\
&=
\frac{1}{\Delta_\infty^X\Delta_0^X \Delta_0^{X^{-1}}\Delta_\infty^{X^{-1}} }\cdot \Delta_0^{X^{-1}}
\frac{\Delta_{k,k+1}}{\Delta_{k+1,k}}
=
\Big(\frac{1}{\Delta_\infty^X\Delta_0^X \Delta_\infty^{X^{-1}}}\Big)
\frac{\Delta_{k,k+1}}{\Delta_{k+1,k}}
= J_{q,t} \frac{\Delta_{k,k+1}}{\Delta_{k+1,k}}.
\end{align*}
Then
\begin{align*}
(s_kf_1,f_2)_{q,t}
&= \mathrm{ct}\big( (s_kf_1)\overline{f_2} J_{q,t} \big)
= \mathrm{ct}\big(s_k(f_1(s_k(\overline{f_2} J_{q,t}))) \big)
= \mathrm{ct}\big( f_1(s_k(f_2 J_{q,t})) \big)
\\
&= \mathrm{ct} \Big(f_1(s_k\overline{f_2}) J_{q,t}\frac{\Delta_{k,k+1}}{\Delta_{k+1,k}} \Big)
= \mathrm{ct} \Big( f_1 \overline{\frac{\Delta_{k,k+1}}{\Delta_{k+1,k}} (s_kf_2)} J_{q,t} \Big)
=(f_1, \frac{\Delta_{k,k+1}}{\Delta_{k+1,k}}(s_kf_2))_{q,t}.
\end{align*}

\item[$\circ$] Adjoint of $T_\pi$:  Recall that the action of $T_\pi$ is given by 
$T_\pi = s_1s_2\cdots s_{n-1}y_n$ and so
$$T_\pi x_n = q^{-1}x_1,
\qquad\hbox{and}\qquad
T_\pi x_i = x_{i+1}\ \hbox{for $i\in \{1, \ldots, n-1\}$.}
$$
Thus, if $i\in \{1,\ldots, n-1\}$ then $T_\pi x_ix^{-1}_n = qx_{i+1}x^{-1}_1$ and
$T_\pi qx^{-1}_ix_n = qq^{-1}x^{-1}_{i+1}x_1$.  Therefore
\begin{align*}
T_\pi J_{q,t} 
&= T_\pi 
\Big(\prod_{1\le i<j\le n-1} 
\frac{(x_ix^{-1}_j;q)_\infty}{(tx_ix^{-1}_j;q)_\infty}\cdot
\frac{(qx^{-1}_ix_j;q)_\infty}{(qtx^{-1}_ix_j;q)_\infty}\Big)
\Big( \prod_{i=1}^{n-1} 
\frac{(x_ix^{-1}_n;q)_\infty}{(tx_ix^{-1}_n;q)_\infty} \cdot
\frac{(qx^{-1}_ix_n;q)_\infty}{(qtx^{-1}_ix_n;q)_\infty} \Big)
\\
&=
\Big(\prod_{2\le i<j\le n} 
\frac{(x_ix^{-1}_j;q)_\infty}{(tx_ix^{-1}_j;q)_\infty}\cdot
\frac{(qx^{-1}_ix_j;q)_\infty}{(qtx^{-1}_ix_j;q)_\infty}\Big)
\Big( \prod_{i=2}^{n} 
\frac{(qx_ix^{-1}_1;q)_\infty}{(qtx_ix^{-1}_1;q)_\infty} \cdot
\frac{(qq^{-1}x^{-1}_ix_1;q)_\infty}{(qtq^{-1}x^{-1}_ix_1;q)_\infty} \Big)
\\
&=
\Big(\prod_{2\le i<j\le n} 
\frac{(x_ix^{-1}_j;q)_\infty}{(tx_ix^{-1}_j;q)_\infty}\cdot
\frac{(qx^{-1}_ix_j;q)_\infty}{(qtx^{-1}_ix_j;q)_\infty}\Big)
\Big( \prod_{i=2}^{n} 
\frac{(qx^{-1}_1x_i;q)_\infty}{(qtx^{-1}_1x_i;q)_\infty} \cdot
\frac{(x_1x^{-1}_i;q)_\infty}{(tx_1x^{-1}_i;q)_\infty} \Big)
= J_{q,t}.
\end{align*}
Thus
\begin{align*}
(T_\pi f_1, f_2)_{q,t} &= \mathrm{ct}\big( (T_\pi f_1) \overline{f_2} J_{q,t} \big)
= \mathrm{ct}\big( T_\pi (f_1 T_{\pi^{-1}}(\overline{f_2} J_{q,t}))\big)
= \mathrm{ct}\big( f_1 T_{\pi^{-1}}(\overline{f_2} J_{q,t})\big)
\\
&= \mathrm{ct}\big( f_1\big( T_{\pi^{-1}}(\overline{f_2}) J_{q,t} \big)\big)
= \mathrm{ct}\big( f_1\cdot \overline{T_{\pi^{-1}}f_2}\cdot J_{q,t} \big)
=(f_1, T_{\pi^{-1}}f_2)_{q,t}.
\end{align*}

\item[$\circ$] Adjoint of $T_k$:
\begin{align*}
(T_{s_k})^*
&= \big(-t^{-\frac12}+(1+s_k)t^{-\frac12}\Delta_{k+1,k})\big)^* 
= -t^{\frac12} + t^{\frac12} \overline{\Delta_{k+1,k}} (1+s_k^*) 
\\
&= -t^{\frac12} + t^{-\frac12} \Delta_{k+1,k} \Big(1+\frac{\Delta_{k,k+1}}{\Delta_{k+1,k}} s_k\Big) 
= -t^{\frac12} + (1+s_k)t^{-\frac12}\Delta_{k+1,k} = T^{-1}_{s_k}.
\end{align*}
\item[$\circ$] Adjoint of $Y_j$: 
$$Y_1^* = (T_\pi T_{n-1}\cdots T_1)^* = T_1^{-1}\cdots T^{-1}_{n-1}T_\pi^{-1} = 
(T_\pi T_{n-1}\cdots T_1)^{-1} = Y^{-1}_1,$$
and
if $j\in \{2, \ldots, n\}$ then
$$Y_j^* = (T^{-1}_{j-1} Y_{j-1} T^{-1}_{j-1})^* = T_{j-1}Y^{-1}_{j-1} T_{j-1} = 
(T^{-1}_{j-1} Y_{j-1} T^{-1}_{j-1})^{-1} = Y^{-1}_j.$$

\end{itemize}
\end{proof}

Next, we look at the adjoint operator of the bosonic and fermionic symmetrizers. Since $T_i^{-1}\mathbf{1}^*_0 = T_i^*\mathbf{1}_0^* = (\mathbf{1}_0T_i)^* = (t^{\frac12}\mathbf{1}_0)^*
=t^{-\frac12}\mathbf{1}_0$ and
\begin{equation*}
\mathbf{1}_0^* = T_{w_0}^{-1}+ (\hbox{lower terms})
= T_{w_0} + (\hbox{lower terms})
\qquad\hbox{then}\qquad
\mathbf{1}_0^* = \mathbf{1}_0.
\end{equation*}
Similarly, since $T^{-1}_i\varepsilon_0^* = T_i^* \varepsilon_0^* = (\varepsilon_0T_i)^* 
= (-t^{-\frac12}\varepsilon_0)^* = -t^{\frac12}\varepsilon_0^*$ and
\begin{equation*}
\varepsilon_0^* = T_{w_0}^{-1} + (\hbox{lower terms})
= T_{w_0} + (\hbox{lower terms})
\qquad\hbox{then}\qquad
\varepsilon_0^* = \varepsilon_0.
\end{equation*}

The relations $Y_i^*= Y^{-1}_i$ in combination with the knowledge of the eigenvalues for
the action of the $Y_i$ on the $E_\mu$ give the following orthogonality relations
for Macdonald polynomials.

\begin{prop}
\item[(a)] Let $\lambda, \mu\in \ZZ^n$.  If $\mu \ne \lambda$ then 
$(E_\lambda, E_\mu)_{q,t} = 0$.
\item[(b)] Let $\lambda, \mu\in (\ZZ^n)^+$.  If $\mu \ne \lambda$ then
$(P_\lambda, P_\mu)_{q,t} = 0$.
\item[(c)] Let $\lambda, \mu\in (\ZZ^n)^+$.  If $\mu \ne \lambda$ then
$(A_{\lambda+\rho}, A_{\mu+\rho})_{q,t} = 0$.
\end{prop}
\begin{proof} Let $i\in \{1, \ldots, n\}$.  Then, by Theorem~\ref{Eeigenvalue},
\begin{align*}
q^{-\lambda_i}t^{-(v_\lambda(i)-1)+\frac12(n-1)} &(E_\lambda, E_\mu)_{q,t}
= (Y_i E_\lambda, E_\mu)_{q,t}
= ( E_\lambda, Y^{-1}_i E_\mu)_{q,t}
= ( E_\lambda, q^{\mu_i}t^{(v_\mu(i)-1)\frac12(n-1)} E_\mu)_{q,t}
\\
&= \overline{q^{\mu_i}t^{(v_\mu(i)-1)-\frac12(n-1)} }
( E_\lambda,  E_\mu)_{q,t}
= q^{-\mu_i}t^{-(v_\mu(i)-1)+\frac12(n-1)} 
( E_\lambda,  E_\mu)_{q,t}.
\end{align*}
If $( E_\lambda,  E_\mu)_{q,t}\ne 0$ then $q^{-\lambda_i} = q^{-\mu_i}$ for $i\in \{1, \ldots, n\}$.
Thus $\lambda_i = \mu_i$ for $i\in \{1, \ldots, n\}$ and so $\lambda = \mu$ (and $v_\lambda = v_\mu$).

\smallskip\noindent
Parts (b) and (c) follow from (a) and the $E$-expansions in Proposition~\ref{Eexpansion}.
\end{proof}

\subsection{Reductions for norms}

\begin{prop} \label{normreduction} 
\item[(a)] Let $\mu = (\mu_1, \ldots, \mu_n)\in \ZZ^n$.
Then
$$\frac{(E_\mu, E_\mu)_{q,t}}{(1,1)_{q,t}} = \ev^t_0(c_{u_\mu}^Y c_{u_\mu}^{Y^{-1}}).$$
\item[(b)] Let $\lambda \in \left(\ZZ^n\right)^+$ 
and let $\rho = (n-1, n-2, \ldots, 1,0)$.
Then
$$
\frac{(P_\lambda, P_\lambda)_{q,t} }{(E_\lambda, E_\lambda)_{q,t} }
= \frac{W_0(t)}{W_\lambda(t)} 
t^{-\frac12\ell(v_\lambda)} \ev^t_\lambda(c_{v_\lambda}^Y)
\qquad\hbox{and}\qquad
\frac{(A_{\lambda+\rho}, A_{\lambda+\rho})_{q,t}  }{(P_{\lambda+\rho}, P_{\lambda+\rho})_{q,t} }
= \ev^t_{\lambda+\rho}\Big (\frac{ c^{Y^{-1}}_{w_0} }{ c_{w_0}^Y}\Big).
$$
\item[(c)] Let $\lambda \in \left(\ZZ^n\right)^+$. 
Then
$$\frac{ (P_\lambda(q,qt), P_\lambda(q,qt))_{q,qt} }{ (P_{\lambda+\rho}(q,t), P_{\lambda+\rho}(q,t))_{q,t} }
= \frac{W_0(qt)}{W_0(t)} \ev^t_{\lambda+\rho}\Big( t^{\ell(w_0)} \frac{ c_{w_0}^{Y^{-1}} } {c_{w_0}^Y }\Big).
$$
\end{prop}
\begin{proof}
\begin{itemize}
\itemsep0.75em
\item[(a)] Recalling the intertwiners $\tau^\vee_i$ from~\eqref{intwnrops} (see also the
proof of Proposition~\ref{normintwn}),
$$(\tau^\vee_i)^* = \Big(T_i+\frac{t^{-\frac12}-t^{\frac12}}{1-Y^{-1}_iY_{i+1}}\Big)^*
= T_i^* + \frac{t^{\frac12}-t^{-\frac12}}{1-Y_iY^{-1}_{i+1}}
= T_i^{-1} + \frac{(t^{-\frac12}-t^{\frac12})Y^{-1}_iY_{i+1} }{1-Y^{-1}_iY_{i+1}}
=\tau^\vee_i,
$$
and
$$(\tau^\vee_i)^2 = \frac{(1-tY_iY^{-1}_{i+1})(1-tY^{-1}_iY_{i+1})}
{(1-Y_iY^{-1}_{i+1}) (1-Y^{-1}Y_{i+1})} = c_{i,i+1}^Y c_{i,i+1}^{Y^{-1}}.
$$
Then, using the creation formula for $E_\mu$ (Theorem~\ref{creationformulathm}),
\begin{align*}
(E_\mu, E_\mu)_{q,t}
&= (t^{-\frac12\ell(v^{-1}_\mu)}\tau^\vee_{u_\mu}\mathbf{1}_Y, 
t^{-\frac12\ell(v^{-1}_\mu)}\tau^\vee_{u_\mu}\mathbf{1}_Y)_{q,t}
= (\tau^\vee_{u^{-1}_\mu} \tau^\vee_{u_\mu}\mathbf{1}_Y, \mathbf{1}_Y)_{q,t}
\\
&= ( c_{u_\mu}^Y c_{u_\mu}^{Y^{-1}} \mathbf{1}_Y, \mathbf{1}_Y)_{q,t}
= \ev^t_0(c_{u_\mu}^Y c_{u_\mu}^{Y^{-1}}) \cdot ( 1,1)_{q,t}.
\end{align*}

\item[(b)]
Note that
$$
\mathbf{1}_0^2 
= t^{-\frac12\ell(w_0)}W_0(t) \mathbf{1}_0
\qquad\hbox{and}\qquad
\varepsilon_0^2 = (-1)^{\ell(w_0)} t^{-\frac12\ell(w_0)}W_0(t)\varepsilon_0.
$$
By Proposition~\ref{Eexpansion},
\begin{align*}
P_\lambda &= \sum_{\mu\in S_n\lambda} b_\lambda^\mu E_\mu,
&&\hbox{with}&&b_\lambda^\lambda = t^{\frac12\ell(v_\lambda)}\ev^t_\lambda(c_{v_\lambda}^Y),
\\
A_{\lambda+\rho} &= \sum_{\mu\in S_n(\lambda+\rho)} d_{\lambda+\rho}^\mu E_\mu,
&&\hbox{with} &&d_{\lambda+\rho}^{\lambda+\rho} 
= (-t^{\frac12})^{\ell(w_0)} \ev^t_{\lambda+\rho}(c_{w_0}^{Y^{-1}}).
\end{align*}
Since $W_\lambda(t^{-1}) = t^{-\ell(w_\lambda)}W_\lambda(t)$ and 
$\ell(w_0)-\ell(w_\lambda) = \ell(v_\lambda)$ then using $P_\lambda = \frac{t^{\frac12\ell(w_0)}}{W_\lambda(t)} \mathbf{1}_0 E_\lambda$ gives
\begin{align*}
(P_\lambda, P_\lambda)_{q,t} 
&= \Big( \frac{t^{\frac12\ell(w_0)}}{W_\lambda(t)} \mathbf{1}_0 E_\lambda,
\frac{t^{\frac12\ell(w_0)}}{W_\lambda(t)}  \mathbf{1}_0 E_\lambda\Big)_{q,t} 
= \frac{1}{W_\lambda(t)W_\lambda(t^{-1})}
(  \mathbf{1}_0^2 E_\lambda, E_\lambda )_{q,t}
\\
&= \frac{t^{-\frac12 \ell(w_0)} W_0(t)}{W_\lambda(t)W_\lambda(t^{-1})}
(  \mathbf{1}_0 E_\lambda, E_\lambda )_{q,t}
= \frac{t^{-\ell(w_0)} W_0(t)}{W_\lambda(t^{-1})}
(  P_\lambda, E_\lambda )_{q,t}
\\
&= \frac{t^{-\ell(w_0)} W_0(t)}{t^{-\ell(w_\lambda)}W_\lambda(t)} b_\lambda^\lambda
(  E_\lambda, E_\lambda )_{q,t}
= \frac{t^{-\ell(v_\lambda)} W_0(t)}{W_\lambda(t)} t^{\frac12\ell(v_\lambda)} 
\ev^t_\lambda(c_{v_\lambda}^Y) (  E_\lambda, E_\lambda )_{q,t}.
\end{align*}
Similarly, using $A_{\lambda+\rho} = t^{\frac12\ell(w_0)}\varepsilon_0 E_{\lambda+\rho}$ gives
\begin{align*}
(A_{\lambda+\rho}, &\ A_{\lambda+\rho} )_{q,t}
= \big( t^{\frac12\ell(w_0)} \varepsilon_0 E_{\lambda+\rho}, 
\ t^{\frac12\ell(w_0)} \varepsilon_0 E_{\lambda+\rho} \big)_{q,t}
=  \big( \varepsilon^2_0 E_{\lambda+\rho}, \ E_{\lambda+\rho} \big)_{q,t} 
\\
&= (-1)^{\ell(w_0)} t^{-\frac12\ell(w_0)}W_0(t)
\big( \varepsilon_0 E_{\lambda+\rho}, \ E_{\lambda+\rho} \big)_{q,t} 
= (-1)^{\ell(w_0)} t^{-\ell(w_0)}W_0(t)
\big( A_{\lambda+\rho}, \ E_{\lambda+\rho} \big)_{q,t} 
\\
&
= (-1)^{\ell(w_0)} W_0(t^{-1}) d_{\lambda+\rho}^{\lambda+\rho}
\big( E_{\lambda+\rho}, \ E_{\lambda+\rho} \big)_{q,t}
= W_0(t^{-1}) t^{\frac12\ell(w_0)} \ev^t_{\lambda+\rho}(c_{w_0}^{Y^{-1}})
\big( E_{\lambda+\rho}, \ E_{\lambda+\rho} \big)_{q,t}.
\end{align*}
Using
$(P_{\lambda+\rho},P_{\lambda+\rho})_{q,t} = W_0(t)t^{-\frac12\ell(w_0)} \ev^t_{\lambda+\rho} (c_{w_0}^Y)
(E_{\lambda+\rho}, E_{\lambda+\rho})_{q,t}$ gives
$$\frac{ (A_{\lambda+\rho}, A_{\lambda+\rho})_{q,t} }{(P_{\lambda+\rho}, P_{\lambda+\rho})_{q,t} }
= \frac{ W_0(t^{-1})t^{\frac12\ell(w_0)} \ev^t_{\lambda+\rho}(c^{Y^{-1}}_{w_0})}
{W_0(t) t^{-\frac12\ell(w_0)} \ev^t_{\lambda+\rho} (c_{w_0}^Y)} 
= \ev^t_{\lambda+\rho}\Big (\frac{ c^{Y^{-1}}_{w_0} }{ c_{w_0}^Y}\Big).
$$
\item[(c)] Using Proposition~\ref{ktokp1comp} and the Weyl character formula 
(Theorem~\ref{WCF}),
\begin{align*}
&(P_\lambda(q,qt), P_\lambda(q,qt))_{q,qt} 
= \frac{W_0(qt)}{W_0(t^{-1})} (A_\rho(t) P_\lambda(q,qt), A_\rho(t) P_\lambda(q,qt))_{q,t}
\\
&\quad = \frac{W_0(qt)}{t^{-\ell(w_0)}W_0(t)} (A_{\lambda+\rho}(q,t), A_{\lambda+\rho}(q,t))_{q,t}
= \frac{W_0(qt)}{W_0(t)} \ev^t_{\lambda+\rho}\Big( t^{\ell(w_0)}\frac{ c_{w_0}^{Y^{-1}} } {c_{w_0}^Y} \Big)
(P_{\lambda+\rho}(q,t), P_{\lambda+\rho}(q,t))_{q,t},
\end{align*}
where the last equality follows from the second formula in (b). 
\end{itemize}
\end{proof}

Evaluating the $c$-functions appearing in Proposition~\ref{normreduction} gives the following corollary.

\begin{cor}
$$(E_\mu, E_\mu)_{q,t} 
= \Big(\prod_{(r,c)\in \mu} \prod_{i=1}^{u_\mu(r,c)}
\frac{(1-q^{\mu_r-c+1} t^{v_\mu(r)-i+1})
(1-q^{\mu_r-c+1}t^{v_\mu(r)-i-1})}{(1-q^{\mu_r-c+1} t^{v_\mu(r)-i})^2}\Big)\cdot (1,1)_{q,t},
$$
$$(P_\lambda, P_\lambda)_{q,t}  = \frac{W_0(t)}{W_\lambda(t)} \Big(\prod_{i<j} 
\frac{1-q^{\lambda_i-\lambda_j}t^{j-i-1}} {1-q^{\lambda_i-\lambda_j}t^{j-i}}
\Big)\cdot (E_\lambda, E_\lambda)_{q,t}.
$$
$$(A_{\lambda+\rho}, A_{\lambda+\rho})_{q,t}  = W_0(t^{-1}) 
\Big(\prod_{i<j} \frac{1-q^{\lambda_i-\lambda_j+j-i} t^{j-i+1}} {1-q^{\lambda_i-\lambda_j+j-i}t^{j-i}}
\Big)\cdot (E_{\lambda+\rho}, E_{\lambda+\rho})_{q,t}.
$$
$$
\frac{(P_\lambda(q,qt), P_\lambda(q,qt))_{q,qt} }{(P_{\lambda+\rho}(q,t), P_{\lambda+\rho}(q,t))_{q,t}}
= \frac{W_0(qt)}{W_0(t)}\Big(\prod_{i<j} \frac{1-q^{\lambda_i-\lambda_j+j-i}t^{j-i+1}}
{1-q^{\lambda_i-\lambda_j+j-i}t^{j-i-1}}\Big).
$$
\end{cor}
\begin{proof}
Using that
$$c^Y_{ij}c^{Y^{-1}}_{ij}
= \Big(\frac{t^{-\frac12}-t^{\frac12}Y_iY^{-1}_j}{1-Y_iY^{-1}_j}\Big)
\Big(\frac{t^{-\frac12}-t^{\frac12}Y^{-1}_iY_j}{1-Y^{-1}_iY_j}\Big)
=\frac{(1-tY_iY^{-1}_j)(1-t^{-1}Y_iY^{-1}_j)}{(1-Y_iY^{-1}_j)^2},
$$
the formula for $(E_\mu, E_\mu)_{q,t}$ follows from Proposition~\ref{normreduction}(a)
and the computation in~\eqref{cumueval}.
The formula for $(P_\lambda, P_\lambda)_{q,t}$ follows from Proposition~\ref{normreduction}(b)
and the computation in~\eqref{Eexpb}.
The formula for $(P_\lambda, P_\lambda)_{q,t}$ follows from Proposition~\ref{normreduction}(b)
and the computation in~\eqref{Eexpd}.

Using 
$\ev^t_\mu(Y_i) = q^{-\mu_i}t^{-(v_\mu(i)-1) + \frac12(n-1)}$ gives
$$
\ev^t_\mu(Y^{-1}_iY_j) 
= q^{\mu_i}t^{(v_\mu(i)-1) - \frac12(n-1)}
q^{-\mu_j}t^{-(v_\mu(j)-1) + \frac12(n-1)}
= q^{\mu_i-\mu_j}t^{v_\mu(i)-v_\mu(j)},
$$
for $i,j\in \{1, \ldots, n\}$ with $i\ne j$.
Then
\begin{equation*}
t\frac{c_{ij}^{Y^{-1}}}{c_{ij}^Y }
= t \frac{\frac{t^{-\frac12}-t^{\frac12}Y^{-1}_iY_j }{1-Y^{-1}_iY_j }}
{\frac{t^{-\frac12} -t^{\frac12}Y_iY^{-1}_j}
{1-Y_iY^{-1}_j }}
= t \frac{\frac{1-tY^{-1}_iY_j }{1-Y^{-1}_iY_j }}
{\frac{Y^{-1}_iY_j- t}
{Y^{-1}_iY_j-1 }}
=  \frac{1-tY^{-1}_iY_j }{t^{-1}(t-Y^{-1}_iY_j)}
=  \frac{1-tY^{-1}_iY_j }{1-t^{-1}Y^{-1}_iY_j}
\end{equation*}
Using that $w_{\lambda+\rho} = w_0$ and $w_0(i) = n-i+1$ then
\begin{equation}
\ev^t_{\lambda+\rho}\Big(t^{\ell(w_0)} \frac{c_{w_0}^{Y^{-1}}}{c_{w_0}^{Y}} \Big)
= 
\ev^t_{\lambda+\rho}\Big(\prod_{i<j} \frac{1-tY^{-1}_iY_j}{1-t^{-1}Y^{-1}_iY_j} \Big)
=
\prod_{i<j} \frac{1-tq^{\lambda_i-\lambda_j+j-i}t^{j-i} }
{1-t^{-1}q^{\lambda_i-\lambda_j+j-i}t^{j-i} }.
\label{receval}
\end{equation}
In view of Proposition~\ref{normreduction}(c), this proves the formula for 
$\dfrac{(P_\lambda(q,qt), P_\lambda(q,qt))_{q,qt} }{(P_{\lambda+\rho}(q,t), P_{\lambda+\rho}(q,t))_{q,t}}$.
\end{proof}

\subsection{Formulas for norms and the constant term}

\begin{thm}  \label{Normformula}
Let $\lambda\in \left(\ZZ^n\right)^+$ 
and let $k\in \ZZ_{\ge0}$.  Then
$$
(P_\lambda(q,q^k), P_\lambda(q,q^k) )_{q,q^k}
= W_0(q^k)
\prod_{r=1}^{k-1} \ev^{q^r}_{\lambda+(k-r)\rho}
\Big(q^{r\cdot\ell(w_0)} \frac{c^{Y^{-1}}_{w_0}(q^r)}{c^Y_{w_0}(q^r)}\Big).
$$
Alternatively,
$$
(P_\lambda(q,q^k), P_\lambda(q,q^k) )_{q,q^k}
= W_0(q^k)\cdot \prod_{i<j} \prod_{r=1}^{k-1} 
\frac{1-q^{\lambda_i-\lambda_j+r} q^{k(j-i)}}
{1-q^{\lambda_i-\lambda_j-r} q^{k(j-i)}}.
$$
\end{thm}
\begin{proof}  The proof is by induction on $k$.  The base case is $k=0$, where
$$
(P_\lambda(q,q^0),P_\lambda(q,q^0))_{q,q^0}
= (m_\lambda,m_\lambda)_{q,1} = W_\lambda(1).
$$
Using Proposition~\ref{normreduction}(c), the first step of the induction is
\begin{align*}
(P_\lambda(q,q),&P_\lambda(q,q))_{q,q}
=
\frac{(P_\lambda(q,q),P_\lambda(q,q))_{q,q}}
{(P_{\lambda+\rho}(q,q^0),P_{\lambda+\rho}(q,q^0))_{q,q^0}} 
(P_{\lambda+\rho}(q,q^0),P_{\lambda+\rho}(q,q^0))_{q,q^0}
\\
&= \frac{W_0(q)}{W_0(1)}\ev_{\lambda+\rho}^1\Big(1^{\ell(w_0)}\frac{c^{Y^{-1}}_{w_0}(q)}{c^Y_{w_0}(q)}\Big)
(m_{\lambda+\rho},m_{\lambda+\rho})_{q,1}
= \frac{W_0(q)}{W_0(1)}\cdot 1\cdot W_0(1) = W_0(q),
\end{align*}
and the general induction step is
\begin{align*}
(P_\lambda(q,q^k),&P_\lambda(q,q^k))_{q,q^k}
=
\frac{(P_\lambda(q,q^k),P_\lambda(q,q^k))_{q,q^k}}
{(P_{\lambda+\rho}(q,q^{k-1}),P_{\lambda+\rho}(q,q^{k-1}))_{q,q^{k-1}}}
(P_{\lambda+\rho}(q,q^{k-1}),P_{\lambda+\rho}(q,q^{k-1}))_{q,q^{k-1}}
\\
&= \frac{W_0(q^k)}{W_0(q^{k-1})} \ev_{\lambda+\rho}^{q^{k-1}}
\Big( q^{(k-1)\ell(w_0)} \frac{c_{w_0}^{Y^{-1}}(q)}{c_{w_0}^Y(q) }\Big)
W_0(q^{k-1}) \prod_{r=1}^{k-2} \ev^{q^r}_{\lambda+(k-1-r)\rho}
\Big(q^{r\cdot\ell(w_0)} \frac{c^{Y^{-1}}_{w_0}}{c^Y_{w_0}}\Big)
\\
&= W_0(q^k)
\prod_{r=1}^{k-1} \ev^{q^r}_{\lambda+(k-r)\rho}
\Big(q^{r\cdot\ell(w_0)} \frac{c^{Y^{-1}}_{w_0}}{c^Y_{w_0}}\Big).
\end{align*}
In a similar manner to the computation in~\eqref{receval},
$$\ev^{q^r}_{\lambda+(k-r)\rho}\Big(q^{r\cdot \ell(w_0)}\frac{c_{w_0}^Y}{c_{w_0}^{Y^{-1}}}\Big) 
=\prod_{i<j} \frac{1-q^r q^{\lambda_i-\lambda_j+(k-r)(j-i)}q^{r(j-i)}}
{1-q^{-r}q^{\lambda_i-\lambda_j+(k-r)(j-i)}q^{r(j-i)} }
=\prod_{i<j} \frac{1-q^{\lambda_i-\lambda_j+r}q^{k(j-i)}}
{1-q^{\lambda_i-\lambda_j-r}q^{k(j-i)}}.
$$
\end{proof}

Specializing Proposition~\ref{Normformula} 
at $\lambda=0$ provides a proof of Macdonald's constant term conjectures.
This proof follows the same framework as the proof exposited in~\cite[\S5.8]{Mac03}.

\begin{prop}  \label{Norm1}
Let $k\in \ZZ_{\ge0}$ and let $t=q^k$. Then
$$
( 1,1)_{q,q^k} = \mathrm{ct}\left(\frac{1}{\Delta^X_\infty \Delta^X_0 \Delta^{X^{-1}}_{\infty} }\right)
=\prod_{i=2}^n \genfrac[]{0pt}{0}{ik}{k}
$$
and
$$
\frac{1}{W_0(q^k)}(1, 1)_{q,q^k} 
= \mathrm{ct}\left(\frac{1}{\Delta^X_\infty \Delta^X_0 \Delta^{X^{-1}}_0 \Delta^{X^{-1}}_{\infty} }\right)
= \prod_{h=2}^{n-1} \genfrac[]{0pt}{0}{hk-1}{k-1}.
$$
\end{prop}
\begin{proof}  The first equality is the definition of the inner product $(1,1)_{q,q^k}$.  Then
\begin{align*}
\prod_{i<j} &\prod_{r=1}^{k-1} 
\frac{1-q^r q^{k(j-i)}}
{1-q^{-r} q^{k(j-i)} }
=\prod_{h=1}^{n-1} \prod_{i<j \atop j-i = h} 
\frac{(1-q^{kh+1})\cdots (1-q^{kh+k-1})}{(1-q^{kh-1})\cdots (1-q^{kh-(k-1)})}
\\
&=\prod_{h=1}^{n-1} 
\frac{ \big( (1-q^{kh+1})\cdots (1-q^{kh+(k-1)}) \big)^{n-h} }
{\big( (1-q^{k(h-1)+1})\cdots (1-q^{k(h-1)+(k-1)})\big)^{n-h} }
\\
&= \prod_{h=1}^{n-1} 
\frac{ \big( (1-q^{kh+1})\cdots (1-q^{kh+(k-1)}) \big)^{n-h} }
{\big( (1-q^{k(h-1)+1})\cdots (1-q^{k(h-1)+(k-1)})\big)^{n-(h-1)} }
\cdot
(1-q^{k(h-1)+1})\cdots (1-q^{k(h-1)+(k-1)})
\\
&=
\frac{\big((1-q^{k(n-1)+1})\cdots (1-q^{k(n-1)+(k-1)})\big)^{n-(n-1)} }
{\big( (1-q)\cdots (1-q^{k-1})\big)^{n-(1-1)} }
\Big(\prod_{h=1}^{n-1}  (1-q^{k(h-1)+1})\cdots (1-q^{k(h-1)+(k-1)}) \Big)
\\
&=\Big( 
\frac{(1-q^{kn-(k-1)})\cdots (1-q^{kn-1)}) }
{(q;q)_{k-1}}\Big)
\prod_{h=1}^{n-1}  \frac{(1-q^{kh-(k-1)})\cdots (1-q^{kh-1})} {(q;q)_{k-1}}
\\
&= \prod_{h=2}^{n-1} \genfrac[]{0pt}{0}{hk-1}{k-1}.
\end{align*}
Then, using $1 = P_0(q,q^k)$ and Theorem~\ref{Normformula},
\begin{align*}
(1,1)_{q,q^k} 
&=W_0(q^k)
\prod_{i=2}^n \genfrac[]{0pt}{0}{ik-1}{k-1}
= \Big(\prod_{i=2}^n \frac{1-q^{ik} }{1-q^k}\Big) 
\prod_{i=2}^n \genfrac[]{0pt}{0}{ik-1}{k-1}
=\prod_{i=2}^n \genfrac[]{0pt}{0}{ik}{k}.
\end{align*}
\end{proof}

\begin{remark} \underline{Converting to general $q$ and $t$.}
Define
$$\Delta^Y_\infty(t) 
= \prod_{1\le i<j\le n} \prod_{r=1}^\infty \frac{1-tq^rY_iY^{-1}_j}{1-q^rY_iY^{-1}_j},
\qquad
\Delta^{Y^{-1}}_\infty(t^{-1}) 
= \prod_{1\le i<j\le n} \prod_{r=1}^\infty \frac{1-t^{-1}q^rY^{-1}_iY_j}{1-q^rY^{-1}_iY_j},
$$
$$
\Delta_0^Y(t)
= \prod_{1\le i<j\le n}  \frac{1-tY_iY^{-1}_j}{1-Y_iY^{-1}_j}.
$$
Using the notation $(z;q)_\infty = (1-z)(1-qz)(1-q^2z)\cdots$, then
$$\Delta^Y_\infty(t) = \prod_{i<j} \frac{(qtY_iY^{-1}_j;)_\infty}{(qY_iY^{-1}_j)_\infty}
\qquad\hbox{and}\qquad
\Delta^{Y^{-1}}_\infty(t^{-1}) = \prod_{i<j} \frac{(qt^{-1}Y^{-1}_iY_j;)_\infty}{(qY^{-1}_iY_j)_\infty}.
$$
Define an evaluation homomorphism $\ev_{q^\lambda t^\rho}\colon \CC[Y]\to \CC$ by
$$\ev_{q^\lambda t^\rho}(Y_i) = q^{\lambda_i}t^{n-j},
\qquad\hbox{so that}\qquad
\ev_{q^\lambda t^\rho}(Y^{-1}_iY_j) = q^{\lambda_i-\lambda_j}t^{j-i}.
$$ 

Let $i,j\in \{1, \ldots, n\}$ with $i<j$ and let $t = q^k$.  
Using
$$(x;q)_\infty = (1-x)(qx;q)_\infty, \qquad\qquad
(x;q)_k = \frac{(x;q)_\infty}{(q^kx;q)_\infty},
$$
and
$$
(x;q^{-1})_k = (q^{-(k-1)}x;q)_k = (q^{-k}qx;q)_k = (t^{-1}qx;q),
$$
then
\begin{align*}
\prod_{r=1}^{k-1} \frac{1-xq^r}{1-xq^{-r}}
&=\prod_{r=0}^{k-1} \frac{1-xq^r}{1-xq^{-r}}
= \frac{(x;q)_k}{(x;q^{-1})_k}
=\frac{(x;q)_k}{q^{-k}qx;q)_k}
= \frac{(x;q)_\infty}{(q^kx;q)_\infty} \frac{(q^k q^{-k}qx;q)_\infty}{(q^{-k}qx;q)_\infty}
\\
&= \frac{(x;q)_\infty}{(q^kx;q)_\infty} \frac{(qx;q)_\infty}{(q^{-k}qx;q)_\infty}
= \frac{(qx;q)_\infty}{(q^k qx;q)_\infty} 
\frac{(1-x)}{(1-q^kx)}
\frac{(qx;q)_\infty}{(q^{-k}qx;q)_\infty}
\\
&= \frac{(x;q)_\infty}{(q^kx;q)_\infty} \frac{(qx;q)_\infty}{(q^{-k}qx;q)_\infty}
= \frac{(qx;q)_\infty}{(t qx;q)_\infty} 
\frac{(1-x)}{(1-tx)}
\frac{(qx;q)_\infty}{(t^{-1}qx;q)_\infty}.
\end{align*}
Then using $\ev_{q^\lambda t^\rho}(Y_iY^{-1}_j) 
= q^{\lambda_i-\lambda_j} t^{j-i}  = \ev_{q^{-\lambda} t^{-\rho}}(Y^{-1}_iY_j)$,
gives
\begin{align*}
\prod_{i<j} \prod_{r=1}^{k-1} \frac{1-q^{\lambda_i-\lambda_j+r}t^{j-i}}{1-q^{\lambda_i-\lambda_j-r}t^{j-i}}
&= \prod_{i<j} \frac{(qq^{\lambda_i-\lambda_j}t^{j-i};q)_\infty}{(t qq^{\lambda_i-\lambda_j}t^{j-i};q)_\infty} 
\frac{(1-q^{\lambda_i-\lambda_j}t^{j-i})}{(1-tq^{\lambda_i-\lambda_j}t^{j-i})}
\frac{(qq^{\lambda_i-\lambda_j}t^{j-i};q)_\infty}{(t^{-1}qq^{\lambda_i-\lambda_j}t^{j-i};q)_\infty}
\\
&= \ev_{q^\lambda t^\rho}\Big(\frac{1}{\Delta_\infty^Y(t)}\Big)
\ev_{q^\lambda t^\rho}\Big(\frac{1}{\Delta_0^Y(t)}\Big)
\ev_{q^{-\lambda} t^{-\rho}}\Big(\frac{1}{\Delta_\infty^{Y^{-1}}(t^{-1})}\Big).
\end{align*}
Thus the last statement of Proposition~\ref{Normformula} can be written in the form
$$
(P_\lambda(q,t), P_\lambda(q,t) )_{q,t}
= W_0(t)\cdot 
\ev_{q^\lambda t^\rho}\Big(\frac{1}{\Delta_\infty^Y(t)}\Big)
\ev_{q^\lambda t^\rho}\Big(\frac{1}{\Delta_0^Y(t)}\Big)
\ev_{q^{-\lambda} t^{-\rho}}\Big(\frac{1}{\Delta_\infty^{Y^{-1}}(t^{-1})}\Big),
$$
a formula which, since it holds for all $t= q^k$ with $k\in \ZZ_{>0}$, holds for general $t$.
\qed
\end{remark}

\subsection{The symmetric inner product}

Define involutions 
\begin{align*}
&\overline{\phantom{T}}\colon \CC[X] \to \CC[X]
&&\hbox{by}
&\bar{f}(x_1, \ldots, x_n;q,t) &= f(x_1^{-1}, \ldots, x_n^{-1}; q^{-1}, t^{-1}),
\nonumber \\
&{}^\sigma\colon \CC[X] \to \CC[X]
&&\hbox{by}
&f^\sigma(x_1, \ldots, x_n;q,t) &= f(x_1^{-1}, \ldots, x_n^{-1}; q, t),
\\
&{}^t \colon \CC[X] \to \CC[X]
&&\hbox{by}
&f^t(x_1, \ldots, x_n;q,t) &= f(x_1, \ldots, x_n; q^{-1}, t^{-1}).
\nonumber
\end{align*}

Define two scalar products $(\ ,\ )\colon \CC[X]\times \CC[X] \to \CC(q,t)$ and
$\langle\ ,\ \rangle\colon \CC[X]\times \CC[X] \to \CC(q,t)$ by 
\begin{equation*}
( f_1,f_2)_{q,t} 
= \mathrm{ct}\Big(  \frac{f_1 \overline{f_2}} {\Delta_\infty^X \Delta_0^X \Delta_\infty^{X^{-1}}} \Big) 
\qquad\hbox{and}\qquad
\langle f_1,f_2\rangle_{q,t} 
= \frac{1}{\vert W_0\vert}
\mathrm{ct}\Big(\frac{f_1 f_2^\sigma}{\Delta_\infty^X \Delta_0^X \Delta_0^{X^{-1}} \Delta_\infty^{X^{-1}}}\Big).
\end{equation*}
The following result 
provides a comparison of $(,)_{q,t}$ and $\langle,\rangle_{q,t}$ 
as inner products on symmetric polynomials.

\begin{prop}   \label{innprodcompA}
Let $f,g\in \CC[X]^{S_n}$ so that $f$ and $g$ are symmetric polynomials.  
Then
\begin{equation*}
\langle f,g \rangle_{q,t} = \frac{1}{W_0(t) } (f,g^t)_{q,t}.
\end{equation*}
\end{prop}
\begin{proof}  Let $f,g\in \CC[X]^{S_n}$.  
Equation~\eqref{Poinbysymm} says 
$$W_0(t) = \sum_{w\in W_0} w (t^{\frac12\ell(w_0)}c_{w_0}^{X^{-1}}),
\qquad\hbox{and using}\qquad 
\Delta_0^{X^{-1}} = t^{\frac12\ell(w_0)}c_{w_0}^{X^{-1}},$$ 
gives
\begin{align*}
\langle f,g \rangle_{q,t}
&= \frac{1}{\vert W_0\vert}
\mathrm{ct}\Big(\frac{f g^\sigma}{\Delta_\infty^X \Delta_0^X \Delta_0^{X^{-1}} \Delta_\infty^{X^{-1}}}\Big)
= \frac{1}{W_0(t)\vert W_0\vert}\mathrm{ct}\Big(\frac{f \overline{g}^t }
{\Delta_\infty^X \Delta_0^X \Delta_0^{X^{-1}} \Delta_\infty^{X^{-1}}} W_0(t) \Big)
\\
&= \frac{1}{W_0(t)\vert W_0\vert}\mathrm{ct}\Big(\frac{f \overline{g}^t}
{\Delta_\infty^X \Delta_0^X \Delta_0^{X^{-1}} \Delta_\infty^{X^{-1}}} 
\big(\sum_{w\in W_0} w (t^{\frac12\ell(w_0)}c_{w_0}^{X^{-1}})\big) \Big)
\\
&= \frac{1}{W_0(t)\vert W_0\vert}\mathrm{ct}\Big( \sum_{w\in W_0} w \big( \frac{f \overline{g}^t}
{\Delta_\infty^X \Delta_0^X \Delta_0^{X^{-1}} \Delta_\infty^{X^{-1}}} t^{\frac12\ell(w_0)}c_{w_0}^{X^{-1}}\big) \Big)
\\
&= \frac{1}{W_0(t)\vert W_0\vert}\mathrm{ct}\Big( \sum_{w\in W_0} w \big( \frac{f \overline{g}^t}
{\Delta_\infty^X \Delta_0^X \Delta_\infty^{X^{-1}}}\big) \Big)
= \frac{1}{W_0(t)}
\mathrm{ct}\Big(  \frac{f \overline{g}^t} {\Delta_\infty^X \Delta_0^X \Delta_\infty^{X^{-1}}} \Big) 
= \frac{1}{W_0(t)} (f,g^t)_{q,t}.
\end{align*}
\end{proof}

\end{document}